\documentclass[oneside]{amsart}

\usepackage{bm}
\usepackage{graphicx, xcolor}
\usepackage{hyperref}

\usepackage{amsmath, mathtools, amssymb}
\usepackage{amsthm}
\usepackage{mathrsfs}
\usepackage[mathscr]{eucal}
\usepackage{tikz}
\usetikzlibrary{matrix, arrows}
\usepackage[abbrev]{amsrefs}
\usepackage{latexsym}
\usepackage{stmaryrd}

\usepackage[top=25truemm,bottom=20truemm,left=20truemm,right=20truemm]{geometry}
\definecolor{eldritch}{RGB}{205,18,29}

\numberwithin{equation}{subsection}
\newtheorem{thm}{Theorem}[subsection]
\newtheorem{prop}[thm]{Proposition}
\newtheorem{lem}[thm]{Lemma}
\newtheorem{cor}[thm]{Corollary}
\theoremstyle{definition}
\newtheorem{dfn}[thm]{Definition}
\newtheorem{ex}[thm]{Example}
\newtheorem{rem}[thm]{Remark}
\newtheorem{obs}[thm]{Observation}

\newcommand{\Def}[1]{\it{#1}\rm{}}
\newcommand{\Z}{\mathbb{Z}}

\newcommand{\R}{\mathbb{R}}
\newcommand{\Q}{\mathbb{Q}}
\newcommand{\ep}{\varepsilon}
\newcommand{\object}{\mathrm{Obj}}
\newcommand{\Hom}[1]{\mathrm{Hom}_{#1}}
\newcommand{\category}[1]{\mathscr{#1}}
\newcommand{\functor}[1]{\mathcal{#1}}
\newcommand{\opposite}{^{\mathrm{op}}}
\newcommand{\id}[1]{\mathrm{id}_{#1}}
\newcommand{\proj}[1]{\mathrm{pr}_{#1}}
\newcommand{\incl}[1]{\mathrm{in}_{#1}}

\newcommand{\Set}{\mathsf{Set}}
\newcommand{\sSet}{\mathsf{sSet}}

\newcommand{\D}{\bm{\Delta}}

\newcommand{\boundary}[1]{\partial\Delta[#1]}

\newcommand{\simplex}[2]{\Delta^{#1}[#2]}
\newcommand{\topsimplex}[1]{\Delta_{#1}}

\newcommand{\K}{\mathbb{K}}
\newcommand{\Der}[3]{\mathrm{Der}_{#1}(#2,#3)}
\newcommand{\qdR}[4]{\Omega^{#2}_{#3}\langle{#1}\rangle_{#4}}
\newcommand{\qC}[4]{\mathcal{C}^{#2}_{#3}\langle{#1}\rangle_{#4}}
\newcommand{\qG}[4]{\mathcal{G}^{#2}_{#3}\langle{#1}\rangle_{#4}}
\newcommand{\dq}{\vartheta}
\newcommand{\df}[1]{\mathrm{d}^{#1}}

\newcommand{\chain}[2]{{#1}_{#2}}
\newcommand{\cact}[2]{^{#1}_{#2}}
\newcommand{\bs}[1]{\mathsf{b}_{#1}}
\newcommand{\fs}[1]{\mathsf{f}_{#1}}
\newcommand{\Fs}[1]{\tilde{\mathsf{f}}_{#1}}
\newcommand{\us}[1]{\mathsf{u}_{#1}}
\newcommand{\F}[1]{\mathsf{F}_{#1}}
\newcommand{\vs}[3]{\mathsf{v}_{#1}(#2)(#3)}
\newcommand{\Rs}[3]{\mathsf{R}_{#1}(#2)(#3)}

\newcommand{\fint}[1]{{#1}_{\ast}}
\newcommand{\bint}[1]{\partial{#1}_{\ast}}
\newcommand{\Ahol}[1]{\mathsf{hol}_{A_{\infty}}^{#1}}
\newcommand{\shol}[2]{\mathscr{H}\mathrm{ol}^{#1}_{#2}}

\begin{document}

\title{Higher Holonomy via a Simplicial viewpoint}
\author{Ryohei Kageyama}
\address{Mathematical Institute, Tohoku University, Sendai 980-8578, Japan}
\email{ryohei.kageyama.s8@dc.tohoku.ac.jp}
\date{\today}
\maketitle
\begin{abstract}
In this paper, we construct an analogy of holonomy of connection to simplicial sets using $A_{\infty}$-categories. To construct it, we develop fiberwise integrals on simplicial sets and define an iterated integral on simplicial sets. It is an analogy to Chen's iterated integral. We also prove an analogy of Stokes's theorem for fiberwise integrals.
\end{abstract}
\tableofcontents 

\section*{Introduction}
The holonomy representation is a kind of representation of the fundamental groupoid of a manifold. More specifically, for each finite-dimensional $\R$-vector space $V$, a connection values in Lie algebra $\mathfrak{gl}(V)$ on a manifold $\mathcal{M}$ gives a representation, that is a functor, $\mathcal{P}_{1}(\mathcal{M})\to\mathrm{GL}(V)$. It is already considered that a generalization of holonomy called $2$-holonomy. It is a strict $2$-functor from path $2$-category $\mathcal{P}_{2}(M)$ to some strict $2$-category. For example, a strict $2$-functor $\mathcal{P}_{2}(M)\to\mathrm{Aut}(\mathcal{V})$ obtained by using a chain complex (of finite type) $\mathcal{V}$ instead of a vector space $V$, a differential crossed module $\mathfrak{gl}(\mathcal{V})$ instead of Lie algebra $\mathfrak{gl}(V)$, a $\mathfrak{gl}(\mathcal{V})$-valued differential form instead of $\mathfrak{gl}(V)$-valued connection is called $2$-holonomy in some papers (\cite{MR3333093},\cite{MR3636693},\cite{MR4151724,MR3571383}).
It is known that these strict ($2$-)functors can be constructed using Chen's iterated integral. (See, for example, \cite{MR3333093} or \cite{MR4151724,MR3571383}.)

All ``homotopical data'' of topological space $\mathcal{M}$ is contained in the singular (stratified) simplicial set $S(\mathcal{M})$. For example, fundamental groupoid $\pi_{1}(\mathcal{M})$ coincides with the homotopy category of $S(\mathcal{M})$. Therefore we would like to consider simplicial sets instead of smooth manifolds or their fundamental groupoids. It is a motivation to construct an analogy of holonomy to simplicial sets.

To construct an analogy of holonomy of connection with values in an $L_{\infty}$-algebra, we use two tools. One of them is iterated integrals. However, the integration on simplicial sets has not been developed. (Fiberwise) integrals are important tools to research spaces, hence developing a fiberwise integral on a simplicial set is expected to be useful. Therefore, in this paper, we also focus on developing them. We define two kinds of integrals in section \ref{sec3}. One is fiberwise integrals along a projection $X\times U\to X$, and the other is fiberwise integral along a projection $X\times U\to X$ on ``boundaries''. We also prove an analogy of Stokes's theorem \ref{Stokes} in section \ref{sec3}. This states the relation between the two integrals.
Another tool is $A_{\infty}$-categories. We construct an $A_{\infty}$-category $\mathcal{P}(X,\Z)$ from arbitrally simplicial set $X$ and construct an $A_{\infty}$-functor form $\mathcal{P}(X,\Z)$ to an $A_{\infty}$-algebra constructed from an $L_{\infty}$-algebra $\mathfrak{g}$ whose underlying chain complex is connected in section \ref{sec4}. They appear to be linearizations of simplicial sets and holonomies, respectively. 
\begin{center}
\begin{tabular}{|c|c|c|c|}\hline
		&holonomy&$2$-holonomy&in this paper\\\hline
		ground ring&field $\mathbb{R}$&feild $\mathbb{R}$&divided power algebra $\mathbb{Z}\langle\dq\rangle$\\\hline
		&finite dimensional&chain complex $\mathcal{V}$&\\
		&vector space $V$&of finite type&\\\hline
	space&(smooth) manifold $\mathcal{M}$&(smooth) manifold $\mathcal{M}$&simplicial set $X$\\\hline
	domain&path groupoid $\mathcal{P}_{1}(\mathcal{M})$&path $2$-categoy $\mathcal{P}_{2}(\mathcal{M})$&$A_{\infty}$-category $\mathcal{P}(X,\Z)$\\\hline
	Lie algebra&Lie algebra $\mathfrak{gl}(V)$&2-Lie algebra $\mathfrak{gl}(\mathcal{V})$&connected $L_{\infty}$-algebra $\mathfrak{g}$\\\hline
	connection&$\mathfrak{gl}(V)$-valued&$\mathfrak{gl}(\mathcal{V})$-valued&simplicial map\\
		&differential $1$-form&differential form&$X\to\Omega^{1}_{\bullet}\langle\dq\rangle_{\mathfrak{g}}^{\wedge}$\\\hline
	codomain&Lie group& strict $2$-category&dg algebra\\
	&$\mathrm{Aut}(V)$&$\mathrm{Aut}(\mathcal{V})$&$\qG{\dq}{}{}{\mathfrak{g}}$\\\hline
\end{tabular}
\end{center}
\section*{Acknowlegdments}
The author would like to thank Yuji Terashima and Ryo Horiuchi for useful communication.
\section{Brief reveiw of $A_{\infty}$-algebras, $A_{\infty}$-categories and $L_{\infty}$-algebras}\label{sec1}
In this section, we fix a commutative ring $\K$.
\subsection{$A_{\infty}$-algebras and $L_{\infty}$-algebras}
The tensor products of graded $\K$-modules $V_{\bullet}$ and $W_{\bullet}$ is defined by $(V\otimes W)_{n}=\underset{s+t=n}{\bigoplus}V_{p}\otimes W_{q}$ and the tensor product of degree $p$ map $f\colon V_{\bullet}\to V'_{\bullet}$ and degree $q$ map $g\colon W_{\bullet}\to W'_{\bullet}$ is defined by $(f\otimes g)(v\otimes w)\coloneqq(-1)^{|v|\cdot q} f(v)\otimes g(w)$. The \Def{graded-tensor algebra $\mathsf{T}V$ of graded $\K$-module $V_{\bullet}$} is defined by
\begin{align*}
	\mathsf{T}V\coloneqq\bigoplus_{r=0}^{\infty}V^{\otimes r}=\K\oplus V\oplus(V\otimes V)\oplus(V\otimes V\otimes V)\oplus\cdots.
\end{align*}
The graded-tensor algebra is a bialgebra. For instance, the product is defined by
\begin{align*}
	(x_{1}\otimes\dots\otimes x_{p})\cdot(y_{1}\otimes\dots\otimes y_{q})=x_{1}\otimes\dots\otimes x_{p}\otimes y_{1}\otimes\dots\otimes y_{q},
\end{align*}
and the coproduct is defined by
\begin{align*}
	\Delta(x_{1}\otimes\dots\otimes x_{r})=\sum_{p+q=r}(x_{1}\otimes\dots\otimes x_{p})\otimes(x_{p+1}\otimes\dots\otimes x_{r}).
\end{align*}
For any $\mathsf{Sym}V$ of a graded $\K$-module $V_{\bullet}$, a quotient of a graded tensor algebra $\mathrm{T}V$ by an ideal $I$ generated by the following elements is called the graded-symmetric algebra:
\begin{itemize}
	\item $x\otimes y-(-1)^{|x|\cdot|y|}y\otimes x$ for each homogeneous elements $x,y\in V_{\bullet}$.
	\item $x\otimes x$ for each homogeneous element $x\in V_{\bullet}$ whose degree is odd.
\end{itemize}
We denote an element $x\otimes y+I$ of $\mathsf{Sym}V$ by $x\wedge y$.
For each positive integer $n>0$, we define a map $\ep\colon\mathfrak{S}_{n}\times \bigcup_{p}V_{p}\times\dots\times\bigcup_{p}V_{p}\to\{\pm1\}$ using the formula
\begin{align*}
	&x_{1}\wedge\dots\wedge x_{n}=\ep(\sigma,x_{1},\dots,x_{n})x_{\sigma(1)}\wedge\dots\wedge x_{\sigma(n)}
\end{align*}
in $\mathsf{Sym}V$. In addition, we define a coproduct $\Delta\colon\mathsf{Sym}V\to\mathsf{Sym}V\otimes\mathsf{Sym}V$ as follows:
\begin{align*}
	\Delta(x_{1}\wedge\dots\wedge x_{r})
		&\coloneqq
	\Delta(x_{1})\cdots\Delta(x_{r})\\
		&=
	(x_{1}\otimes1+1\otimes x_{1})\cdots(x_{r}\otimes1+1\otimes x_{r})\\
		&=
	\underset{\sigma\in\mathrm{Sh}(p,q)}{\sum_{p+q=r}}\ep(\sigma,x_{1},\dots,x_{r})(x_{\sigma(1)}\wedge\dots\wedge x_{\sigma(p)})\otimes(x_{\sigma(p+1)}\wedge\dots\wedge x_{\sigma(n)})
\end{align*}

For arbitrary graded $\K$-modules $V_{\bullet}$, a \Def{coderivation on a graded-tensor algebra $\mathsf{T}V$} is a (degree preserving) linear map $D\colon\mathsf{T}V\to\mathsf{T}V$ satisfying the Leibniz rule
\begin{align*}
	\Delta D=(D\otimes1+1\otimes D)\Delta
\end{align*}
and other two conditions $\proj{0}D=0$ and $D\incl{0}=0$. Similarly, a \Def{coderivation on a graded-tensor algebra $\mathsf{Sym}V$} is a (degree preserving) linear map $D\colon\mathsf{Sym}V\to\mathsf{Sym}V$ satisfying the Leibniz rule
\begin{align*}
	\Delta D=(D\otimes1+1\otimes D)\Delta
\end{align*}
and other two conditions $\proj{0}D=0$ and $D\incl{0}=0$. The \Def{suspension} of a graded $\K$-module $V_{\bullet}$ is a graded $\K$-module $V[1]$ defined by $V[1]_{n}=V_{n-1}$. In general, the graded $\K$-module $V[p]$ is defined by $V[p]_{n}\coloneqq V_{n-p}$ for each integer $p$.

An \Def{$A_{\infty}$-algebra over $\K$} is a pair of a graded $\K$-module $\mathcal{A}_{\bullet}$ and a degree $-1$ coderivation $D$ on $\mathsf{T}\mathcal{A}[1]$ satisfying $D\circ D=0$. For any $A_{\infty}$-algebras $\mathcal{A},\mathcal{A}'$, an \Def{$A_{\infty}$-map} from $\mathcal{A}$ to $\mathcal{A}'$ is an augmented coalgebra homogeneous $f\colon\mathsf{T}\mathcal{A}[1]\to\mathsf{T}\mathcal{A}'[1]$ satisfying $f\circ D=D\circ f$.
\begin{ex}
	Let $(\mathcal{A}_{\bullet},d,\wedge)$ be a dg algebra. We define two degree $-1$ maps $D_{1},D_{2}\colon\mathsf{T}\mathcal{A}[1]\to\mathsf{T}\mathcal{A}[1]$ as
	\begin{align*}
		D_{1}(x_{1}[1]\wedge\dots\wedge x_{r}[1])
			&=
		\sum_{i=1}^{r}(-1)^{\nu_{i-1}}x_{1}[1]\wedge\dots\wedge dx_{i}[1]\wedge\dots\wedge x_{r}[1]\\
		D_{2}(x_{1}[1]\wedge\dots\wedge x_{r}[1])
			&=
		\sum_{i=1}^{r}(-1)^{\nu_{i}}x_{1}[1]\wedge\dots\wedge(x_{i}\wedge x_{i+1)}[1]\wedge\dots\wedge x_{r}[1]
	\end{align*}
	where $\nu_{i}=|x_{1}|+\dots+|x_{i}|+i$. And we define a degree $-1$ map $D\colon\mathsf{T}\mathcal{A}[1]\to\mathsf{T}\mathcal{A}[1]$ as $D=D_{1}+D_{2}$. Then $(\mathcal{A},D)$ is an $A_{\infty}$-algebra.
\end{ex}
On the other hand, a pair of a graded $\K$-module $\mathfrak{g}_{\bullet}$ and a degree $-1$ coderivation $D$ on $\mathsf{Sym}\mathfrak{g}[1]$ satisfying $D\circ D=0$ is called an \Def{$L_{\infty}$-algebra over $\K$}.
\begin{ex}
	The pair of a (trivial) graded $\K$-module $\K$ and the zero map $0\colon\mathsf{Sym}\K[1]\to\mathsf{Sym}\K[1]$ is an $L_{\infty}$-algebra.
\end{ex}
\begin{ex}
Let $(\mathfrak{g}_{\bullet},\partial,[-,-])$ be a dg Lie algebra over $\K$. In the other words, we consider a pair of a chain complex $\mathfrak{g}_{\bullet}=(\mathfrak{g}_{\bullet},\partial_{\bullet})$ of $\K$-modules and a chain map $[-,-]\colon\mathfrak{g}\otimes\mathfrak{g}\to\mathfrak{g}$ satisfing the following conditions:
\begin{align*}
	[x,y]&=-(-1)^{|x|\cdot|y|}[y,x]&&(\text{skew-symmetric}),\\
	\partial[x,y]&=[\partial x,y]+(-1)^{|x|}[x,\partial y]&&(\text{Leibniz rule}),\\
	[x,[y,z]]&=[[x,y],z]+(-1)^{|x|\cdot|y|}[y,[x,z]]&&(\text{Jacobi identity}).
\end{align*}
We define two degree $-1$ maps $D_{1},D_{2}\colon\mathsf{Sym}\mathfrak{g}[1]\to\mathsf{Sym}\mathfrak{g}[1]$ as
\begin{align*}
	&D_{1}(x_{1}[1]\wedge\dots\wedge x_{r}[1])\\
		\coloneqq
	&\sum_{i=1}^{r}(-1)^{\nu_{i-1}}x_{1}[1]\wedge\dots\wedge\partial x_{i}[1]\wedge\dots\wedge x_{r}[1],\\
	&D_{2}(x_{1}[1]\wedge\dots\wedge x_{r}[1])\\
		\coloneqq
	&\sum_{i<j}(-1)^{(|x_{i}|+1)\nu_{i-1}+(|x_{j}|+1)\nu_{j-1}+(|x_{i}|+1)|x_{j}|}[x_{i},x_{j}][1]\wedge x_{1}[1]\wedge\dots\wedge\check{x_{i}}[1]\wedge\dots\wedge\check{x_{j}}[1]\wedge\dots\wedge x_{r}[1]
\end{align*}
where $\nu_{i}=|x_{1}|+\dots+|x_{i}|+i$.
And we define a degree $-1$ map $D\colon\mathsf{Sym}\mathfrak{g}[1]\to\mathsf{Sym}\mathfrak{g}[1]$ as $D=D_{1}+D_{2}$. Then $(\mathfrak{g},D)$ is an $L_{\infty}$-algebra.
\end{ex}
The \Def{universal enveloping algebra $\mathbb{U}_{\infty}\mathfrak{g}$} of an $L_{\infty}$-algebra $(\mathfrak{g},D)$ is a dg algebra defined as follows:
\begin{itemize}
	\item The underlying graded $\K$-algebra is a graded tensor algebra of the desuspension of the kernel of counit $\proj{0}\colon\mathsf{Sym}\mathfrak{g}[1]\to\K$
		\begin{align*}
			\mathsf{T}\mathrm{Ker}(\mathsf{Sym}\mathfrak{g}[1]\xrightarrow{\proj{0}}\K)[-1]=\bigoplus_{r=0}^{\infty}(\mathrm{Ker}(\mathsf{Sym}\mathfrak{g}[1]\xrightarrow{\proj{0}}\K)[-1])^{\otimes r}
		\end{align*}
	\item The differential $\delta$ of $\mathbb{U}_{\infty}\mathfrak{g}$ is deterimed by 
		\begin{align*}
			\delta(x[-1])=D(x)[-1]-\sum_{i}(-1)^{|x_{i}|}x_{i}[-1]\otimes y_{i}[-1]
		\end{align*}
		if $\Delta(x)-x\otimes1-1\otimes x=\sum_{i}x_{i}\otimes y_{i}$.
\end{itemize}
We denote the completion of $\mathbb{U}_{\infty}\mathfrak{g}$
\begin{align*}
	\hat{\mathsf{T}}\mathrm{Ker}(\mathsf{Sym}\mathfrak{g}[1]\xrightarrow{\proj{0}}\K)[-1]=\prod_{r=0}^{\infty}(\mathrm{Ker}(\mathsf{Sym}\mathfrak{g}[1]\xrightarrow{\proj{0}}\K)[-1])^{\otimes r}
\end{align*}
as $\hat{\mathbb{U}}_{\infty}\mathfrak{g}$.

\subsection{$A_{\infty}$-categories and $A_{\infty}$-nerve}
$A_{\infty}$-categories\cite{MR1270931} are the many object version of $A_{\infty}$-algebras. To define $A_{\infty}$-categories, graded $\K$-quiver and their tensor product are used instead of graded $\K$-modules. A \Def{graded $\K$-quiver} $\category{Q}$ is a pair of the following data:
\begin{itemize}
	\item  a small set of \Def{objects}, denoted $\object\category{Q}$;
	\item for each pair $(x,y)$ of objects of $\category{Q}$, a graded $\K$-module, denoted by $\category{Q}(x,y)$.
\end{itemize}
And for any graded $\K$-quivers $\category{Q}_{1},\category{Q}_{2}$, a morphism of $\functor{F}\colon\category{Q}_{1}\to\category{Q}_{2}$ is a pair of the following data:
\begin{itemize}
	\item  a $\object\functor{F}\colon\object\category{Q}_{1}\to\object\category{Q}_{2}$;
	\item for each pair $(x,y)$ of objects of $\category{Q}_{1}$, a morphism $\functor{F}_{x,y}\colon\category{Q}_{1}(x,y)\to\category{Q}_{2}(x,y)$.
\end{itemize}
We call a graded quiver $\category{Q}$ which $\category{Q}(x,y)$ is a chain complex for each pair $(x,y)$ of objects of $\category{Q}$ \Def{dg $\K$-quiver}.

For any pair $(\category{Q}_{1},\category{Q}_{2})$ of graded $\K$-quivers satisfying $\object\category{Q}_{1}=\object\category{Q}_{2}$, define a tensor product $\category{Q}_{1}\otimes\category{Q}_{2}$ as follows:
\begin{align*}
	\object(\category{Q}_{1}\otimes\category{Q}_{2})&\coloneqq\object\category{Q}_{i}\\
		(\category{Q}_{1}\otimes\category{Q}_{2})(x,y)&\coloneqq\bigoplus_{z\in\object\category{Q}_{i}}\category{Q}_{1}(x,z)\otimes\category{Q}_{2}(z,y)
\end{align*}
In addition, we define a (differential) graded $\K$-quiver $\K Q$ for any small sets $Q$ as follows:
\begin{align*}
	\object(kQ)&\coloneqq Q\\
		(\K Q)(x,y)&\coloneqq
			\begin{cases}
				\K&(x=y)\\
				0&(x\neq y)
			\end{cases}.
\end{align*}
A graded $\K$-quiver $\category{Q}$ gives a graded $\K$-quiver
$$\mathsf{T}\category{Q}=\bigoplus_{r=0}^{\infty}\category{Q}^{\otimes r}=(\K\object\category{Q})\oplus\category{Q}\oplus(\category{Q}\otimes\category{Q})\oplus(\category{Q}\otimes\category{Q}\otimes\category{Q})\oplus\cdots$$
and a cocomposition $\Delta\colon\mathsf{T}\category{Q}\to\mathsf{T}\category{Q}\otimes\mathsf{T}\category{Q}$
$$\Delta(f_{1}\otimes\dots\otimes f_{r})=\sum_{p+q=r}(f_{1}\otimes\dots\otimes f_{p})\otimes(f_{p+1}\otimes\dots\otimes f_{r}).$$
Then we obtain an augmented graded cocategory $(\mathsf{T}\category{Q},\Delta,\proj{0},\incl{0})$.
An \Def{$A_{\infty}$-category} is a pair of a graded $\K$-quiver $\category{A}_{\bullet}$ and a degree -1 codervation $D\colon\mathsf{T}\category{A}[1]\to\mathsf{T}\category{A}[1]$ satisfying $D\circ D=0$. For any $A_{\infty}$-categories $\category{A},\category{A}'$, a \Def{strict $A_{\infty}$-functor} from $\category{A}$ to $\category{A}'$ is an augmented cocategory homogeneous $\functor{F}\colon\mathsf{T}\category{A}[1]\to\mathsf{T}\category{A}'[1]$ satisfying $\functor{F}\circ D=D\circ\functor{F}$.
For any $A_{\infty}$-categories $(\category{A},D)$, the underlying graded quiver $\category{A}$ is a dg $\K$-quiver where the differential is given by follows for each pair $(x,y)$ of objects of $\object\category{A}$:
\begin{align*}
	\category{A}(x,y)=\category{A}(x,y)[1][-1]\xrightarrow{\incl{1}[-1]}(\mathsf{T}\category{A}[1])[-1]\xrightarrow{D[-1]}(\mathsf{T}\category{A}[1])[-1]\xrightarrow{\proj{1}[-1]}\category{A}(x,y)[1][-1]=\category{A}(x,y).
\end{align*}
Thus we obtain a forgetful functor from the category of (small) $A_{\infty}$-categories over $\K$ and strict $A_{\infty}$-functors to the category of dg quivers and their morphisms (that is a morphism of graded quivers which preserve differentials.)
\begin{thm}\label{free A infty category}\rm (free $A_{\infty}$-categories \cite{MR2223034})\it
	The above forgetful functor has a left adjoint.
\end{thm}
An $A_{\infty}$-category $(\category{A},D)$ is \Def{strictly unital} if, for each object $x\in\object\category{A}$, there is an element $\id{x}\in\category{A}(x,x)_{0}$, called a \Def{strict unit}, such that the following conditions are satisfied:
\begin{align*}
	\proj{1}D(f_{1}[1]\otimes\dots f_{p}[1]\otimes\id{x}[1]\otimes f_{p+1}[1]\dots\otimes f_{p+q}[1])
		=
	\begin{cases}
		f_{1}&((p,q)=(1,0),(0,1))\\
		0&(\text{others})
	\end{cases}.
\end{align*}
\begin{prop}\label{unit}
	The forgetful functor from the category of strict unital $A_{\infty}$-categories over $\K$ to the category of $A_{\infty}$-categories has a left adjoint.
\end{prop}
\begin{proof}
	Let ($\category{A},D)$ be an $A_{\infty}$-category. We define a graded quiver $\overline{\category{A}}$ as follows:
	\begin{align*}
		\object\overline{\category{A}}&\coloneqq\object\category{A},\\
		\overline{\category{A}}(x,y)&\coloneqq
			\begin{cases}
				\category{A}(x,x)\oplus\K\cdot\id{x}&(x=y)\\
				\category{A}(x,y)&(x\neq y)
			\end{cases}.
	\end{align*}
	And we define a degree $-1$ auto morphism $\overline{D}\colon\mathsf{T}\overline{\category{A}}[1]\to\mathsf{T}\overline{\category{A}}[1]$ as
	\begin{align*}
		&\overline{D}(f_{1}[1]\otimes\dots\otimes f_{n}[1])\\
			\coloneqq
		&\sum_{p+r+q=n}(-1)^{|f_{1}|+\dots+|f_{p}|+p}f_{1}[1]\otimes\dots f_{p}[1]\otimes\proj{1}\overline{D}(f_{p+1}[1]\otimes\dots\otimes f_{p+r}[1])\otimes f_{p+r+1}[1]\otimes\dots\otimes f_{n}[1]
	\end{align*}
	where $\proj{1}\overline{D}$ is given as follows for any composable pair $f_{1},\dots,f_{p_{1}+\dots+p_{r}}$ of arrows of $\category{A}$:
	\begin{align*}
		&\proj{1}\overline{D}(f_{1}[1]\otimes\dots\otimes f_{p_{0}}[1]\otimes\id{x_{1}}[1]\otimes\dots\otimes\id{x_{r}}[1]\otimes f_{p_{1}+\dots+p_{r-1}+1}[1]\otimes\dots\otimes f_{p_{1}+\dots p_{r-1}+p_{r}}[1])\\
			=
		&\begin{cases}
			(-1)^{|f_{1}|+1}f_{1}[1]&((p_{0},p_{1})=(1,0)\text{ and }r=1)\\
			-f_{1}[1]&((p_{0},p_{1})=(0,1)\text{ and }r=1)\\
			-\id{x_{0}}[1]&((p_{0},p_{1},p_{2})=(0,0,0)\text{ and }r=2)\\
			0&(\text{others})
		\end{cases}.
	\end{align*}
	For any composable pair $f_{1},\dots,f_{n}$ of arrows of $\overline{\category{A}}$, the following hold:
	\begin{itemize}
		\item If there is no integer $i=1,\dots,n$ which satifies $f_{i}=\id{}$,
			\begin{align*}
				&\sum_{p+r+q=n}(\pm)\proj{1}\overline{D}(f_{1}[1]\otimes\dots\otimes f_{p}[1]\otimes\proj{1}\overline{D}(f_{p+1}[1]\otimes\dots\otimes f_{p+r}[1])\otimes f_{p+r+1}[1]\otimes\dots\otimes f_{n}[1])\\
					=
				&\sum_{p+r+q=n}(\pm)\proj{1}D(f_{1}[1]\otimes\dots\otimes f_{p}[1]\otimes\proj{1}D(f_{p+1}[1]\otimes\dots\otimes f_{p+r}[1])\otimes f_{p+r+1}[1]\otimes\dots\otimes f_{n}[1])\\
					=
				&0
			\end{align*}
			holds.
		\item $f_{1}=\id{}$ implies the following:
			\begin{align*}
				&\sum_{p+r+q=n}(\pm)\proj{1}\overline{D}(f_{1}[1]\otimes\dots\otimes f_{p}[1]\otimes\proj{1}\overline{D}(f_{p+1}[1]\otimes\dots\otimes f_{p+r}[1])\otimes f_{p+r+1}[1]\otimes\dots\otimes f_{n}[1])\\
					=
				&\proj{1}\overline{D}(\proj{1}\overline{D}(\id{}[1]\otimes f_{2}[1])\otimes f_{3}[1]\otimes\dots\otimes f_{n}[1])
				-\proj{1}\overline{D}(\id{}[1]\otimes\proj{1}D(f_{2}[1]\otimes\dots\otimes f_{n}[1]))\\
					=
				&0.
			\end{align*}
		\item $f_{n}=\id{}$ implies the following:
			\begin{align*}
				&\sum_{p+r+q=n}(\pm)\proj{1}\overline{D}(f_{1}[1]\otimes\dots\otimes f_{p}[1]\otimes\proj{1}\overline{D}(f_{p+1}[1]\otimes\dots\otimes f_{p+r}[1])\otimes f_{p+r+1}[1]\otimes\dots\otimes f_{n}[1])\\
					=
				&\proj{1}\overline{D}(\proj{1}\overline{D}(f_{1}[1]\otimes\dots\otimes f_{n-1}[1])\otimes\id{}[1])\\
				&+(-1)^{|f_{1}|+\dots+|f_{n-2}|+n-2}\proj{1}\overline{D}(f_{1}[1]\otimes\dots\otimes f_{n-2}[1]\otimes\proj{1}D(f_{n-1}[1]\otimes\id{}[1]))\\
					=
				&0.
			\end{align*}
		\item If there is an integer $1<i=1<n$ which satifies $f_{i}=\id{}$,
			\begin{align*}
				&\sum_{p+r+q=n}(\pm)\proj{1}\overline{D}(f_{1}[1]\otimes\dots\otimes f_{p}[1]\otimes\proj{1}\overline{D}(f_{p+1}[1]\otimes\dots\otimes f_{p+r}[1])\otimes f_{p+r+1}[1]\otimes\dots\otimes f_{n}[1])\\
					=
				&(-1)^{|f_{1}|+\dots+|f_{i-2}|+i-2}\proj{1}\overline{D}(f_{1}[1]\otimes\dots\otimes f_{i-2}[1]\otimes\proj{1}D(f_{i-1}[1]\otimes\id{}[1])\otimes f_{i+1}[1]\otimes\dots\otimes f_{n}[1])\\
				&+(-1)^{|f_{1}|+\dots+|f_{i-1}|+i-1}\proj{1}\overline{D}(f_{1}[1]\otimes\dots\otimes f_{i-1}[1]\otimes\proj{1}D(\id{}[1]\otimes f_{i+1}[1])\otimes f_{i+2}[1]\otimes\dots\otimes f_{n}[1])\\
					=
				&0
			\end{align*}
			holds.
	\end{itemize}
	In other words, we obtain a strict unital $A_{\infty}$-category $\overline{\category{A}}$.
	And then, for arbitrary strict $A_{\infty}$-functor $\varphi\colon\category{A}\to\category{B}$, we defined a map $\overline{\varphi}\colon\mathsf{T}\overline{A}[1]\to\mathsf{T}\overline{\category{B}}[1]$ as follows for any composable pair $f_{1},\dots,f_{p_{1}+\dots+p_{r}}$ of arrows of $\category{A}$:
	\begin{align*}
		&\overline{\varphi}(f_{1}[1]\otimes\dots\otimes f_{p_{0}}[1]\otimes\id{x_{1}}[1]\otimes\dots\otimes\id{x_{r}}[1]\otimes f_{p_{1}+\dots+p_{r-1}+1}[1]\otimes\dots\otimes f_{p_{1}+\dots p_{r-1}+p_{r}}[1])\\
			=
		&\varphi(f_{1}[1]\otimes\dots\otimes f_{p_{0}}[1])\otimes\id{f(x_{1})}[1]\otimes\dots\otimes\id{f(x_{r})}[1]\otimes\varphi(f_{p_{1}+\dots+p_{r-1}+1}[1]\otimes\dots\otimes f_{p_{1}+\dots p_{r-1}+p_{r}}[1]).
	\end{align*}
	We prove $\overline{\varphi}$ is an $A_{\infty}$-functor by induction on the number of arrows. First, for any arrow $f$ of $\category{A}$, the following hold:
	\begin{align*}
		\overline{\varphi}\overline{D}(f[1])
			&=
		\begin{cases}
			\varphi D(f[1])&(f\neq\id{})\\
			\overline{\varphi}(0)&(f=\id{})
		\end{cases}\\
			&=
		\begin{cases}
			D\varphi(f[1])&(f\neq\id{})\\
			\overline{D}(\id{}[1])&(f=\id{})
		\end{cases}\\
			&=
		\overline{D}\overline{\varphi}(f[1]).
	\end{align*}
	Suppose that the following holds for any pair of $n$ composable arrows $f_{1},\dots,f_{n}$ of $\overline{\category{A}}$:
	\begin{align*}
		\overline{\varphi}\overline{D}(f_{1}[1]\otimes\dots\otimes f_{n}[1])
			=
		\overline{D}\overline{\varphi}(f_{1}[1]\otimes\dots\otimes f_{n}[1]).
	\end{align*}
	Let $(f_{1},\dots,f_{m},f_{m+1},\dots,f_{n})$ be a pair of composable arrows of $\overline{A}$. Then
	\begin{align*}
		&\overline{\varphi}\overline{D}(f_{1}[1]\otimes\dots\otimes f_{m}[1]\otimes\id{}[1]\otimes f_{m+1}[1]\otimes\dots\otimes f_{n}[1])\\
			=
		&\sum_{p+r+q=m}(\pm)\overline{\varphi}(f_{1}[1]\otimes\dots\otimes\overline{D}(f_{p+1}[1]\otimes\dots\otimes f_{p+r}[1])\otimes\dots\otimes\id{}[1]\otimes\dots\otimes f_{n}[1])\\
		&+(-1)^{|f_{1}|+\dots+|f_{m-1}|+m-1}\overline{\varphi}(f_{1}[1]\otimes\dots\otimes f_{m-1}[1]\otimes\overline{D}(f_{m}[1]\otimes\id{}[1])\otimes f_{m+1}[1]\otimes\dots\otimes f_{n}[1])\\
		&+(-1)^{|f_{1}|+\dots+|f_{m}|+m}\overline{\varphi}(f_{1}[1]\otimes\dots\otimes f_{m}[1]\otimes\overline{D}(\id{}[1]\otimes f_{m+1}[1])\otimes f_{m+2}[1]\otimes\dots\otimes f_{n}[1])\\
		&+\sum_{p+r+q=n-m}(\pm)\overline{\varphi}(f_{1}[1]\otimes\dots\otimes\id{}[1]\otimes\dots\otimes\overline{D}(f_{m+P+1}[1]\otimes\dots\otimes f_{m+p+r+1}[1])\otimes\dots\otimes f_{n}[1])\\
			=
		&\overline{D}\overline{\varphi}(f_{1}[1]\otimes\dots\otimes f_{m}[1])\otimes\id{}[1]\otimes\overline{\varphi}(f_{m+1}[1]\otimes\dots\otimes f_{n}[1])\\
		&+(-1)^{|f_{1}|+\dots+|f_{m}|+m+1}\overline{\varphi}(f_{1}[1]\otimes\dots f_{m}[1])\otimes\id{}[1]\otimes\overline{D}\overline{\varphi}(f_{m+1}[1]\otimes\dots\otimes f_{n}[1])\\
			=
		&\overline{D}\overline{\varphi}(f_{1}[1]\otimes\dots\otimes f_{m}[1]\otimes\id{}[1]\otimes f_{m+1}[1]\otimes\dots\otimes f_{n}[1])
	\end{align*}
	holds.
	
	Let $\category{A}$ be an $A_{\infty}$-category and $\category{B}$ be a strictly unital $A_{\infty}$-category. Then any $A_{\infty}$-functor $\category{A}\to\category{B}$ gives a (unit preserving) strict $A_{\infty}$-functor $\overline{\category{A}}\to\category{B}$ in the same way as above, and induces a natural bijection
	$$\Hom{\mathsf{u}A_{\infty}\mathsf{Cat}_{\K}}(\overline{\category{A}},\category{B})\cong\Hom{A_{\infty}\mathsf{Cat}_{\K}}(\category{A},\category{B}).$$
\end{proof}
For each non-negative integer $n\geq0$, a (strictly unital) $A_{\infty}$-category $\category{A}_{\infty}^{n}$ is defined as follows:
\begin{itemize}
	\item $\object\category{A}_{\infty}^{n}=\{0,\dots,n\}$.
	\item For $0\leq i,j\leq n$
		\begin{align*}
			\category{A}_{\infty}^{n}(i,j)=
				\begin{cases}
					\K\cdot(i,j)&(i\leq j)\\
					0&(\text{others})
				\end{cases}.
		\end{align*}
	\item A degree $-1$ coderivation $D\colon\mathsf{T}(\category{A}_{\infty}^{n})[1]\to\mathsf{T}(\category{A}_{\infty}^{n})[1]$ given by 
		\begin{align*}
			D((i_{0},i_{1})[1]\otimes\dots\otimes(i_{n-1},i_{n})[1])
				\coloneqq
			\sum_{p=1}^{n-1}(-1)^{p}(i_{0},i_{1})[1]\otimes\dots\otimes(i_{p-1},i_{p+1})[1]\otimes\dots\otimes(i_{n-1},i_{n})[1]
		\end{align*}
\end{itemize}
In addition, we define a strict $A_{\infty}$-functor $\alpha_{\ast}\colon\category{A}_{\infty}^{m}\to\category{A}_{\infty}^{n}$ for each order-preserving map $\alpha\colon[m]\to[n]$:
\begin{itemize}
	\item $\alpha_{\ast}(i)\coloneqq\alpha(i)$ for each objects $i=0,\dots,m$.
	\item For each elements $(i_{0},i_{1})[1]\otimes\dots\otimes(i_{-1},i_{r})[1]$
		\begin{align*}
			\alpha_{\ast}((i_{0},i_{1})[1]\otimes\dots\otimes(i_{-1},i_{r})[1])
				\coloneqq
			(\alpha(i_{0}),\alpha(i_{1}))[1]\otimes\dots\otimes(\alpha(i_{-1}),\alpha(i_{r}))[1].
		\end{align*}
\end{itemize}
Then we obtain a cosimplicial $A_{\infty}$-category $\category{A}_{\infty}^{\bullet}$. It gives a functor $\mathscr{N}_{A_{\infty}}\coloneqq\Hom{\mathsf{u}A_{\infty}\mathsf{Cat}_{\K}}(\category{A}_{\infty}^{\bullet},-)$ from the category of (small) $A_{\infty}$-categories (with the strict unit) over $\K$ and strict $A_{\infty}$-functors preserving strict unit $\mathsf{u}A_{\infty}\mathsf{Cat}_{\K}$ to the category of simplicial sets $\sSet$, called simplicial nerve of $A_{\infty}$-categories \cite{MR3607208}.

\section{Calculation on Standard simplices}
\subsection{Divided Power de Rham Complexes}
An order-preserving map $\alpha\colon[m]\to[n]$ gives an affine map $\alpha_{\ast}\mathbb{A}_{\R}^{m+1}\to\mathbb{A}_{\R}^{n+1}$
\begin{align*}
	(x_{0},\dots,x_{m})\mapsto(\sum_{\alpha(j)=0}x_{j},\dots,\sum_{\alpha(j)=n}x_{j}),
\end{align*}
and an affine map $\bm{V}(\underset{0\leq i\leq m}{\sum}x_{i}-1)\to\bm{V}(\underset{0\leq i\leq n}{\sum}x_{i}-1)$ between hyperplanes. It induces a map between subspaces
\begin{center}
\begin{tikzpicture}[auto]
	\node (11) at (0, 1.2) {$\triangle^{m}$};
	\node (12) at (4, 1.2) {$\triangle^{n}$};
	\node (21) at (0, 0) {$\bm{V}(\underset{0\leq i\leq m}{\sum}x_{i}-1)$};
	\node (22) at (4, 0) {$\bm{V}(\underset{0\leq i\leq n}{\sum}x_{i}-1)$};
	\path[draw, right hook->] (11) -- (21);
	\path[draw, right hook->] (12) -- (22);
	\path[draw, ->, dashed] (11) --node {$\scriptstyle \alpha_{\ast}$} (12);
	\path[draw, ->] (21) --node {$\scriptstyle \alpha_{\ast}$} (22);
\end{tikzpicture}
\end{center}
which defined as $\triangle^{n}\coloneqq\{(x_{0},\dots,x_{n})\in\bm{V}(\sum_{i}x_{i}-1)|x_{i}\in[0,1]\}$ for each $n\geq0$. For each $n\geq0$, there is an isomorphim $\mathbb{A}_{\R}^{n}\cong\bm{V}(\sum_{i}x_{i}-1)$ defined as follows:
\begin{align*}
	\mathbb{A}_{\R}^{n}\to\bm{V}(\sum_{i}x^{i}-1),\ &\ (t^{1},\dots,t^{n})\mapsto(1-t^{1},t^{1}-t^{2},\dots,t^{n-1}-t^{n},t^{n}-0)\\
	\bm{V}(\sum_{i}x^{i}-1)\to\mathbb{A}_{\R}^{n},\ &\ (x^{0},\dots,x^{n})\mapsto(\sum_{i=1}^{n}x^{i},\dots,\sum_{i=n}^{n}x^{i}).
\end{align*}
The image of $\triangle^{n}$ under the isomorphism is given by
\begin{align*}
	\topsimplex{n}\coloneqq\{(t_{1},\dots,t_{n})|1\geq t_{1}\geq\dots\geq t_{n}\geq0\}.
\end{align*}
For each order-preserving map $\alpha\colon[m]\to[n]$, we obtain a commutative diagram
\begin{center}
\begin{tikzpicture}[auto]
	\node (11) at (0, 1.2) {$\mathbb{A}_{\R}^{m}$};
	\node (21) at (0, 0) {$\mathbb{A}_{\R}^{n}$};
	\node (12) at (2.5, 1.2) {$\bm{V}(\sum_{j}X^{j}-1)$};
	\node (22) at (2.5, 0) {$\bm{V}(\sum_{j}X^{j}-1)$};
	\node (13) at (5, 1.2) {$\mathbb{A}_{\R}^{m+1}$};
	\node (23) at (5, 0) {$\mathbb{A}_{\R}^{n+1}$};
	\path[draw, ->] (11) --node {$\scriptstyle \simeq$} (12);
	\path[draw, right hook->] (12) -- (13);
	\path[draw, ->] (21) --node {$\scriptstyle \simeq$} (22);
	\path[draw, right hook->] (22) -- (23);
	\path[draw, ->] (11) --node[swap] {$\scriptstyle \alpha_{\ast}$} (21);
	\path[draw, ->] (12) --node {$\scriptstyle \alpha_{\ast}$} (22);
	\path[draw, ->] (13) --node {$\scriptstyle \alpha_{\ast}$} (23);
\end{tikzpicture}.
\end{center}
Where $\alpha_{\ast}\colon\mathbb{A}_{\R}^{m}\to\mathbb{A}_{\R}^{n}$ is defined as follows:
\begin{align*}
	\proj{i}\alpha_{\ast}(t_{1},\dots,t_{n})\coloneqq
		\begin{cases}
			t_{\min\{j\in[m]|\alpha(j)\geqslant i\}}&(\alpha(m)\geqslant i)\\
			0&(\alpha(m)<i)
		\end{cases}.
\end{align*}
An affine space $\mathbb{A}_{\Q}^{n}$ corresponds to a polynomial ring $\Q[t_{1},\dots,t_{n}]$ and a hyperplane $\bm{V}(\sum_{i}x_{i}-1)\subset\mathbb{A}_{\Q}^{n+1}$ corresponds to a quotient ring $\Q[x_{0},\dots,x_{n}]/(\sum_{i}x_{i}-1)$.  In addition, the isomorphism $\mathbb{A}_{\Q}^{n}\colon\bm{V}(\sum_{i}x_{i}-1)$ corresponds to a ring isomorphism
\begin{align*}
	\mathbb{Q}[t_{1},\dots,t_{n}]\cong\mathbb{Q}[x_{0},\dots,x_{n}]/(\sum_{i}x_{i}-1).
\end{align*}
The quotient ring $\mathbb{Q}[x_{0},\dots,x_{n}]/(\sum_{i}x_{i}-1)$ just coincide with a ring whose elements are (Sullivan's) differential $0$-form on an  $n$-dimensional standard simplex $\simplex{}{n}$. Therefore it is not unnatural to regard the polynomial ping $\mathbb{Q}[t_{1},\dots,t_{n}]$ as a ring of functions on an $n$-dimensional standard simplex $\simplex{}{n}$.

However, the de Rham complex (which corresponds to this ring) has trivial torsion (as Abelian group). In addition, we must assume the character of the ring we are considering is $0$. Therefore we use a ring that does not contain $\Q$. The most extreme candidate is $\Z$, in which case ``integration'' cannot be defined. So we consider a divided power polynomial algebra over $\Z$, that is a free Abelian group
\begin{align*}
	\Z\langle x_{0},\dots,x_{n}\rangle\coloneqq\bigoplus_{N_{0},\dots,N_{n}\geq0}\Z x_{0}^{[N_{0}]}\dots x_{n}^{[N_{n}]}
\end{align*}
with product defined as
\begin{align*}
	(x_{0}^{[N_{10}]}\dots x_{n}^{[N_{1n}]})(x_{0}^{[N_{20}]}\dots x_{n}^{[N_{2n}]})=\frac{(N_{10}+N_{20})}{N_{10}!N_{20}!}\dots\frac{(N_{1n}+N_{2n})}{N_{1n}!N_{2n}!}x_{0}^{[N_{10}+N_{20}]}\dots x_{n}^{[N_{1n}+N_{2n}]}
\end{align*}
to be the ring of ``functions on an $n$-dimensional standard simplex $\simplex{}{n}$''. We denote $x_{i}^{[1]}$ as $x_{i}$. This ring can be embedded in the polynomial ring $\Q[x_{0},\dots,x_{n}]$ by the canonical way which is given by following morphism :
\begin{align*}
	x_{0}^{[N_{0}]}x_{1}^{[N_{1}]}\dots x_{n}^{[N_{n}]}\mapsto\frac{1}{N_{0}!N_{1}!\dots N_{n}!}x_{0}^{N_{0}}x_{1}^{N_{1}}\dots x_{n}^{N_{n}}.
\end{align*}
Similarly, three kinds of (canonical) morphisms
	\begin{align*}
		\Z\langle x_{0},\dots,x_{n}\rangle&\to\Q\langle x_{1},\dots,x_{n}\rangle\\
		\Z\langle x_{0},\dots,x_{n}\rangle&\to\Z_{(p)}\langle x_{1},\dots,x_{n}\rangle\\
		\Z\langle x_{0},\dots,x_{n}\rangle&\to\Z\langle x_{1},\dots,x_{n}\rangle
	\end{align*}
	are given as follows where $p$ is a prime number:
	\begin{align*}
		x_{0}^{[N_{0}]}x_{1}^{[N_{1}]}\dots x_{n}^{[N_{n}]}
			&\mapsto
		\frac{1}{N_{0}!}x_{1}^{[N_{1}]}\dots x_{n}^{[N_{n}]},\\
		x_{0}^{[N_{0}]}x_{1}^{[N_{1}]}\dots x_{n}^{[N_{n}]}
			&\mapsto
		\frac{1}{N_{0}!}p^{N_{0}}x_{1}^{[N_{1}]}\dots x_{n}^{[N_{n}]},\\
		x_{0}^{[N_{0}]}x_{1}^{[N_{1}]}\dots x_{n}^{[N_{n}]}
			&\mapsto
		\begin{cases}
			x_{1}^{[N_{1}]}\dots x_{n}^{[N_{n}]}&(N_{0}=0)\\
			0&(N_{0}\neq0)
		\end{cases}.
	\end{align*}
	More generally, a divided power polynomial algebra has a universal property like polynomial rings. Therefore, for each map $\ep\colon\{x_{0},x_{1},\dots,x_{n}\}\to\{x_{0},x_{1},\dots,x_{n}\}$, there exists a unique morphism $\overline{\ep}\colon\Z\langle x_{0},\dots,x_{n}\rangle\to\Z\langle x_{0},\dots,x_{n}\rangle$ satisfies $\overline{\ep}(x_{i}^{[N_{i}]})=\ep(x_{i})^{[N_{i}]}$ for each $i=0,\dots,n$.

We define a morphism $\alpha^{\ast}\colon\Z\langle x_{0},\dots,x_{n}\rangle\to\Z\langle x_{0},\dots,x_{m}\rangle$ as
\begin{align*}
	\alpha^{\ast}(x_{i}^{[N]})\coloneqq
		\begin{cases}
			x_{\min\{j|\alpha(j)\geq i\}}^{[N]}&(\alpha(m)\geq i)\\
			0&(\alpha(m)<i)
		\end{cases}
\end{align*}
for each order-preserving maps $\alpha\colon[m]\to[n]$. We obtain a simplicial $\Z\langle x_{0}\rangle$-algebra $\qdR{x_{0}}{0}{\bullet}{}$ by above.
Hereafter we denote $x_{0}$ of these rings as $\dq$, and consider $\dq$ to be an element like the unit of the ring.

For each non-negative integer $n\geq0$ and arbitrary $\qdR{\dq}{0}{n}{}$-modules $M$, an (Abelian) group morphism of $\theta\colon\qdR{\dq}{0}{n}{}\to M$ which satisfies the following is called a divided power $\Z\langle\dq\rangle$-derivation:
\begin{align*}
	\theta(a)&=0&&\text{for all }a\in\Z\langle\dq\rangle,\\
	\theta(fg)&=g\theta(f)+f\theta(g)&&\text{for all }f,g\in\qdR{\dq}{0}{n}{},\\
	\theta(x_{i}^{[N]})&=x_{i}^{[N-1]}\theta(x_{i})&&\text{for all }i=1,\dots,n\text{ and }N\geq1.
\end{align*}
Denote the $\qdR{\dq}{0}{}{}$-module of divided power $\Z\langle\dq\rangle$-derivations of $\qdR{\dq}{0}{n}{}$ into $M$ by $\Der{\Z\langle\dq\rangle}{\qdR{\dq}{0}{n}{}}{M}$. It gives a representable functor $\Der{\Z\langle\dq\rangle}{\qdR{\dq}{0}{n}{}}{-}\colon\mathrm{Mod}_{\qdR{\dq}{0}{n}{}}\to\mathrm{Mod}_{\qdR{\dq}{0}{n}{}}$. It is represented by a free $\Z\langle\dq,x_{1},\dots,x_{n}\rangle$-module $\qdR{\dq}{1}{n}{}$ generated by formal elements $\df{}x_{1},\dots,\df{}x_{n}$. In addition the derivation $\df{0}\colon\qdR{\dq}{0}{n}{}\to\qdR{\dq}{1}{n}{}$ corresponding to the identity $\id{}\colon\qdR{\dq}{1}{n}{}\to\qdR{\dq}{1}{n}{}$ is given as follows:
\begin{align*}
	\df{0}(\sum_{N_{1},\dots,N_{n}}f_{N_{1},\dots,N_{n}}x_{1}^{[N_{1}]}\dots x_{n}^{[N_{n}]})
		\coloneqq
	\sum_{i=1}^{n}(\sum_{N_{1},\dots,N_{n}}f_{N_{1},\dots,N_{n}}x_{1}^{[N_{1}]}\dots x_{i}^{[N_{i}-1]}\dots x_{n}^{[N_{n}]})\df{}x_{i}
\end{align*}
We denote the derivation $\qdR{\dq}{0}{n}{}\to\qdR{\dq}{0}{n}{}$ corresponding to the ``standard dual base'' $\chi_{\df{}x_{i}}\colon\qdR{\dq}{1}{n}{}\to\qdR{\dq}{0}{n}{}$
\begin{align*}
	\chi_{\df{}x_{i}}(\sum_{j}f_{j}\df{}x_{j})\coloneqq f_{i}
\end{align*}
by $\frac{\partial}{\partial x_{i}}$.

They give a graded (commutative) $\qdR{\dq}{0}{n}{}$-algebra
\begin{align*}
	\qdR{\dq}{\bullet}{n}{}\coloneqq\mathsf{Sym}\qdR{\dq}{1}{n}{}[1]=\qdR{\dq}{0}{n}{}\oplus\qdR{\dq}{1}{n}{}\oplus\qdR{\dq}{2}{n}{}\oplus\dots\oplus\qdR{\dq}{n}{n}{}
\end{align*}
and a degree $-1$ derivation $\df{}\colon\qdR{\dq}{\bullet}{n}{}\to\qdR{\dq}{\bullet}{n}{}$. In other words, we obtain a dg (commutative) algebra $\qdR{\dq}{}{n}{}$. We call the dg algebra the \Def{divided power de Rham complex on standard simplex $\simplex{}{n}$}. For each order-preserving map $\alpha\colon[m]\to[n]$, the $\Z\langle\dq\rangle$-algebra morphism $\alpha^{\ast}$ gives a dg algebra morphism $\alpha^{\ast}\colon\qdR{\dq}{}{n}{}\to\qdR{\dq}{}{m}{}$. Therefore we obtain a simplicial dg (commutative) algebra $\qdR{\dq}{}{}{}\colon\D\opposite\to\mathsf{dgA}_{\Z\langle\dq\rangle}$ and a simplicial dg (commutative) algebra $\qdR{\dq}{}{}{\K}\coloneqq\qdR{\dq}{}{}{}\otimes_{\Z}\K$.

\subsection{Formal Differential Forms}
Let $(\mathfrak{g}_{\bullet},D)$ be a connected $L_{\infty}$-algebra (over $\Z$), that is an $L_{\infty}$-algebra whose underlying chain complex is connected.
Then, for each non-negative integer $n\geq0$, we obtain graded $\Z$-module
\begin{align*}
	\qdR{\dq}{\bullet}{n}{\mathfrak{g}}^{\wedge}
		&\coloneqq
	\prod_{p+\bullet=q}\mathfrak{g}_{p}\otimes\qdR{\dq}{q}{n}{}.
\end{align*}
We call a degree $1$ element $\omega\in\qdR{\dq}{1}{n}{\mathfrak{g}}^{\wedge}$ of above graded $\Z$-module a \Def{generalized connection with values in $\mathfrak{g}$ on the standard simplex $\simplex{}{n}$}. 

Roughly speaking, connections with values in an arbitrary Lie algebra $\mathfrak{g}$ are analogous to differential $1$-forms. So we want to define a concept that can be said to be analogous to differential forms. For this purpose, using the universal enveloping (dg) algebra of $L_{\infty}$-algebra. Using this (dg) algebra, we obtain a dg algebra
\begin{align*}
	\qdR{\dq}{\bullet}{n}{\mathbb{U}_{\infty}\mathfrak{g}}^{\wedge}\coloneqq\prod_{p+\bullet=q}\mathbb{U}_{\infty}\mathfrak{g}_{p}\otimes\qdR{\dq}{q}{n}{}
\end{align*}
for each non-negative integer $n\geq0$ where the differential is defined as
\begin{align*}
	\df{}(\sum g\otimes\omega)\coloneqq g\otimes\df{}\omega,
\end{align*}
and obtain a simplicial dg algebra $\qdR{\dq}{}{\bullet}{\mathbb{U}_{\infty}\mathfrak{g}}^{\wedge}$.

\subsection{Integration on Standard Simplices}
To define the integration of formal differential forms, we observe the classical case, in other words, the integration of a polynomial function of real coefficients. For any integer $a\in\R$ and non-negative ingeter $N$, the following (redundant) equation holds:
\begin{align*}
	\int_{\alpha}^{\beta}a\frac{x^{N}}{N!}dx=a\frac{\beta^{N+1}}{(N+1)!}-a\frac{\alpha^{N+1}}{(N+1)!}
\end{align*}
\begin{dfn}(iterated integral of divided power polynomial functions)
	Let $f=\sum_{N_{1},\dots,N_{r}}m_{N_{1},\dots,N_{r}}x_{1}^{[N_{1}]}\cdots x_{r}^{[N_{r}]}$ be an $r$-variable divided power polynomial of integer coefficients, that is an element of $\qdR{\dq}{0}{r}{}=\Z\langle\dq,x_{1},\dots,x_{n}\rangle$. Then we define the \Def{iterated integration} of $f$
	\begin{align*}
		&\int_{\alpha_{p}}^{\beta_{p}}\dots\int_{\alpha_{1}}^{\beta_{1}}f\df{}x_{i_{1}}\cdots\df{}x_{i_{p}}
		&(\alpha_{1},\dots,\alpha_{p},\beta_{1},\dots,\beta_{p}\in\{\dq,x_{1},\dots,x_{r},0\})
	\end{align*}
	inductively as follows:
	\begin{align*}
		\int_{\alpha_{1}}^{\beta_{1}}f\df{}x_{i_{1}}
			&\coloneqq
		\sum_{N_{1},\dots,N_{r}}m_{N_{1},\dots,N_{r}}x_{1}^{[N_{1}]}\cdots(\beta_{1}^{[N_{i_{1}}+1]}-\alpha_{1}^{[N_{i_{1}}+1]})\cdots x_{r}^{[N_{r}]},\\
		\int_{\alpha_{p}}^{\beta_{p}}\dots\int_{\alpha_{1}}^{\beta_{1}}f\df{}x_{i_{1}}\cdots\df{}x_{i_{p}}
			&\coloneqq
		\int_{\alpha_{p}}^{\beta_{p}}(\int_{\alpha_{p-1}}^{\beta_{p-1}}\dots\int_{\alpha_{1}}^{\beta_{1}}f\df{}x_{i_{1}}\cdots\df{}x_{i_{p-1}})\df{}x_{i_{p}}
	\end{align*}
\end{dfn}
\begin{lem}\label{Int Lem11}
	For any elements $X,Y\in\{\dq,x_{1},\dots,x_{r},0\}$ and any divided power polynomial $f\in\qdR{\dq}{0}{r}{}$,
	\begin{align*}
		\int_{X}^{Y}(\frac{\partial}{\partial x_{i}}f)\df{}x_{i}=\overline{\ep_{i,Y}}(f)-\overline{\ep_{i,X}}(f)
	\end{align*}
	holds where the map $\ep_{i,X}\colon\{\dq,x_{1},\dots,x_{n}\}\to\{\dq,x_{1},\dots,x_{n}\}$ is given as follows:
	\begin{align*}
		\ep_{i,X}(\dq)
			&=
		\dq,\\
		\ep_{i,X}(x_{j})
			&=
		\begin{cases}
			X&(j=i)\\
			x_{j}&(j\neq i)
		\end{cases}.
	\end{align*}
\end{lem}
\begin{proof}
	We can assume that $f=x_{1}^{[N_{1}]}\cdots x_{r}^{[N_{r}]}$. Then
	\begin{align*}
		\int_{X}^{Y}(\frac{\partial}{\partial x_{i}}f)\df{}x_{i}
			&=
		\int_{X}^{Y}(\chi_{\df{}x_{i}}\df{}(x_{1}^{[N_{1}]}\cdots x_{r}^{[N_{r}]}))\df{}x_{i}\\
			&=
		\int_{X}^{Y}(\sum_{j=1}^{r}x_{1}^{[N_{1}]}\cdots x_{j}^{[N_{j}-1]}\cdots x_{r}^{[N_{r}]}\chi_{\df{}x_{i}}(\df{}x_{j}))\df{}x_{i}\\
			&=
		\int_{X}^{Y}(x_{1}^{[N_{1}]}\cdots x_{i}^{[N_{i}-1]}\cdots x_{r}^{[N_{r}]})\df{}x_{i}\\
			&=
		x_{1}^{[N_{1}]}\cdots Y^{[N_{i}-1]}\cdots x_{r}^{[N_{r}]}
			-x_{1}^{[N_{1}]}\cdots X^{[N_{i}-1]}\cdots x_{r}^{[N_{r}]}\\
			&=
		\overline{\ep_{i,Y}}(f)-\overline{\ep_{i,X}}(f)
	\end{align*}
	holds.
\end{proof}
\begin{cor}\label{Int Lem12}
	For any elements $X,Y\in\{\dq,x_{1},\dots,x_{r},0\}$ and any divided power polynomial $f\in\qdR{\dq}{0}{r}{}$,
	\begin{align*}
		\int_{X}^{Y}f(\frac{\partial}{\partial x_{i}}g)\df{}x_{i}
			=
		(\overline{\ep_{i,Y}}(fg)-\overline{\ep_{i,X}}(fg))
			-\int_{X}^{Y}(\frac{\partial}{\partial x_{i}}f)g\df{}x_{i}
	\end{align*}
	holds where the map $\ep_{i,X}\colon\{\dq,x_{1},\dots,x_{n}\}\to\{\dq,x_{1},\dots,x_{n}\}$ is given as follows:
	\begin{align*}
		\ep_{i,X}(\dq)
			&=
		\dq,\\
		\ep_{i,X}(x_{j})
			&=
		\begin{cases}
			X&(j=i)\\
			x_{j}&(j\neq i)
		\end{cases}.
	\end{align*}
\end{cor}
\begin{proof}
	The lemma \ref{Int Lem11} implies the following:
	\begin{align*}
		(\overline{\ep_{i,Y}}(fg)-\overline{\ep_{i,X}}(fg))
			=
		\int_{X}^{Y}(\frac{\partial}{\partial x_{i}}(fg))\df{}x_{i}
			=
		\int_{X}^{Y}(\frac{\partial}{\partial x_{i}}f)g\df{}x_{i}
			+\int_{X}^{Y}f(\frac{\partial}{\partial x_{i}}g)\df{}x_{i}
	\end{align*}
\end{proof}
\begin{lem}\label{Int Lem13}
	For any pair of variables $X,Y\in\{x_{1},\dots,x_{r}\}$ which satisfies $X\neq Y$ and any divided power polynomial $f\in\qdR{\dq}{0}{r}{}$ which does not contain $X$ as a variable, the following holds:
	\begin{align*}
		\frac{\partial}{\partial X}\int_{X}^{Y}f\df{}x_{i}=-\overline{\ep_{i,X}}(f)
	\end{align*}
	holds where the map $\ep_{i,X}\colon\{\dq,x_{1},\dots,x_{n}\}\to\{\dq,x_{1},\dots,x_{n}\}$ is given as follows:
	\begin{align*}
		\ep_{i,X}(\dq)
			&=
		\dq,\\
		\ep_{i,X}(x_{j})
			&=
		\begin{cases}
			X&(j=i)\\
			x_{j}&(j\neq i)
		\end{cases}.
	\end{align*}
\end{lem}
\begin{proof}
	We can assume that $X=x_{j}$ and $f=x_{1}^{[N_{1}]}\cdots x_{i}^{[N_{i}]}\cdots x_{j-1}^{[N_{j-1}]}x_{j+1}^{[N_{j+1}]}\cdots x_{r}^{[N_{r}]}$. Then
	\begin{align*}
		\frac{\partial}{\partial X}\int_{X}^{Y}f\df{}x_{i}
			&=
		\frac{\partial}{\partial X}
		(x_{1}^{[N_{1}]}\cdots Y^{[N_{i}+1]}\cdots x_{j-1}^{[N_{j-1}]}x_{j+1}^{[N_{j+1}]}\cdots x_{r}^{[N_{r}]}
			-
		x_{1}^{[N_{1}]}\cdots X^{[N_{i}+1]}\cdots x_{j-1}^{[N_{j-1}]}x_{j+1}^{[N_{j+1}]}\cdots x_{r}^{[N_{r}]})\\
			&=
		-\overline{\ep_{i,X}}(f)
	\end{align*}
	holds.
\end{proof}
\section{Calculation on Simplicial Sets}\label{sec3}
\subsection{Lemmas for Glueing}
\begin{obs}(glueing)
Let $\category{C}$ be a complete category, $\mathrm{U}\colon\category{C}\to\Set$ be a functor that preserves all limits and $M\colon\D\opposite\to\category{C}$ be a simplicial object. Then the functor $M\opposite\colon\D\to\category{C}\opposite$ gives two functors $\overline{M\opposite}\colon\sSet\to\category{C}\opposite$, $\mathrm{U}\opposite\overline{M\opposite}\colon\sSet\to\Set\opposite$ by left Kan extension along the Yoneda embedding. And then the following holds for each simplicial set $X$:
\begin{align*}
	\Hom{\sSet}(X,\mathrm{U}M)
		\cong\varprojlim_{\simplex{}{n}\to X}\mathrm{U}M_{n}
		\cong\mathrm{U}\opposite\overline{M\opposite}(X)
\end{align*}
Therefore we can regard as follows:
\begin{itemize}
	\item $M$ is elementary pieces or ``model''.
	\item A simplicial map $X\to\mathrm{U}M$ is an element of ``$\overline{M\opposite}(X)$''.
\end{itemize}
\end{obs}
This observation suggests the following definitions.
\begin{dfn}
	A \Def{generalized connection with values in a connected $L_{\infty}$-algebra $\mathfrak{g}$ on a simplicial set $X$} is a simplicial map $X\to\qdR{\dq}{1}{}{\mathfrak{g}}^{\wedge}$, and the composition $\kappa(\omega)=\kappa\circ\omega$ is called the \Def{curvature of $\omega$}.
\end{dfn}
\begin{dfn}(formal differential form values in an $L_{\infty}$-algebra)
	A \Def{formal differential form values in a connected $L_{\infty}$-algebra $\mathfrak{g}$ on a simplicial set $X$} or simply \Def{$\mathfrak{g}$-valued formal differential form on $X$} is a simplicial map $X\to\qdR{\dq}{}{}{\mathbb{U}_{\infty}\mathfrak{g}}^{\wedge}$. Especially, for each $p$, we call a simplicial map $X\to\qdR{\dq}{p}{}{\mathbb{U}_{\infty}\mathfrak{g}}^{\wedge}$ a \Def{formal differential $p$-form} values in $\mathfrak{g}$ on $X$.
\end{dfn}
Clearly, any generalized connection $X\to\qdR{\dq}{1}{\bullet}{\mathfrak{g}}^{\wedge}$ with values in $L_{\infty}$-algebra $\mathfrak{g}$ on $X$ gives a formal differential $1$-form by a composition 
$X\to\qdR{\dq}{1}{\bullet}{\mathfrak{g}}^{\wedge}\to\qdR{\dq}{1}{\bullet}{\mathbb{U}_{\infty}\mathfrak{g}}^{\wedge}$.

Since $\qdR{\dq}{}{}{\mathbb{U}_{\infty}\mathfrak{g}}^{\wedge}$ has a wedge product defined as
\begin{align*}
	(\sum_{\alpha}v_{1\alpha}\otimes\omega_{1\alpha})\wedge(\sum_{\beta}v_{2\beta}\otimes\omega_{2\beta})\coloneqq\sum_{\alpha,\beta}(-1)^{|\omega_{1\alpha}|\cdot|v_{2\beta}|}(v_{1\alpha}\otimes v_{2\alpha})\otimes(\omega_{1\alpha}\wedge\omega_{2\beta}),
\end{align*}
the $\K$-module $\qdR{\dq}{}{}{}(X,\mathfrak{g})\coloneqq\Hom{\sSet}(X,\qdR{\dq}{}{}{\mathbb{U}_{\infty}\mathfrak{g}}^{\wedge})$ has a canonical wedge product
\begin{align*}
	X\xrightarrow{\text{diagonal}}X\times X
		\xrightarrow{\omega_{1}\times\omega_{2}}
	\qdR{\dq}{}{}{\mathbb{U}_{\infty}\mathfrak{g}}^{\wedge}\times\qdR{\dq}{}{}{\mathbb{U}_{\infty}\mathfrak{g}}^{\wedge}
		\xrightarrow{\wedge}
	\qdR{\dq}{}{}{\mathbb{U}_{\infty}\mathfrak{g}}^{\wedge}.
\end{align*}
In addition, for any formal differential form $\omega\colon X\to\qdR{\dq}{}{}{\mathbb{U}_{\infty}\mathfrak{g}}^{\wedge}$, its derivation $\df{}\omega$ is defined as the composition
\begin{center}
\begin{tikzpicture}[auto]
	\node (11) at (0, 1) {$X$};
	\node (12) at (1.5, 1) {$\qdR{\dq}{}{}{\mathbb{U}_{\infty}\mathfrak{g}}^{\wedge}$};
	\node (13) at (5, 1) {$\qdR{\dq}{}{}{\mathbb{U}_{\infty}\mathfrak{g}}^{\wedge}$};
	\node (22) at (1.5, 0) {$\mathbb{U}_{\infty}\mathfrak{g}_{p}\otimes\qdR{\dq}{q}{}{}$};
	\node (23) at (5, 0) {$\mathbb{U}_{\infty}\mathfrak{g}_{p}\otimes\qdR{\dq}{q+1}{}{}$};
	\path[draw, ->] (11) --node {$\scriptstyle \omega$} (12);
	\path[draw, ->] (12) --node {$\scriptstyle \df{}$} (13);
	\path[draw, ->] (12) --node[swap] {$\scriptstyle \proj{p,q}$} (22);
	\path[draw, ->] (13) --node {$\scriptstyle \proj{p,q+1}$} (23);
	\path[draw, ->] (22) --node {$\scriptstyle \id{}\otimes\df{}$} (23);
\end{tikzpicture}.
\end{center}
They give a dg algebra $\qdR{\dq}{}{}{}(X,\mathfrak{g})$. The \Def{pullback of formal differential form $\omega\colon Y\to\qdR{\dq}{}{}{\mathbb{U}_{\infty}\mathfrak{g}}^{\wedge}$ by a simplicial map $f\colon X\to Y$} is also defined as a composition $\omega\circ f$ and denoted by $f^{\ast}\omega$. It is obvious that any simplicial map $f\colon X\to Y$ gives a morphism of dg algebra $f^{\ast}\colon\qdR{\dq}{}{}{}(Y,\mathfrak{g})\to\qdR{\dq}{}{}{}(X,\mathfrak{g})$ in this way.

In this paper, we define a fiberwise integration along projection $X\times U\to U$ for arbitrary simplicial sets $X,U$. As we can see from the above observation, products and projections $[n]\times[r]\to[n]$ of the (non-empty) finite total ordered sets are important and we need some propositions about them. They are elementary. However, they are so important to this paper that they are reviewed.

For each non-negative integer $n,r\geq0$, we obtain the (categorical) product of the total ordered sets $[n],[r]$ by defining the order as follows:
\begin{align*}
	(i_{1},j_{1})\leqslant(i_{2},j_{2})\text{ iff }i_{1}\leqslant i_{2}, j_{1}\leqslant j_{2}.
\end{align*}
We call an injective order-preserving map $\Gamma\colon[p]\hookrightarrow[n]\times[r]$ \Def{chain}, Since an order-preserving map $\Gamma\colon[p]\to[n]\times[r]$ is injective only if $p\leq n+r$, we call a chain $\Gamma\colon[n+r]\hookrightarrow[n]\times[r]$ \Def{maximal}.
\begin{prop}\label{GC Lem1}
	Let $\Gamma\colon[n+r]\to[n]\times[r]$ be an order-preserving map. If $\Gamma$ is injective, $\proj{1}\Gamma\colon[n+r]\to[n]$ is surjective.
\end{prop}
\begin{proof}
	The set $[n+r]$ can be partitioned into $[n+r]=\bigcup_{i}(\proj{1}\Gamma)^{-1}(i)$. For this partition, the map $\proj{2}\Gamma$ is injective on each subset $(\proj{1}\Gamma)^{-1}(i)$ for each $i=0,\dots,n$. Since $\Gamma$ is an order-preserving and $[n+r]$ is a totally ordered set, for each pair $(l_{i},l_{j})\in(\proj{1}\Gamma)^{-1}(i)\times(\proj{1}^{-1}\Gamma)^{-1}(j)$, $i<j$ implies $l_{i}<l_{j}$ and thus $\proj{2}\Gamma(l_{i})\leq\proj{2}\Gamma(l_{2})$. If $(\proj{1}\Gamma)^{-1}(i)$ and $(\proj{1}\Gamma)^{-1}(j)$ are non-empty sets,
	\begin{align*}
		\bigcap_{h=i,j}\proj{2}\Gamma((\proj{1}\Gamma)^{-1}(h)\setminus\{\min(\proj{1}\Gamma)^{-1}(h)\})
			=
		\emptyset
	\end{align*}
	holds by injectivity of $(\proj{2}\Gamma)|_{(\proj{1}\Gamma)^{-1}(i)}$ and $(\proj{2}\Gamma)|_{(\proj{1}\Gamma)^{-1}(i)}$. Hence
	\begin{align*}
		n+r+1
			&=
		\sum_{i\in[n]}|(\proj{1}\Gamma)^{-1}(i)|\\
			&=
		\sum_{(\proj{1}\Gamma)^{-1}(i)\neq\emptyset}|(\proj{1}\Gamma)^{-1}(i)|\\
			&=
		\sum_{(\proj{1}\Gamma)^{-1}(i)\neq\emptyset}|((\proj{1}\Gamma)^{-1}(i)\setminus\{\min(\proj{1}\Gamma)^{-1}(i)\})\cup\{\min(\proj{1}\Gamma)^{-1}(i)\}|\\
			&=
		\sum_{(\proj{1}\Gamma)^{-1}(i)\neq\emptyset}\left|(\proj{1}\Gamma)^{-1}(i)\setminus\{\min(\proj{1}\Gamma)^{-1}(i)\}\right|+\sum_{(\proj{1}\Gamma)^{-1}(i)\neq\emptyset}|\{\min(\proj{1}\Gamma)^{-1}(i)\}|\\
			&=
		\left|\bigcup_{(\proj{1}\Gamma)^{-1}(i)\neq\emptyset}(\proj{1}\Gamma)^{-1}(i)\setminus\{\min(\proj{1}\Gamma)^{-1}(i)\}\right|+\sum_{(\proj{1}\Gamma)^{-1}(i)\neq\emptyset}|\{\min(\proj{1}\Gamma)^{-1}(i)\}|\\
			&=
		\left|\bigcup_{(\proj{1}\Gamma)^{-1}(i)\neq\emptyset}\proj{2}\Gamma((\proj{1}\Gamma)^{-1}(i)\setminus\{\min(\proj{1}\Gamma)^{-1}(i)\})\right|+\sum_{(\proj{1}\Gamma)^{-1}(i)\neq\emptyset}|\{\min(\proj{1}\Gamma)^{-1}(i)\}|\\
			&\leq
		\left|[r]\setminus\{0\}\right|+|\{i\in[n]|(\proj{1}\Gamma)^{-1}(i)\neq\emptyset\}|\\
			&=
		r+|\{i\in[n]|(\proj{1}\Gamma)^{-1}(i)\neq\emptyset\}|
	\end{align*}
	and thus $n+1\leq|\{i\in[n]|(\proj{1}\Gamma)^{-1}(i)\neq\emptyset\}|\leq|[n]|\leq n+1$ holds.
\end{proof}
\begin{cor}\label{GC Lem2}
	For any maximal chain $\Gamma\colon[n+r]\hookrightarrow[n]\times[r]$, $\bs{\Gamma}$ and $\fs{\Gamma}$ are injective.
\end{cor}
Focusing on this property, as a generalization of maximal chain, we call a chain $\Gamma\colon[p]\hookrightarrow[n]\times[r]$ which induces a surjection $\proj{1}\Gamma\colon[p]\to[n]$ \Def{global} chain. For any global chains $\Gamma\colon[p]\hookrightarrow[n]\times[r]$, we denote the map $\proj{1}\Gamma\colon[p]\to[n]$ (resp. $\proj{2}\Gamma\colon[p]\to[r]$) as $\chain{\bs{}}{\Gamma}$ (resp. $\chain{\fs{}}{\Gamma}$).  In addition, a global chain $\Gamma\colon[p]\hookrightarrow[n]\times[r]$ define a two order-preserving maps $\bs{\Gamma}\colon[n]\to[p],\fs{\Gamma}\colon[r]\to[p]$ as follows:
\begin{align*}
	\bs{\Gamma}(i)&\coloneqq\min\{j\in[p]|\chain{\Gamma}{\bs{}}(j)=i\},\\
	\fs{\Gamma}(i)&\coloneqq\min\{j\in[p]|\chain{\Gamma}{\fs{}}(j)\geq i\}.
\end{align*}
It is easy to show that $\bs{\Gamma}$ and $\fs{\Gamma}$ are injective for any maximal chains $\Gamma\colon[n+r]\hookrightarrow[n]\times[r]$. Especially $\bs{\Gamma}$ is injective for any global chains. Thus we obtain an isomorphism
$$\Fs{\Gamma}\colon\{1,\dots,p-n\}\xrightarrow{\simeq}[p]\setminus\mathrm{Im}\bs{\Gamma}$$
for any global chain $\Gamma\colon[p]\hookrightarrow[n+r]$. It is trivial that $\fs{\Gamma}|_{\{1,\dots,r\}}=\Fs{\Gamma}$ holds for any maximal chain $\Gamma\colon[n+r]\hookrightarrow[n]\times[r]$. The order-preserving map define an order-preserving map $\us{\Gamma}\colon\{1,\dots,p-n\}\to[p]$ as
$$\us{\Gamma}(i)\coloneqq\Fs{\Gamma}(\min\{j\in\{1,\dots,p\}|\Fs{\Gamma}(j)-j=\Fs{\Gamma}(i)-i\})-1.$$
There exists a unique pair of a pisitive integer $\mathsf{n}_{\Gamma}$, a surjective order-preserving map $\F{\Gamma}\colon\{1,\dots,p-n-1\}\to\{1,\dots,\mathsf{n}_{\Gamma}\}$ and an injective order-preserving map $\vs{\Gamma}{-}{0}\colon\{1,\dots,\mathsf{n}_{\Gamma}\}\to[p]$ which satisfies $\us{\Gamma}=\vs{\Gamma}{-}{0}\circ\F{\Gamma}$.
\begin{center}
\begin{tikzpicture}[auto]
	\node (0) at (0, 1) {$\{1,\dots,p-n-1\}$};
	\node (1) at (1.5, 0) {$\{1,\dots,\mathsf{n}_{\Gamma}\}$};
	\node (2) at (3, 1) {$[p]$};
	\path[draw, ->] (0) --node {$\scriptstyle \us{\Gamma}$} (2);
	\path[draw, ->>] (0) --node[swap] {$\scriptstyle \F{\Gamma}$} (1);
	\path[draw, right hook->] (1) --node[swap] {$\scriptstyle \vs{\Gamma}{-}{0}$} (2);
\end{tikzpicture}
\end{center}
And then we obtain the following subsets:
\begin{align*}
	[r]\cact{P}{j}
		&=
	\{\vs{P}{j}{0},\dots,\vs{P}{j}{r_{j}}\}\\
		&\coloneqq
	\{\us{P}(i)|i\in\F{P}^{-1}(j)\}\cup\{\fs{P}(i)|i\in\F{P}^{-1}(j)\},\\
	[r]\cact{P}{}
		&=
	\bigcup_{j=1}^{\mathsf{n}_{\Gamma}}[r]\cact{P}{j}\\
		&=
	\{\us{P}(i)|i=1,\dots,r\}\cup\{\fs{P}(i)|i=1,\dots,r\}.
\end{align*}
For any global chain $\Gamma\colon[p]\hookrightarrow[n]\times[r]$, the subset $[r]\cact{\Gamma}{}$ can be partitioned into three subsets
\begin{align*}
	\mathsf{Inn}_{\fs{}}(\Gamma)
		\coloneqq
	&\{v\in[r]\cact{\Gamma}{}|v>0,\ \Gamma(v+1)=(\chain{\Gamma}{\bs{}}(v-1)+1,\chain{\Gamma}{\fs{}}(v-1)+1)=(\chain{\Gamma}{\bs{}}(v)+1,\chain{\Gamma}{\fs{}}(v))\},\\
	\mathsf{Inn}_{\bs{}}(\Gamma)
		\coloneqq
	&\{v\in[r]\cact{\Gamma}{}|v>0,\ \Gamma(v+1)=(\chain{\Gamma}{\bs{}}(v-1)+1,\chain{\Gamma}{\fs{}}(v-1)+1)=(\chain{\Gamma}{\bs{}}(v),\chain{\Gamma}{\fs{}}(v)+1)\},\\
	\mathsf{Out}(\Gamma)
		\coloneqq
	&\{v\in[r]\cact{\Gamma}{}|v\not\in\mathsf{Inn}_{\fs{}}(\Gamma)\cup\mathsf{Inn}_{\bs{}}(\Gamma)\}\\
		=&
	\{v\in[r]\cact{\Gamma}{}|v\geqslant0,\ \Gamma(v+1)\neq(\chain{\Gamma}{\bs{}}(v-1)+1,\chain{\Gamma}{\fs{}}(v-1)+1)\}.
\end{align*}
For any maximal chain $\Gamma\colon[n+r]\hookrightarrow[n]\times[r]$ and any vartex $v\in\mathsf{Inn}_{\fs{}}(\Gamma)\cup\mathsf{Inn}_{\bs{}}(\Gamma)$, we obtain a (unique) maximal chain $\Gamma'\colon[n+r]\hookrightarrow[n]\times[r]$ which satisfies $\Gamma\neq\Gamma'$ and $\Gamma\delta_{v}=\Gamma'\delta_{v}$ as follows:
\begin{align*}
	\Gamma'(i)
		\coloneqq
	\begin{cases}
		\Gamma(i)&(i\neq v)\\
		(\chain{\Gamma}{\bs{}}(i-1)+1,\chain{\Gamma}{\fs{}}(i-1))&(i=v\in\mathsf{Inn}_{\fs{}}(\Gamma))\\
		(\chain{\Gamma}{\bs{}}(i-1),\chain{\Gamma}{\fs{}}(i-1)+1)&(i=v\in\mathsf{Inn}_{\bs{}}(\Gamma))
	\end{cases}.
\end{align*}
By considering the above for any maximal chains, we obtain a limit cone
\begin{center}
\begin{tikzpicture}[auto]
	\node (21) at (0, 0.5) {$[n+r-1]$};
	\node (12) at (2.5, 0) {$[n+r1]$};
	\node (32) at (2.5, 1) {$[n+r1]$};
	\node (23) at (5, 0.5) {$[n]\times[r]$};
	\path[draw, right hook->] (21) -- (12);
	\path[draw, right hook->] (21) -- (32);
	\path[draw, right hook->] (12) --node[swap] {$\scriptstyle \Gamma_{2}$} (23);
	\path[draw, right hook->] (32) --node {$\scriptstyle \Gamma_{1}$} (23);
\end{tikzpicture}.
\end{center}
In other words, there exists a partition
\begin{align}\label{chain partition}
	[n]\times[r]\cong\bigcup_{\Gamma\colon[n+r]\hookrightarrow[n]\times[r]}[n+r].
\end{align}
\begin{rem}(geometrical meanings)
	The geometric realization of the nerve of a poset $[n]\times[r]$ is just the product of topological standard simplices $\topsimplex{n}\times\topsimplex{r}$. The above partition \ref{chain partition} is a canonical wap to partition of the space into topological standard simplices. The intersection of a fiber $\proj{\topsimplex{n}}^{-1}(\bm{x})\subset\topsimplex{n}\times\topsimplex{r}$ of projection $\proj{\topsimplex{n}}\colon\topsimplex{n}\times\topsimplex{r}\to\topsimplex{n}$ and the image of each embedding $\Gamma_{\ast}\colon\topsimplex{n+r}\to\topsimplex{n}\times\topsimplex{r}$ is given as follows:
	\begin{align*}
		\proj{\topsimplex{n}}^{-1}(\bm{x})\cap\mathrm{Im}\Gamma_{\ast}
			&\cong
		(\chain{\Gamma}{\bs{}})_{\ast}^{-1}(\bm{x})\\
			&=
		\{(t_{1},\dots,t_{n+r})\in\topsimplex{n+r}|t_{\min\{j|\chain{\Gamma}{\bs{}}(j)\geq i\}}=x_{i}\}\\
			&=
		\{(t_{1},\dots,t_{n+r})\in\topsimplex{n+r}|t_{\bs{\Gamma}(i)}=x_{i}\}.
	\end{align*}
	For each maximal chain $\Gamma\colon[n+r]\hookrightarrow[n]\times[r]$,
	$\mathrm{Im}\bs{\Gamma}\cap\mathrm{Im}\fs{\Gamma}=\{0\}$, $\mathrm{Im}\bs{\Gamma}\cup\mathrm{Im}\fs{\Gamma}=[n+r]$ hold. Hence we can regard
	\begin{itemize}
		\item $\bs{\Gamma}$ represents the ``base direction''.
		\item $\fs{\Gamma}$ represents the ``fiber direction''.
	\end{itemize}
	(The standard coordinate of $\topsimplex{n+r}$ can be split into two kinds of ``direction'', ``base direction'' and ``fiber direction''.) In addition the following holds:
	$$\proj{\topsimplex{n}}^{-1}(\bm{x})\cap\mathrm{Im}\Gamma_{\ast}\cong(\chain{\Gamma}{\bs{}})^{-1}(\bm{x})\cong\topsimplex{r_{1}}\times\dots\times\topsimplex{r_{\mathsf{n}_{\Gamma}}}.
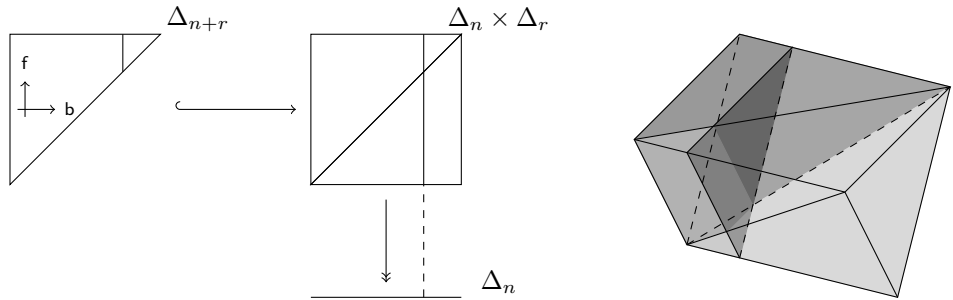
$$
	\begin{center}
	\begin{tikzpicture}[auto]
		\node (1) at (2.5, 3.7) {$\topsimplex{n+r}$};
		\node (2) at (6.5, 3.7) {$\topsimplex{n}\times\topsimplex{r}$};
		\node (3) at (6.5, 0.2) {$\topsimplex{n}$};
		\path[draw] (0, 1.5) -- (0, 3.5) -- (2, 3.5) -- cycle;
		\path[draw] (4, 1.5) -- (4, 3.5) -- (6, 3.5) -- cycle;
		\path[draw] (4, 1.5) -- (6, 1.5) -- (6, 3.5) -- cycle;
		\path[draw] (4, 0) -- (6, 0);
		\path[draw] (1.5, 3) -- (1.5, 3.5);
		\path[draw] (5.5, 1.5) -- (5.5, 3.5);
		\path[draw, dashed] (5.5, 0) -- (5.5, 1.5);
		\path[draw, right hook->] (2+0.2, 2.5) -- (4-0.2, 2.5);
		\path[draw, ->>] (5, 1.5-0.2) -- (5, 0+0.2);
		\node (b) at (0.8, 2.5) {$\scriptstyle \bs{}$};
		\node (f) at (0.2, 3.1) {$\scriptstyle \fs{}$};
		\path[draw, ->] (0.1, 2.5) -- (b);
		\path[draw, ->] (0.2, 2.4) -- (f);
		\draw[black!35, fill=black!35] (9, 0.7) -- (8.3, 2.1) -- (9.7, 3.5) -- (12.5, 2.8) --cycle;
		\draw[black!15, fill=black!15] (9, 0.7) -- (12.5, 2.8) -- (11.8, 0) --cycle;
		\draw[black!40, fill=black!40] (8.3, 2.1) -- (9.35, 2.275) -- (10.4, 3.325) -- (9.7, 3.5) --cycle;
		\draw[black!30, fill=black!30] (9, 0.7) -- (8.3, 2.1) -- (0.5*0.7+9, 3.25*0.7) -- (1.25*0.7+9, 1.75*0.7) -- (0.75*0.7+9, 1.25*0.7) -- cycle;
		\draw[black!20, fill=black!20] (9, 0.7) -- (0.75*0.7+9, 1.25*0.7) -- (1.25*0.7+9, 1.75*0.7) -- (9.7, 0.75*0.7) -- cycle;
		\draw[black!60, fill=black!60] (10.25*0.7-9*0.7+9, 1.75*0.7) -- (9.5*0.7-9*0.7+9, 3.25*0.7) -- (11*0.7-9*0.7+9, 4.75*0.7) --cycle;
		\draw[black!40, fill=black!40] (9.7, 0.75*0.7) -- (0.75*0.7+9, 1.25*0.7) -- (1.25*0.7+9, 1.75*0.7) --cycle;
		\draw[black!50, fill=black!50] (0.75*0.7+9, 1.25*0.7) -- (9, 2.75*0.7) -- (0.5*0.7+9, 3.25*0.7) -- (1.25*0.7+9, 1.75*0.7) --cycle;
		\path[draw, dashed] (9, 1*0.7) -- (9.7, 5*0.7);
		\path[draw] (11.8, 0*0.7) -- (11.1, 2*0.7);
		\path[draw] (8.3, 3*0.7) -- (11.1, 2*0.7);
		\path[draw, dashed] (9, 1*0.7) -- (12.5, 4*0.7);
		\path[draw] (9.7, 0.75*0.7) -- (9, 2.75*0.7) -- (10.4, 4.75*0.7);
		\path[draw, dashed] (9.7, 0.75*0.7) -- (10.4, 4.75*0.7);
		\path[draw] (9, 1*0.7) -- (8.3, 3*0.7) -- (9.7, 5*0.7) -- (12.5, 4*0.7) -- (11.8, 0*0.7) --cycle;
		\path[draw] (8.3, 3*0.7) -- (12.5, 4*0.7);
		\path[draw] (9, 1*0.7) -- (11.1, 2*0.7) -- (12.5, 4*0.7);
	\end{tikzpicture}
	\end{center}
\end{rem}
\begin{prop}\label{GC Lem3}
	Let $\Gamma\colon[n+r]\hookrightarrow[n]\times[r]$ be a maximal chain. Then $\chain{\Gamma}{\bf{}}(i)+\chain{\Gamma}{\fs{}}(i)=i$ holds for each $i\in[n+r]$.
\end{prop}
\begin{proof}
	For each $0\leq i\leq j\leq n+r$,
	$$0\leq\chain{\Gamma}{\bs{}}(i)+\chain{\Gamma}{\fs{}}(i)<\chain{\Gamma}{\bs{}}(j)+\chain{\Gamma}{\fs{}}(j)\leq n+r$$
	holds since $\Gamma$ is injective.
\end{proof}
\begin{prop}\label{GC Lem4}
	Let $\Gamma\colon[n+r]\hookrightarrow[n]\times[r]$ be a maximal chain. Then the following hold for each $j,l=1,\dots,m$ and each $i=0,\dots,r_{j}$:
	\begin{align*}
		\min\{h|\bs{\Gamma}(h)\geq\vs{\Gamma}{l}{r_{l}}+1\}
			=
		\begin{cases}
			\min\{h|\bs{\Gamma\delta_{\vs{\Gamma}{j}{i}}}(h)\geq\vs{\Gamma}{l}{r_{l}}+1\}&(l<j)\\
			\min\{h|\bs{\Gamma\delta_{\vs{\Gamma}{j}{i}}}(h)\geq\vs{\Gamma}{l}{r_{l}}\}&(l\geq j)
		\end{cases}
	\end{align*}
\end{prop}
\begin{proof}
	Denote $\min\{h|\bs{\Gamma}(h)\geq\vs{\Gamma}{l}{r_{l}}+1\}$ as $h$. Then $\bs{\Gamma}(h)=\vs{\Gamma}{l}{r_{l}}+1$ holds. Hence $l<j$ implies
	\begin{align*}
		\chain{(\Gamma\delta_{\vs{\Gamma}{j}{i}})}{\bs{}}(\vs{\Gamma}{l}{r_{l}}+1)
			&=
		\chain{\Gamma}{\bs{}}(\vs{\Gamma}{l}{r_{l}}+1)\\
			&=
		\chain{\Gamma}{\bs{}}\bs{\Gamma}(h),\\
			&=
		h\\
		\chain{(\Gamma\delta_{\vs{\Gamma}{j}{i}})}{\bs{}}(\vs{\Gamma}{l}{r_{l}})
			&=
		\chain{\Gamma}{\bs{}}(\vs{\Gamma}{l}{r_{l}})\\
			&<
		h
	\end{align*}
	On the other hand, $l\geq j$ implies
	\begin{align*}
		\chain{(\Gamma\delta_{\vs{\Gamma}{j}{i}})}{\bs{}}(\vs{\Gamma}{l}{r_{l}})
			&=
		\chain{\Gamma}{\bs{}}(\vs{\Gamma}{l}{r_{l}}+1)\\
			&=
		\chain{\Gamma}{\bs{}}\bs{\Gamma}(h),\\
			&=
		h\\
		\chain{(\Gamma\delta_{\vs{\Gamma}{j}{i}})}{\bs{}}(\vs{\Gamma}{l}{r_{l}}-1)
			&=
		\begin{cases}
			\chain{\Gamma}{\bs{}}(\vs{\Gamma}{l}{r_{l}})&(\vs{\Gamma}{j}{i}<\vs{\Gamma}{l}{r_{l}})\\
			\chain{\Gamma}{\bs{}}(\vs{\Gamma}{l}{r_{l}}-1)&(\vs{\Gamma}{j}{i}\geq\vs{\Gamma}{l}{r_{l}})
		\end{cases}\\
			&<
		h
	\end{align*}
	Thus the statement follows.
\end{proof}
\begin{prop}\label{GC Lem5}
	Let $\Gamma\colon[n+r]\hookrightarrow[n]\times[r]$ be a global chain and assume that $v\in\mathsf{Inn}_{\fs{}}(\Gamma)\cup\mathsf{Inn}_{\bs{}}(\Gamma)$.
	\begin{enumerate}
		\item For each $i\leq v$, the following hold:
			$$\min\{j|\bs{\Gamma}(j)\geq i\}=\min\{j|\bs{\Gamma\delta_{v}}(j)\geq i\}.$$
		\item Assume that $v\in\mathsf{Inn}_{\fs{}}(\Gamma)$. Then, for each $i>v$, the following hold:
			$$\min\{j|\bs{\Gamma}(j)\geq i\}=\min\{j|\bs{\Gamma\delta_{v}}(j)\geq i-1\}.$$
		\item Assume that $v\in\mathsf{Inn}_{\bs{}}(\Gamma)$. Then, for each $i>v+1$, the following hold:
			$$\min\{j|\bs{\Gamma}(j)\geq i\}=\min\{j|\bs{\Gamma\delta_{v}}(j)\geq i-1\}.$$
	\end{enumerate}
\end{prop}
\begin{proof}
	First, assume that $i\leq v$. Denote $\min\{j|\bs{\Gamma\delta_{v}}(j)\geq i\}$ as $m'$. Then
	\begin{align*}
		\chain{(\Gamma\delta_{v})}{\bs{}}(\bs{\Gamma}(m'))
			&=
		\begin{cases}
			\chain{\Gamma}{\bs{}}(\bs{\Gamma}(m'))&(\bs{\Gamma}(m')<v)\\
			\chain{\Gamma}{\bs{}}(\bs{\Gamma}(m')+1)&(\bs{\Gamma}(m')\geq v)
		\end{cases}\\
			&\geq
		\chain{\Gamma}{\bs{}}(\bs{\Gamma}(m'))\\
			&=
		m'
	\end{align*}
	holds, thus $\bs{\Gamma}(m')\geq\bs{\Gamma\delta_{v}}(m')\geq i$ holds. Therefore $\min\{j|\bs{\Gamma}(j)\geq i\}\leq\min\{j|\bs{\Gamma\delta_{v}}(j)\geq i\}$ holds. Now denote $\min\{j|\bs{\Gamma}(j)\geq i\}$ as $m$. Since $i\leq v$, $\chain{(\Gamma\delta_{v})}{\bs{}}(i-1)=\chain{\Gamma}{\bs{}}(i-1)$ holds. Since $\chain{\Gamma}{\bs{}}(i-1)\geq m$ implies
	\begin{align*}
		i-1
			\geq
		\bs{\Gamma}\chain{\Gamma}{\bs{}}(i-1)
			\geq
		\bs{\Gamma}(m)
			\geq
		i,
	\end{align*}
	$\chain{\Gamma}{\bs{}}(i-1)<m$ holds. Thus $\chain{(\Gamma\delta_{v})}{\bs{}}(i-1)<m$ hold. Threfore $\bs{\Gamma\delta_{v}}(m)\geq i$ holds, and we obtain $\min\{j|\bs{\Gamma}(j)\geq i\}\geq\min\{j|\bs{\Gamma\delta_{v}}(j)\geq i\}$.
	
	Next, assume that $i>v+1$. Denote $\min\{j|\bs{\Gamma\delta_{v}}(j)\geq i-1\}$ as $m'$. Then $\bs{\Gamma\delta_{v}}(m')\geq i-1>v$ holds thus the following hold:
	\begin{align*}
		\bs{\Gamma\delta_{v}}(m')
			&=
		\min\{l|\chain{(\Gamma\delta_{v})}{\bs{}}(l)=m'\}\\
			&=
		\min\{l\geq v|\chain{(\Gamma\delta_{v})}{\bs{}}(l)=m'\}\\
			&=
		\min\{l\geq v|\chain{\Gamma}{\bs{}}(l+1)=m'\}\\
			&=
		\min\{l>v|\chain{\Gamma}{\bs{}}(l)=m'\}-1,\\
		m'
			&>
		\chain{(\Gamma\delta_{v})}{\bs{}}(v)\\
			&=
		\chain{\Gamma}{\bs{}}(v+1)\\
			&\geq
		\chain{\Gamma}{\bs{}}(v).
	\end{align*}
	Thus
	$$i=(i-1)+1\leq\bs{\Gamma\delta_{v}}(m')+1=\min\{l|\chain{\Gamma}{\bs{}}(l)=m'\}-1+1=\bs{\Gamma}(m')$$
	holds. Hence $\min\{j|\bs{\Gamma}(j)\geq i\}\leq\min\{j|\bs{\Gamma\delta_{v}}(j)\geq i-1\}$ holds. Now denote $\min\{j|\bs{\Gamma}(j)\geq i\}$ as $m$. Since $v\leq i-2$,
	$$\chain{(\Gamma\delta_{v})}{\bs{}}(i-2)=\chain{\Gamma}{\bs{}}(i-1)\leq\chain{\Gamma}{\bs{}}(i)\leq\chain{\Gamma}{\bs{}}\bs{\Gamma}(m)=m$$
	holds. Since $\chain{\Gamma}{\bs{}}(i-1)=m$ implies
	$$i\leq\bs{\Gamma}(m)=\bs{\Gamma}\chain{\Gamma}{\bs{}}(i-1)\leq i-1,$$
	$\chain{\Gamma}{\bs{}}(i-1)<m$ holds. Therefore $\bs{\Gamma\delta_{v}}(m)\geq i-1$ holds, we obtain $\min\{j|\bs{\Gamma}(j)\geq i\}\geq\min\{j|\bs{\Gamma\delta_{v}}(j)\geq i-1\}$.
	
	Finally, assume that $v\in\mathsf{Inn}_{\fs{}}(\Gamma)$. Denote $\min\{j|\bs{\Gamma\delta_{v}}(j)\geq v\}$ as $m'$. Then $\bs{\Gamma\delta_{v}}(m')\geq v$ and
	$$\bs{\Gamma\delta_{v}}(m')=\min\{l>v|\chain{\Gamma}{\bs{}}(l)=m'\}-1$$
	hold. Since $v\in\mathsf{Inn}_{\fs{}}(\Gamma)$,
	$$\chain{\Gamma}{\bs{}}(v)<\chain{\Gamma}{\bs{}}(v+1)=\chain{(\Gamma\delta_{v})}{\bs{}}(v)\leq\chain{(\Gamma\delta_{v})}{\bs{}}\bs{\Gamma\delta_{v}}(m')=m'$$
	holds. Thus
	$$v+1\leq\bs{\Gamma\delta_{v}}(m')+1=\min\{l|\chain{\Gamma}{\bs{}}(l)=m'\}-1+1=\bs{\Gamma}(m')$$
	holds. Hence $\min\{j|\bs{\Gamma}(j)\geq v+1\}\leq\min\{j|\bs{\Gamma\delta_{v}}(j)\geq v\}$ holds. Now denote $\min\{j|\bs{\Gamma}(j)\geq v+1\}$ as $m$. Then
	$$\chain{(\Gamma\delta_{v})}{\bs{}}(v-1)=\chain{\Gamma}{\bs{}}(v-1)\leq\chain{\Gamma}{\bs{}}(v)<\chain{\Gamma}{\bs{}}(v+1)\leq\chain{\Gamma}{\bs{}}\bs{\Gamma}(m)=m$$
	thus $\bs{\Gamma\delta_{v}}(m)\geq v$. Therefore $\min\{j|\bs{\Gamma}(j)\geq v+1\}\geq\min\{j|\bs{\Gamma\delta_{v}}(j)\geq v\}$ holds.
\end{proof}

To consider the partition of the product $[n]\times[r]$ into (maximal) chains, it is important to consider the ``pullback of a chain'', that is, the following (commutative) diagram:
\begin{center}
\begin{tikzpicture}[auto]
	\node (11) at (0, 1) {$[p]$};
	\node (12) at (3, 1) {$[n+r]$};
	\node (21) at (0, 0) {$[m]\times[r]$};
	\node (22) at (3, 0) {$[n]\times[r]$};
	\path[draw, ->] (11) --node {$\scriptstyle \beta$} (12);
	\path[draw, right hook->] (11) --node[swap] {$\scriptstyle \Gamma_{2}$} (21);
	\path[draw, right hook->] (12) --node {$\scriptstyle \Gamma_{1}$} (22);
	\path[draw, ->] (21) --node {$\scriptstyle \alpha\times\id{}$} (22);
\end{tikzpicture}.
\end{center}
We check properties of this diagram.
\begin{prop}\label{Gl Lem1}
	Consider the following pullback diagram of a maximal chain $\Gamma_{1}\colon[n+r]\hookrightarrow[n]\times[r]$ along an order-preserving map $\alpha\times\id{}\colon[m]\times[r]\to[n]\times[r]$ where $\alpha\colon[m]\to[n]$ is injective:
	\begin{center}
	\begin{tikzpicture}[auto]
		\node (11) at (0, 1) {$[p]$};
		\node (12) at (3, 1) {$[n+r]$};
		\node (21) at (0, 0) {$[m]\times[r]$};
		\node (22) at (3, 0) {$[n]\times[r]$};
		\path[draw, ->] (11) --node {$\scriptstyle \beta$} (12);
		\path[draw, right hook->] (11) --node[swap] {$\scriptstyle \Gamma_{2}$} (21);
		\path[draw, right hook->] (12) --node {$\scriptstyle \Gamma_{1}$} (22);
		\path[draw, ->] (21) --node {$\scriptstyle \alpha\times\id{}$} (22);
		\path[draw] (0.2, 0.5) -- (0.5, 0.5) -- (0.5, 0.8);
	\end{tikzpicture}.
	\end{center}
	Then $\beta\bs{\Gamma_{2}}=\bs{\Gamma_{1}}\alpha$ holds.
\end{prop}
\begin{proof}
	For each $i\in[m]$,
	$$\Gamma_{1}\bs{\Gamma_{1}}\alpha(i)=(\chain{\Gamma_{1}}{\bs{}}\bs{\Gamma_{1}}\alpha(i),\chain{\Gamma_{1}}{\fs{}}\bs{\Gamma_{1}}\alpha(i))=(\alpha(i),\chain{\Gamma_{1}}{\fs{}}\bs{\Gamma_{1}}\alpha(i))$$
	holds. Thus there is a elements $j\in[p]$ satisfies $\beta(j)=\bs{\Gamma_{1}}\alpha(i)$ and $\Gamma_{2}(j)=(i,\chain{\Gamma_{1}}{\fs{}}\bs{\Gamma_{1}}\alpha(i))$. Especially $\chain{\Gamma_{2}}{\bs{}}$ is surjective. Since
	$$\chain{\Gamma_{1}}{\bs{}}\beta\bs{\Gamma_{2}}(i)=\alpha\chain{\Gamma_{2}}{\bs{}}\bs{\Gamma_{2}}(i)=\alpha(i)$$
	holds by surjectivity of $\chain{\Gamma_{2}}{\bs{}}$, $\bs{\Gamma_{1}}\alpha(i)\leq\beta\bs{\Gamma_{2}}(i)$ holds. Therefore 
	$$\chain{\Gamma_{2}}{\fs{}}(j)=\chain{\Gamma_{1}}{\fs{}}\bs{\Gamma_{1}}\alpha(i)\leq\chain{\Gamma_{1}}{\fs{}}\beta\bs{\Gamma_{2}}(i)=\chain{\Gamma_{2}}{\fs{}}\bs{\Gamma_{2}}(i)$$
	holds, and $\Gamma_{2}(j)\leq\Gamma_{2}\bs{\Gamma_{2}}(i)$ follows. Hence $j\leq\bs{\Gamma_{2}}(i)$ holds.
\end{proof}
\begin{prop}\label{Gl Lem2}
	Let $\Gamma_{1}\colon[n+r]\hookrightarrow[n]\times[r],\Gamma_{2}\colon[m+r]\hookrightarrow[m]\times[r]$ be maximal chains and $\alpha\colon[m]\to[n],\beta\colon[m+r]\to[n+r]$ be order-preserving maps, and assume that $(\alpha\times\id{})\Gamma_{2}=\Gamma_{1}\beta$ holds.
	\begin{center}
	\begin{tikzpicture}[auto]
		\node (11) at (0, 1) {$[m+r]$};
		\node (12) at (3, 1) {$[n+r]$};
		\node (21) at (0, 0) {$[m]\times[r]$};
		\node (22) at (3, 0) {$[n]\times[r]$};
		\path[draw, ->] (11) --node {$\scriptstyle \beta$} (12);
		\path[draw, right hook->] (11) --node[swap] {$\scriptstyle \Gamma_{2}$} (21);
		\path[draw, right hook->] (12) --node {$\scriptstyle \Gamma_{1}$} (22);
		\path[draw, ->] (21) --node {$\scriptstyle \alpha\times\id{}$} (22);
	\end{tikzpicture}.
	\end{center}
	Then $\beta\fs{\Gamma_{2}}=\fs{\Gamma_{1}}$ holds.
\end{prop}
\begin{proof}
	Let $i$ be a positive integer $1,\dots,r$. $\chain{\Gamma_{2}}{\fs{}}(\fs{\Gamma_{2}}(i)-1)=i-1$ holds by assumption therefore
	$$\chain{\Gamma_{2}}{\bs{}}(\fs{\Gamma_{2}}(i)-1)=(\fs{\Gamma_{2}}(i)-1)-\chain{\Gamma_{2}}{\fs{}}(\fs{\Gamma_{2}}(i)-1)=(\fs{\Gamma_{2}}(i)-1)-(i-1)=\fs{\Gamma_{2}}(i)-\chain{\Gamma_{2}}{\fs{}}\fs{\Gamma_{2}}(i)=\chain{\Gamma_{2}}{\bs{}}\fs{\Gamma_{2}}(i)$$
	holds. Thus
	$$\beta(\fs{\Gamma_{2}}(i)-1)=\alpha\chain{\Gamma_{2}}{\bs{}}(\fs{\Gamma_{2}}(i)-1)+\chain{\Gamma_{2}}{\fs{}}(\fs{\Gamma_{2}}(i)-1)=\alpha\chain{\Gamma_{2}}{\bs{}}\fs{\Gamma_{2}}(i)+\chain{\Gamma_{2}}{\fs{}}\fs{\Gamma_{2}}(i)-1=\beta\fs{\Gamma_{2}}(i)-1$$
	holds (by proposition \ref{GC Lem3}). Hence
	$$\chain{\Gamma_{1}}{\fs{}}(\beta\fs{\Gamma_{2}}(i)-1)=\chain{\Gamma_{1}}{\fs{}}\beta(\fs{\Gamma_{2}}(i)-1)=\chain{\Gamma_{2}}{\fs{}}(\fs{\Gamma_{2}}(i)-1)=i-1$$
	holds, and $\beta\fs{\Gamma_{2}}=\fs{\Gamma_{1}}$ follows.
\end{proof}
\begin{prop}\label{Gl Lem3}
	Let $\Gamma_{1}\colon[n+r]\hookrightarrow[n]\times[r],\Gamma_{2}\colon[m+r]\hookrightarrow[m]\times[r]$ be maximal chains and $\alpha\colon[m]\to[n],\beta\colon[m+r]\to[n+r]$ be order-preserving maps, and assume that $(\alpha\times\id{})\Gamma_{2}=\Gamma_{1}\beta$ holds.
	\begin{center}
	\begin{tikzpicture}[auto]
		\node (11) at (0, 1) {$[m+r]$};
		\node (12) at (3, 1) {$[n+r]$};
		\node (21) at (0, 0) {$[m]\times[r]$};
		\node (22) at (3, 0) {$[n]\times[r]$};
		\path[draw, ->] (11) --node {$\scriptstyle \beta$} (12);
		\path[draw, right hook->] (11) --node[swap] {$\scriptstyle \Gamma_{2}$} (21);
		\path[draw, right hook->] (12) --node {$\scriptstyle \Gamma_{1}$} (22);
		\path[draw, ->] (21) --node {$\scriptstyle \alpha\times\id{}$} (22);
	\end{tikzpicture}.
	\end{center}
	Then $\beta\us{\Gamma_{2}}=\us{\Gamma_{1}}$ holds.
\end{prop}
\begin{proof}
	Let $i$ be a positive integer $1,\dots,r$. For any positive integer $j$ satisfies $\min\{j|\fs{\Gamma_{2}}(j)-j=\fs{\Gamma_{2}}(i)-i\}\leq j<i$, we can show $\beta\fs{\Gamma_{2}}(j)+1=\beta\fs{\Gamma_{2}}(j+1)$ in the same way as above (part of the proof of proposition \ref{Gl Lem2}) since $\chain{\Gamma_{2}}{\bs{}}\fs{\Gamma_{2}}(j)=\chain{\Gamma_{2}}{\bs{}}\fs{\Gamma_{2}}(j+1)$ holds. Therefore
	$$\fs{\Gamma_{1}}(j)+1=\beta\fs{\Gamma_{2}}(j)+1=\beta\fs{\Gamma_{2}}(j+1)=\fs{\Gamma_{1}}(j+1)$$
	holds. Thus
	$\min\{j|\fs{\Gamma_{1}}(j)-j=\fs{\Gamma_{1}}(i)-i\}\leq\min\{j|\fs{\Gamma_{2}}(j)-j=\fs{\Gamma_{2}}(i)-i\}$
	holds (as elements of $[r]$). And then
	\begin{align*}
		&\chain{\Gamma_{1}}{\bs{}}\fs{\Gamma_{1}}(\min\{j|\fs{\Gamma_{1}}(j)-j=\fs{\Gamma_{1}}(i)-i\})\\
			=
		&\fs{\Gamma_{1}}(\min\{j|\fs{\Gamma_{1}}(j)-j=\fs{\Gamma_{1}}(i)-i\})-\chain{\Gamma_{1}}{\fs{}}\fs{\Gamma_{1}}(\min\{j|\fs{\Gamma_{1}}(j)-j=\fs{\Gamma_{1}}(i)-i\})\\
			=
		&\fs{\Gamma_{1}}(\min\{j|\fs{\Gamma_{1}}(j)-j=\fs{\Gamma_{1}}(i)-i\})-\min\{j|\fs{\Gamma_{1}}(j)-j=\fs{\Gamma_{1}}(i)-i\}\\
			=
		&\fs{\Gamma_{1}}(i)-i\\
			=
		&\fs{\Gamma_{1}}(\min\{j|\fs{\Gamma_{2}}(j)-j=\fs{\Gamma_{2}}(i)-i\})-\min\{j|\fs{\Gamma_{2}}(j)-j=\fs{\Gamma_{2}}(i)-i\}\\
			=
		&\chain{\Gamma_{1}}{\bs{}}\fs{\Gamma_{1}}(\min\{j|\fs{\Gamma_{2}}(j)-j=\fs{\Gamma_{2}}(i)-i\})
	\end{align*}
	hols. Hence the following holds:
	\begin{align*}
		\us{\Gamma_{1}}(i)
			&=
		\bs{\Gamma_{1}}\chain{\Gamma_{1}}{\bs{}}\us{\Gamma_{1}}(i)\\
			&=
		\bs{\Gamma_{1}}\chain{\Gamma_{1}}{\bs{}}(\fs{\Gamma_{1}}(\min\{j|\fs{\Gamma_{1}}(j)-j=\fs{\Gamma_{1}}(i)-i\})-1)\\
			&=
		\bs{\Gamma_{1}}\chain{\Gamma_{1}}{\bs{}}\fs{\Gamma_{1}}(\min\{j|\fs{\Gamma_{1}}(j)-j=\fs{\Gamma_{1}}(i)-i\})\\
			&=
		\bs{\Gamma_{1}}\chain{\Gamma_{1}}{\bs{}}\beta\fs{\Gamma_{2}}(\min\{j|\fs{\Gamma_{2}}(j)-j=\fs{\Gamma_{2}}(i)-i\})\\
			&=
		\bs{\Gamma_{1}}\alpha\chain{\Gamma_{2}}{\bs{}}\fs{\Gamma_{2}}(\min\{j|\fs{\Gamma_{2}}(j)-j=\fs{\Gamma_{2}}(i)-i\})\\
			&=
		\beta\bs{\Gamma_{2}}\chain{\Gamma_{2}}{\bs{}}\fs{\Gamma_{2}}(\min\{j|\fs{\Gamma_{2}}(j)-j=\fs{\Gamma_{2}}(i)-i\})\\
			&=
		\beta\bs{\Gamma_{2}}\chain{\Gamma_{2}}{\bs{}}\us{\Gamma_{2}}(i)\\
			&=
		\beta\us{\Gamma_{2}}(i).
	\end{align*}
	Therefore, the statement holds.
\end{proof}
\begin{prop}\label{Gl Lem4}
	Let $\Gamma_{1}\colon[n+r]\hookrightarrow[n]\times[r],\Gamma_{2}\colon[m+r]\hookrightarrow[m]\times[r]$ be maximal chains and $\alpha\colon[m]\to[n],\beta\colon[m+r]\to[n+r]$ be injective order-preserving maps, and assume that $(\alpha\times\id{})\Gamma_{2}=\Gamma_{1}\beta$ holds.
	\begin{center}
	\begin{tikzpicture}[auto]
		\node (11) at (0, 1) {$[m+r]$};
		\node (12) at (3, 1) {$[n+r]$};
		\node (21) at (0, 0) {$[m]\times[r]$};
		\node (22) at (3, 0) {$[n]\times[r]$};
		\path[draw, ->] (11) --node {$\scriptstyle \beta$} (12);
		\path[draw, right hook->] (11) --node[swap] {$\scriptstyle \Gamma_{2}$} (21);
		\path[draw, right hook->] (12) --node {$\scriptstyle \Gamma_{1}$} (22);
		\path[draw, ->] (21) --node {$\scriptstyle \alpha\times\id{}$} (22);
	\end{tikzpicture}.
	\end{center}
	Then the following holds for any $i=1,\dots,r$:
	$$\min\{j|\beta(j)\geq\fs{\Gamma_{1}}(i)+1\}=\fs{\Gamma_{2}}(i)+1.$$
\end{prop}
\begin{proof}
	Proposition \ref{Gl Lem2} and the injectivity of $\beta$ implies
	$$\fs{\Gamma_{1}}(i)+1=\beta\fs{\Gamma_{2}}(i)+1\leq\beta(\fs{\Gamma_{2}}(i)+1).$$
	And, for each $j$ which satisfies $\fs{\Gamma_{1}}(i)+1\leq\beta(j)\leq\beta(\fs{\Gamma_{2}}(i)+1)$,
	$$\fs{\Gamma_{2}}(i)<j\leq\fs{\Gamma_{2}}(i)+1$$
	holds.
\end{proof}
\begin{prop}\label{Gl Lem5}
	Let $\Gamma_{2}\colon[p]\hookrightarrow[m]\times[r]$ be a pullback of a maximal chain $\Gamma_{1}\colon[n+r]\hookrightarrow[n]\times[r]$ along an order-preserving map $\alpha\times\id{}\colon[m]\times[r]\to[n]\times[r]$, where $\alpha$ is injective.
	\begin{center}
	\begin{tikzpicture}[auto]
		\node (11) at (0, 1) {$[p]$};
		\node (12) at (3, 1) {$[n+r]$};
		\node (21) at (0, 0) {$[m]\times[r]$};
		\node (22) at (3, 0) {$[n]\times[r]$};
		\path[draw, ->] (11) --node {$\scriptstyle \beta$} (12);
		\path[draw, right hook->] (11) --node[swap] {$\scriptstyle \Gamma_{2}$} (21);
		\path[draw, right hook->] (12) --node {$\scriptstyle \Gamma_{1}$} (22);
		\path[draw, ->] (21) --node {$\scriptstyle \alpha\times\id{}$} (22);
		\path[draw] (0.2, 0.5) -- (0.5, 0.5) -- (0.5, 0.8);
	\end{tikzpicture}.
	\end{center}
	If $\Gamma_{2}$ is not maximal, there exists an element $l\not\in\mathrm{Im}\alpha$ such that $\bs{\Gamma_{1}}(l)+1\not\in\mathrm{Im}\bs{\Gamma}$.
\end{prop}
\begin{proof}
	$\bs{\Gamma_{2}}$ is injective since proposition \ref{Gl Lem1} implies $\beta\bs{\Gamma_{2}}=\bs{\Gamma_{1}}\alpha$. Assume that $\bs{\Gamma_{1}}(l)+1\in\mathrm{Im}\bs{\Gamma_{1}}$ holds for any $l\not\in\mathrm{Im}\alpha$. Since, for each positive integer $i=1,\dots,r$,
	$$\bs{\Gamma_{1}}\chain{\Gamma_{1}}{\bs{}}\fs{\Gamma_{1}}(i)+1\leq\fs{\Gamma_{1}}(i)<\bs{\Gamma_{1}}(\chain{\Gamma_{1}}{\bs{}}\fs{\Gamma_{1}}(i)+1)$$
	holds, there exists an element $j\in[m]$ satisfies $\alpha(j)=\chain{\Gamma_{1}}{\bs{}}\fs{\Gamma_{1}}(i)$ by above assumption. Thus there exists an element $h\in[p]$ satisfies $\Gamma_{2}(h)=(j,i)$. Especially $\Gamma_{2}(h-1)=(j,i-1)$ holds therefore $h\not\in\mathrm{Im}\bs{\Gamma_{2}}$. It contradicts $p<m+r$ since $\bs{\Gamma_{2}}$ is injective.
\end{proof}
Then we will see how such a commutative diagram is given.
\begin{prop}\label{Gl Lem6}
	Let $\Gamma_{1}\colon[n+r]\hookrightarrow[n]\times[r]$ be a maximal chain and $\alpha\colon[m]\to[n]$ be an injective order-preserving map. Then the pullback P of $\Gamma_{1}$ along $\alpha\times\id{}\colon[m]\times[r]\to[n]\times[r]$ is a total ordered set.
	\begin{center}
	\begin{tikzpicture}[auto]
		\node (11) at (0, 1) {$P$};
		\node (12) at (3, 1) {$[n+r]$};
		\node (21) at (0, 0) {$[m]\times[r]$};
		\node (22) at (3, 0) {$[n]\times[r]$};
		\path[draw, ->, dashed] (11) --node {$\scriptstyle \beta$} (12);
		\path[draw, right hook->, dashed] (11) --node[swap] {$\scriptstyle \Gamma_{2}$} (21);
		\path[draw, right hook->] (12) --node {$\scriptstyle \Gamma_{1}$} (22);
		\path[draw, ->] (21) --node {$\scriptstyle \alpha\times\id{}$} (22);
		\path[draw] (0.2, 0.5) -- (0.5, 0.5) -- (0.5, 0.8);
	\end{tikzpicture}.
	\end{center}
\end{prop}
\begin{proof}
	Let $(i,j)$ be a pair of elements of $P$. We can assume that $\beta(i)\leq\beta(j)$ holds. Then
	\begin{align*}
		\chain{\Gamma_{2}}{\fs{}}(i)&\leq\chain{\Gamma_{2}}{\fs{}}(j),\\
		\alpha\chain{\Gamma_{2}}{\bs{}}(i)&\leq\alpha\chain{\Gamma_{2}}{\bs{}}(j)
	\end{align*}
	holds. Since $\alpha$ is injective, $\chain{\Gamma_{2}}{\bs{}}(i)\leq\chain{\Gamma_{2}}{\bs{}}(j)$ holds in $[m]$. Thus $i\leq j$ holds.
\end{proof}
\begin{lem}\label{Gl Lem7}
	Let $\alpha\colon[m]\to[n]$ be an order-preserving map and $\Gamma\colon[m+r]\hookrightarrow[m]\times[r]$ be a maximal chain. Then there exists a unique pair $(\alpha_{\ast}\Gamma\colon[n+r]\hookrightarrow[n]\times[r],\Gamma_{\ast}\alpha\colon[m+r]\to[n+r])$ of a maximal chain and an order-preserving map which satisfies
	$$(\alpha\times\id{})\Gamma=(\alpha_{\ast}\Gamma)(\Gamma^{\ast}\alpha).$$
	\begin{center}
	\begin{tikzpicture}[auto]
		\node (11) at (0, 1) {$[m+r]$};
		\node (12) at (3, 1) {$[n+r]$};
		\node (21) at (0, 0) {$[m]\times[r]$};
		\node (22) at (3, 0) {$[n]\times[r]$};
		\path[draw, ->, dashed] (11) --node {$\scriptstyle \Gamma^{\ast}\alpha$} (12);
		\path[draw, right hook->] (11) --node[swap] {$\scriptstyle \Gamma$} (21);
		\path[draw, right hook->, dashed] (12) --node {$\scriptstyle \alpha_{\ast}\Gamma$} (22);
		\path[draw, ->] (21) --node {$\scriptstyle \alpha\times\id{}$} (22);
	\end{tikzpicture}
	\end{center}
\end{lem}
\begin{proof}
	Assume that there exists a pair $(\Gamma',\beta)$ satisfies the above conditions. For each $j>0$,
	$$\chain{\Gamma'}{\fs{}}\beta(\fs{\Gamma}(j)-1)=\chain{\Gamma}{\fs{}}(\fs{\Gamma}(j)-1)<\chain{\Gamma}{\fs{}}\fs{\Gamma}(j)=\chain{\Gamma'}{\fs{}}\beta\fs{\Gamma}(j)$$
	holds thus $\beta(\fs{\Gamma}(j)-1)<\beta\fs{\Gamma}(j)$ holds. Assume that $\beta(\fs{\Gamma}(j)-1)+1<\beta\fs{\Gamma}(j)$ holds. Then
	$$\beta(\fs{\Gamma}(j)-1)<\beta(\fs{\Gamma}(j)-1)+1<\beta\fs{\Gamma}(j)$$
	holds. Hence
	$$(\alpha\times\id{})\Gamma(\fs{\Gamma}(j)-1)<\Gamma'(\beta(\fs{\Gamma}(j)-1)+1)<(\alpha\times\id{})\Gamma\fs{\Gamma}(j)$$
	follows from injectivity of $\Gamma'$. Therefore
	$$\alpha\chain{\Gamma}{\bs{}}(\fs{\Gamma}(j)-1)\leq\chain{\Gamma'}{\bs{}}(\beta(\fs{\Gamma}(j)-1)+1)\leq\alpha\chain{\Gamma}{\bs{}}\fs{\Gamma}(j)=\alpha\chain{\Gamma}{\bs{}}(\fs{\Gamma}(j)-1)$$
	holds. On the other hands,
	$$j-1=\chain{\Gamma}{\fs{}}(\fs{\Gamma}(j)-1)\leq\chain{\Gamma'}{\fs{}}(\beta(\fs{\Gamma}(j)-1)+1)\leq\chain{\Gamma}{\fs{}}\fs{\Gamma}(j)=j$$
	holds. It contradicts injectivity of $\Gamma'$. Hence $\beta(\fs{\Gamma}(j)-1)=\beta\fs{\Gamma}(j)-1$ holds. Define $\beta\fs{\Gamma}(r+1)$ as $n+r+1$. Then we obtain a partition
	\begin{align}\label{partition}
		[n+r]=\bigcup_{j=0}^{r}\{i\in[n+r]|\beta\fs{\Gamma}(j)\leq i<\beta\fs{\Gamma}(j+1)\}.
	\end{align}
	Let $j$ be a non-negative integer that satisfies $j\leq r$. Since, for each $i\in[n+r]$ satisfying $\beta\fs{\Gamma}(j)\leq i<\beta\fs{\Gamma}(j+1)$,
	\begin{align*}
		j
			=
		\chain{\Gamma}{\fs{}}\fs{\Gamma}(j)
			=
		\chain{\Gamma'}{\fs{}}\beta\fs{\Gamma}(j)
			\leq
		\chain{\Gamma'}{\fs{}}(i)
			\leq
		\chain{\Gamma'}{\fs{}}(\beta\fs{\Gamma}(j+1)-1)
			=
		\chain{\Gamma'}{\fs{}}\beta(\fs{\Gamma}(j+1)-1)
			=
		\chain{\Gamma}{\fs{}}(\fs{\Gamma}(j+1)-1)
			=
		j
	\end{align*}
	holds, $\chain{\Gamma'}{\fs{}}(i)=j$ holds. Hence
	$$\chain{\Gamma'}{\bs{}}\beta\fs{\Gamma}(j)\leq\chain{\Gamma'}{\bs{}}(i)<\chain{\Gamma'}{\bs{}}(\beta\fs{\Gamma}(j+1)-1)\leq\chain{\Gamma'}{\bs{}}\beta\fs{\Gamma}(j+1)$$
	holds for ecah $i\in[n+r]$ satisfies $\beta\fs{\Gamma}(j)\leq i<\beta\fs{\Gamma}(j+1)-1$. In addition,
	\begin{align*}
		\left|\bigcup_{j=0}^{r}\{i\in[n+r]|\beta\fs{\Gamma}(j)\leq i<\beta\fs{\Gamma}(j+1)-1\}\cup\{n+r\}\right|
			&=
		\sum_{j=0}^{r}|\{i\in[n+r]|\beta\fs{\Gamma}(j)\leq i<\beta\fs{\Gamma}(j+1)-1\}|+1\\
			&=
		\sum_{j=0}^{r}((\beta\fs{\Gamma}(j+1)-1)-\beta\fs{\Gamma}(j))+1\\
			&=
		(n+r+1)-(r+1)+1\\
			&=
		n+1
	\end{align*}
	holds. Therefore $(\Gamma',\beta)$ can be recovered from the partition \ref{partition} of $[n+r]$ which is determined by $\beta\fs{\Gamma}$. Furthermore
	\begin{align*}
		\beta\fs{\Gamma}(j)
			&=
		\sum_{l=0}^{j-1}|\{i\in[n+r]|\beta\fs{\Gamma}(j)\leq i<\beta\fs{\Gamma}(j+1)\}|\\
			&=
		\sum_{l=0}^{j-1}|\{\chain{\Gamma'}{\bs{}}(i)\in[n]|\beta\fs{\Gamma}(j)\leq i<\beta\fs{\Gamma}(j+1)\}|\\
			&=
		\sum_{l=0}^{j-1}(\alpha\chain{\Gamma}{\bs{}}\fs{\Gamma}(l+1)+1-\alpha\chain{\Gamma}{\bs{}}\fs{\Gamma}(l))\\
			&=
		\alpha\chain{\Gamma}{\bs{}}\fs{\Gamma}(j)+j
	\end{align*}
	holds, therefore $(\Gamma',\beta)$ is determined by $\Gamma$ and $\alpha$.
\end{proof}
\begin{prop}\label{Gl Lem8}
	let $\Gamma\colon[n+r]\hookrightarrow[n]\times[r]$ be a maximal chain. Then, for each $v\in\mathbb{O}_{\Gamma}$, there exists a unique pair of a maximal chain $\Gamma_{v}\colon[n+r-1]\to[n]\times[r-1]$ and an element $h\in[r]$ which satisfies
	$$\Gamma\delta_{v}=(1\times\delta_{h})\Gamma_{v}.$$
	\begin{center}
	\begin{tikzpicture}[auto]
		\node (11) at (0, 1) {$[n+r-1]$};
		\node (12) at (3, 1) {$[n+r]$};
		\node (21) at (0, 0) {$[n]\times[r-1]$};
		\node (22) at (3, 0) {$[n]\times[r]$};
		\path[draw, ->] (11) --node {$\scriptstyle \delta_{v}$} (12);
		\path[draw, right hook->] (11) --node[swap] {$\scriptstyle \Gamma_{v}$} (21);
		\path[draw, right hook->] (12) --node {$\scriptstyle \Gamma$} (22);
		\path[draw, ->] (21) --node {$\scriptstyle \id{}\times\delta_{h}$} (22);
	\end{tikzpicture}.
	\end{center}
\end{prop}
\begin{proof}
	For each element $i\in[n+r-1]$, the folowing holds:
	\begin{align*}
		\chain{\Gamma}{\fs{}}\delta_{v}(i)
			&=
		\begin{cases}
			\chain{\Gamma}{\fs{}}(i)&(i<v)\\
			\chain{\Gamma}{\fs{}}(i+1)&(i\geq v)
		\end{cases}\\
			&\neq
		\chain{\Gamma}{\fs{}}(v)
	\end{align*}
	Thus, for any pair $(\Gamma_{v},h)$ which satisfies the above condition, $h=\chain{\Gamma}{\fs{}}(v)$ and
	\begin{align*}
		\Gamma_{v}(i)
			=
		\begin{cases}
			\Gamma(i)&(i<v)\\
			(\chain{\Gamma}{\bs{}}(i),\chain{\Gamma}{\fs{}}(i+1)-1)&(i\geq v)
		\end{cases}
	\end{align*}
	hold.
\end{proof}

Let $n$ and $r$ be non-negative integers, $\mathfrak{g}$ be a connected $L_{\infty}$-algebra and $\Gamma\colon[n+r]\hookrightarrow[n]\times[r]$ be a maximal chain. Since $\mathrm{Im}\bs{\Gamma}\cap\mathrm{Im}\fs{\Gamma}=\emptyset$ and $\mathrm{Im}\bs{\Gamma}\cup\mathrm{Im}\fs{\Gamma}=[n+r]$ hold, we can define a (differential graded) ring morphism
$$\qdR{\dq}{}{n+r}{\mathbb{U}_{\infty}\mathfrak{g}}^{\wedge}\xhookrightarrow{\incl{\Gamma}}\qdR{\dq}{}{n,r}{\mathbb{U}_{\infty}\mathfrak{g}}^{\wedge}\coloneqq\prod_{p+\bullet=q}\mathbb{U}_{\infty}\mathfrak{g}_{p}\otimes\mathsf{Sym}(\langle\df{}\bs{1},\dots,\df{}\bs{n},\df{}\fs{1},\dots,\df{}\fs{r}\rangle_{\Z\langle\dq,\bs{1},\dots,\bs{n},\fs{1},\dots,\fs{r}\rangle}[1])^{q}$$
as follows:
\begin{align*}
	x_{i}^{[N]}
		&\mapsto
	\begin{cases}
		\bs{j}^{[N]}&(\bs{\Gamma}(j)=i)\\
		\fs{j}^{[N]}&(\fs{\Gamma}(j)=i)
	\end{cases},&
	\df{}x_{i}
		&\mapsto
	\begin{cases}
		\df{}\bs{j}&(\bs{\Gamma}(j)=i)\\
		\df{}\fs{j}&(\fs{\Gamma}(j)=i)
	\end{cases}.
\end{align*}
On the other hand, we obtain a retraction $\mathrm{re}_{\Gamma}\colon\qdR{\dq}{}{n,r}{\mathbb{U}_{\infty}\mathfrak{g}}^{\wedge}\hookrightarrow\qdR{\dq}{}{n+r}{\mathbb{U}_{\infty}\mathfrak{g}}^{\wedge}$ as
\begin{align*}
	\bs{i}^{[N]}
		&\mapsto
	x_{\bs{\Gamma}(i)}^{[N]},&
	\df{}\bs{i}
		&\mapsto
	\df{}x_{\bs{\Gamma}(i)},\\
	\fs{i}^{[N]}
		&\mapsto
	x_{\fs{\Gamma}(i)}^{[N]},&
	\df{}\fs{i}
		&\mapsto
	\df{}x_{\fs{\Gamma}(i)}.
\end{align*}
They give morphisms as follows:
\begin{align*}
	\prod_{\Gamma\colon[n+r]\hookrightarrow[n]\times[r]}\incl{\Gamma}&\colon\prod_{\Gamma}\qdR{\dq}{}{n+r}{\mathbb{U}_{\infty}\mathfrak{g}}^{\wedge}\to\prod_{\Gamma}\qdR{\dq}{}{n,r}{\mathbb{U}_{\infty}\mathfrak{g}}^{\wedge}&(\omega_{\Gamma})_{\Gamma}&\mapsto(\incl{\Gamma}(\omega_{\Gamma}))_{\Gamma},\\
	\prod_{\Gamma\colon[n+r]\hookrightarrow[n]\times[r]}\mathrm{re}_{\Gamma}&\colon\prod_{\Gamma}\qdR{\dq}{}{n,r}{\mathbb{U}_{\infty}\mathfrak{g}}^{\wedge}\to\prod_{\Gamma}\qdR{\dq}{}{n+r}{\mathbb{U}_{\infty}\mathfrak{g}}^{\wedge}&(\omega_{\Gamma})_{\Gamma}&\mapsto(\mathrm{re}_{\Gamma}(\omega_{\Gamma}))_{\Gamma}.
\end{align*}
In addition, by using partition \ref{chain partition}, we obtain an embedding
$$[\simplex{}{r},\qdR{\dq}{}{}{\mathbb{U}_{\infty}\mathfrak{g}}^{\wedge}]_{n}\cong\Hom{}(\bigcup_{\Gamma}\simplex{}{n+r},\qdR{\dq}{}{}{\mathbb{U}_{\infty}\mathfrak{g}}^{\wedge})\subset\prod_{\Gamma}\qdR{\dq}{}{n+r}{\mathbb{U}_{\infty}\mathfrak{g}}^{\wedge}\subset\prod_{\Gamma}\qdR{\dq}{}{n,r}{\mathbb{U}_{\infty}\mathfrak{g}}^{\wedge}.$$
Each order-preserving map $\alpha\colon[m]\to[n]$ gives a dg algebra morphism $\alpha\colon\qdR{\dq}{}{n,r}{\mathbb{U}_{\infty}\mathfrak{g}}^{\wedge}\to\qdR{\dq}{}{m,r}{\mathbb{U}_{\infty}\mathfrak{g}}^{\wedge}$ as
\begin{align*}
	\bs{i}^{[N]}
		&\mapsto
	\begin{cases}
		\bs{\min\{j|\alpha(j)\geq i\}}^{[N]}&(\alpha(m)\geq i)\\
		0&(\alpha(m)<i)
	\end{cases},&
	\df{}\bs{i}
		&\mapsto
	\begin{cases}
		\df{}\bs{\min\{j|\alpha(j)\geq i\}}&(\alpha(m)\geq i)\\
		0&(\alpha(m)<i)
	\end{cases},\\
	\fs{i}^{[N]}
		&\mapsto
	\fs{i}^{[N]},&
	\df{}\fs{i}
		&\mapsto
	\df{}\fs{i}.
\end{align*}
Furthermore, by using Lemma \ref{Gl Lem7}, we obtain a morphism
\begin{center}
\begin{tikzpicture}[auto]
	\node (11) at (0, 1) {$\prod_{\Gamma}\qdR{\dq}{}{n,r}{\mathbb{U}_{\infty}\mathfrak{g}}^{\wedge}$};
	\node (12) at (3, 1) {$\prod_{\Gamma}\qdR{\dq}{}{m,r}{\mathbb{U}_{\infty}\mathfrak{g}}^{\wedge}$};
	\node (21) at (0, 0) {$\qdR{\dq}{}{n,r}{\mathbb{U}_{\infty}\mathfrak{g}}^{\wedge}$};
	\node (22) at (3, 0) {$\qdR{\dq}{}{m,r}{\mathbb{U}_{\infty}\mathfrak{g}}^{\wedge}$};
	\path[draw, ->] (21) --node[swap] {$\scriptstyle \alpha$} (22);
	\path[draw, ->] (11) --node[swap] {$\scriptstyle \proj{\alpha_{\ast}\Gamma}$} (21);
	\path[draw, ->] (12) --node {$\scriptstyle \proj{\Gamma}$} (22);
	\path[draw, ->, dashed] (11) --node {$\scriptstyle \alpha$} (12);
\end{tikzpicture}.
\end{center}
\begin{prop}\label{Stb Lem1}
	Let $\alpha\colon[m]\to[n]$ be an order-preserving map and $\textcolor{eldritch}{P}\colon[m+r]\hookrightarrow[m]\times[r]$ be a maximal chain. Furthermore, let $(\alpha_{\ast}P\colon[n+r]\hookrightarrow[n]\times[r],P_{\ast}\alpha\colon[m+r]\to[n+r])$ be a pair of a maximal chain and an order-preserving map such that the following diagram commute:
	\begin{center}
	\begin{tikzpicture}[auto]
		\node (11) at (0, 1) {$[m+r]$};
		\node (12) at (3, 1) {$[n+r]$};
		\node (21) at (0, 0) {$[m]\times[r]$};
		\node (22) at (3, 0) {$[n]\times[r]$};
		\path[draw, ->, dashed] (11) --node {$\scriptstyle P^{\ast}\alpha$} (12);
		\path[draw, right hook->] (11) --node[swap] {$\scriptstyle P$} (21);
		\path[draw, right hook->, dashed] (12) --node {$\scriptstyle \alpha_{\ast}P$} (22);
		\path[draw, ->] (21) --node {$\scriptstyle \alpha\times\id{}$} (22);
	\end{tikzpicture}.
	\end{center}
	Then the following holds for each $1\leq i\leq n+r$:
	\begin{align*}
		\alpha^{\ast}\incl{\alpha_{\ast}P}(X_{i}^{[N]})=\incl{P}(P^{\ast}\alpha)^{\ast}(X_{i}^{[N]}).
	\end{align*}
\end{prop}
\begin{proof}
	Recall that there is a partition of $[n+r]$ into $\mathrm{Im}\bs{\alpha_{\ast}\Gamma}\cup\mathrm{Im}\fs{\alpha_{\ast}\Gamma}$ since $\alpha_{\ast}\Gamma$ is a maximal chain.
	
	First, assume that there exists an element $j\in[n]$ satisfying $\bs{\alpha_{\ast}\Gamma}(j)=i$. Then
	\begin{align*}
		\chain{(\alpha_{\ast}\Gamma)}{\bs{}}(\Gamma^{\ast}\alpha)\bs{\Gamma}(\min\{h|\alpha(h)\geq j\})
			&=
		\alpha\chain{\Gamma}{\bs{}}\bs{\Gamma}(\min\{h|\alpha(h)\geq j\})\\
			&=
		\alpha(\min\{h|\alpha(h)\geq j\})\\
			&\geq
		j\\
		\chain{(\alpha_{\ast}\Gamma)}{\bs{}}(\Gamma^{\ast}\alpha)\bs{\Gamma}(\min\{h|\alpha(h)\geq j\}-1)
			&=
		\alpha\chain{\Gamma}{\bs{}}\bs{\Gamma}(\min\{h|\alpha(h)\geq j\}-1)\\
			&=
		\alpha(\min\{h|\alpha(h)\geq j\}-1)\\
			&<
		j
	\end{align*}
	hold, therefore
	\begin{align*}
		(\Gamma^{\ast}\alpha)\bs{\Gamma}(\min\{h|\alpha(h)\geq j\})
			&\geq
		\bs{\alpha_{\ast}\Gamma}(j)\\
			&=
		i\\
		(\Gamma^{\ast}\alpha)\bs{\Gamma}(\min\{h|\alpha(h)\geq j\}-1)
			&<
		\bs{\alpha_{\ast}\Gamma}(j)\\
			&=
		i
	\end{align*}
	hold. Hence $\bs{\Gamma}(\min\{h|\alpha(h)\geq j\})=\min\{h|(\Gamma^{\ast}\alpha)(h)\geq i\}$ holds.
	
	Next, assume that there exists an element $j\in[n]$ satisfying $\fs{\alpha_{\ast}\Gamma}(j)=i$. Then
	$$(\Gamma^{\ast}\alpha)\fs{\Gamma}(j)=\fs{\alpha_{\ast}\Gamma}(j)=i$$
	follows from Proposition \ref{Gl Lem2}. On the other hands
	$$\chain{(\alpha_{\ast}\Gamma)}{\fs{}}(\Gamma^{\ast}\alpha)(\fs{\Gamma}(j)-1)=\chain{\Gamma}{\fs{}}(\fs{\Gamma}(j)-1)=j-1$$
	hold therefore
	$$(\Gamma^{\ast}\alpha)(\fs{\Gamma}(j)-1)<\fs{\alpha_{\ast}\Gamma}(j)=i$$
	hold. Thus $\fs{\Gamma}(j)=\min\{h|(\Gamma^{\ast}\alpha)(h)\geq i\}$ follows.
	
	Therefore the following follows:
	\begin{align*}
		\alpha^{\ast}\incl{\alpha_{\ast}\Gamma}(x_{i}^{[N]})
			&=
		\begin{cases}
			\alpha^{\ast}(\bs{j}^{[N]})&(\bs{\alpha_{\ast}\Gamma}(j)=i)\\
			\alpha^{\ast}(\fs{j}^{[N]})&(\fs{\alpha_{\ast}\Gamma}(j)=i)
		\end{cases}\\
			&=
		\begin{cases}
			\bs{\min\{h|\alpha(h)\geq j\}}^{[N]}&(\bs{\alpha_{\ast}\Gamma}(j)=i,\alpha(m)\geq j)\\
			0&(\bs{\alpha_{\ast}\Gamma}(j)=i,\alpha(m)<j)\\
			\fs{j}^{[N]}&(\fs{\alpha_{\ast}\Gamma}(j)=i)
		\end{cases}\\
			&=
		\begin{cases}
			\bs{j}^{[N]}&(\bs{\Gamma}(j)=\min\{h|(\Gamma^{\ast}\alpha)(h)\geq i\})\\
			\fs{j}^{[N]}&(\fs{\Gamma}(j)=\min\{h|(\Gamma^{\ast}\alpha)(h)\geq i\})\\
			0&((\Gamma^{\ast}\alpha)(m)<i)
		\end{cases}\\
			&=
		\begin{cases}
			x_{\min\{h|(\Gamma^{\ast}\alpha)(h)\geq i\}}^{[N]}&((\Gamma^{\ast}\alpha)(m)\geq i)\\
			0&((\Gamma^{\ast}\alpha)(m)<i)
		\end{cases}\\
			&=
		\incl{\Gamma}(\Gamma^{\ast}\alpha)^{\ast}(x_{i}^{[N]}).
	\end{align*}
\end{proof}
\begin{lem}\label{Stb Lem2}
	Let $\mathfrak{g}$ be a connected $L_{\infty}$-algebra and $\omega$ be an $n$-simplex of $[\simplex{}{r},\qdR{\dq}{}{}{\mathbb{U}_{\infty}\mathfrak{g}}^{\wedge}]$. Then, if $\omega$ is non-degenerate, the $i$th face $d_{i}\omega$ is also non-degenerate for each integer $i=0,\dots,n$.
\end{lem}
\begin{proof}
	Proposition \ref{Stb Lem1} gives the following commutative diagram for each order-preserving map $\alpha\colon[m]\to[n]$:
	\begin{center}
	\begin{tikzpicture}[auto]
		\node (11) at (0, 1) {$[\simplex{}{r},\qdR{\dq}{}{}{\mathbb{U}_{\infty}\mathfrak{g}}^{\wedge}]_{n}$};
		\node (12) at (3, 1) {$\prod_{\Gamma}\qdR{\dq}{}{n+r}{\mathbb{U}_{\infty}\mathfrak{g}}^{\wedge}$};
		\node (13) at (7, 1) {$\prod_{\Gamma}\qdR{\dq}{}{n,r}{\mathbb{U}_{\infty}\mathfrak{g}}^{\wedge}$};
		\node (14) at (11, 1) {$\prod_{\Gamma}\qdR{\dq}{}{n+r}{\mathbb{U}_{\infty}\mathfrak{g}}^{\wedge}$};
		\node (21) at (0, 0) {$[\simplex{}{r},\qdR{\dq}{}{}{\mathbb{U}_{\infty}\mathfrak{g}}^{\wedge}]_{m}$};
		\node (22) at (3, 0) {$\prod_{\Gamma}\qdR{\dq}{}{m+r}{\mathbb{U}_{\infty}\mathfrak{g}}^{\wedge}$};
		\node (23) at (7, 0) {$\prod_{\Gamma}\qdR{\dq}{}{m,r}{\mathbb{U}_{\infty}\mathfrak{g}}^{\wedge}$};
		\node (24) at (11, 0) {$\prod_{\Gamma}\qdR{\dq}{}{m+r}{\mathbb{U}_{\infty}\mathfrak{g}}^{\wedge}$};
		\path[draw, ->] (11) --node[swap] {$\scriptstyle \alpha^{\ast}$} (21);
		\path[draw, ->] (12) --node[swap] {$\scriptstyle \alpha^{\ast}$} (22);
		\path[draw, ->] (13) --node {$\scriptstyle \alpha^{\ast}$} (23);
		\path[draw, ->] (14) --node {$\scriptstyle \alpha^{\ast}$} (24);
		\path[draw, right hook->] (11) -- (12);
		\path[draw, right hook->] (12) --node {$\scriptstyle \prod_{\Gamma}\incl{\Gamma}$} (13);
		\path[draw, ->>] (13) --node {$\scriptstyle \prod_{\Gamma}\mathrm{re}_{\Gamma}$} (14);
		\path[draw, right hook->] (21) -- (22);
		\path[draw, right hook->] (22) --node[swap] {$\scriptstyle \prod_{\Gamma}\incl{\Gamma}$} (23);
		\path[draw, ->>] (23) --node[swap] {$\scriptstyle \prod_{\Gamma}\mathrm{re}_{\Gamma}$} (24);
	\end{tikzpicture}.
	\end{center}
	Let $\omega$ be an $n$-simplices of $[\simplex{}{r},\qdR{\dq}{}{}{\mathbb{U}_{\infty}\mathfrak{g}}^{\wedge}]$ and $i$ be a non-negative integer satisfying $0\leq i\leq n$. In addition, assume that there exists a pair of an $(n-2)$-simplices $\tilde{\omega}\in[\simplex{}{r},\qdR{\dq}{}{}{\mathbb{U}_{\infty}\mathfrak{g}}^{\wedge}]_{n-2}$ and a non-negative integer satisfying $0\leq i\leq n-2$ which satisfies $d_{i}\omega=s_{j}\tilde{\omega}$. Since
	\begin{align*}
		\min\{l|\delta_{i}(l)\geq h\}
			&=
		\begin{cases}
			h&(h\leq i)\\
			h-1&(h>i)
		\end{cases},\\
		\min\{l|\delta_{i}(l)\geq h\}
			&=
		\begin{cases}
			h&(h\leq i)\\
			h+1&(h>i)
		\end{cases}\\
			&\neq
		j+1
	\end{align*}
	hold,
	\begin{align*}
		\{\bs{h}^{[N]}|\delta_{i}^{\ast}(\bs{h}^{[N]})\in\mathrm{Im}\sigma_{j}^{\ast}\}\cup\{\df{}\bs{h}|\delta_{i}^{\ast}(\df{}\bs{h})\in\mathrm{Im}\sigma_{j}^{\ast}\}
			&=
		\{\bs{h}^{[N]}|\delta_{i}^{\ast}(\bs{h}^{[N]})\neq\bs{j+1}^{[N]}\}\cup\{\df{}\bs{h}|\delta_{i}^{\ast}(\df{}\bs{h})\neq\df{}\bs{j+1}\}\\
			&=
		\begin{cases}
			\{\bs{h}^{[N]},\df{}\bs{h}|h\neq j+1\}&(j+1<i)\\
			\{\bs{h}^{[N]},\df{}\bs{h}|h\neq j+1,j+2\}&(j+1=i)\\
			\{\bs{h}^{[N]},\df{}\bs{h}|h\neq j+2\}&(j+1>i)
		\end{cases}
	\end{align*}
	hold. Thus
	\begin{align*}
		(\incl{\Gamma}(\omega))_{\Gamma}
			\in
		\begin{cases}
			\prod_{\Gamma}\underset{p+\bullet=q}{\prod}\mathbb{U}_{\infty}\mathfrak{g}_{p}\otimes\mathsf{Sym}(\langle\df{}\bs{1},\dots,\check{\df{}\bs{j+1}},\dots,\df{}\bs{n},\df{}\fs{1},\dots,\df{}\fs{r}\rangle_{\Z\langle\dq,\bs{1},\dots,\bs{n},\fs{1},\dots,\check{\bs{j+1}},\dots,\fs{r}\rangle}[1])^{q}&(j+1\leq i)\\
			\prod_{\Gamma}\underset{p+\bullet=q}{\prod}\mathbb{U}_{\infty}\mathfrak{g}_{p}\otimes\mathsf{Sym}(\langle\df{}\bs{1},\dots,\check{\df{}\bs{j+2}},\dots,\df{}\bs{n},\df{}\fs{1},\dots,\df{}\fs{r}\rangle_{\Z\langle\dq,\bs{1},\dots,\bs{n},\fs{1},\dots,\check{\bs{j+2}},\dots,\fs{r}\rangle}[1])^{q}&(j+1\geq i)
		\end{cases}
	\end{align*}
	holds. Hence, if we define
	\begin{align*}
		h
			\coloneqq
		\begin{cases}
			j&(j+1\leq i)\\
			j+1&(j+1\geq i)
		\end{cases},
	\end{align*}
	then
	$$s_{h}d_{h}(\omega)=(\prod_{\Gamma}\mathrm{re}_{\Gamma})(\prod_{\Gamma}\incl{\Gamma})(s_{h}d_{h}\omega)=(\prod_{\Gamma}\mathrm{re}_{\Gamma})\sigma_{h}^{\ast}\delta_{h}^{\ast}(\prod_{\Gamma}\incl{\Gamma})(\omega)=(\prod_{\Gamma}\mathrm{re}_{\Gamma})(\prod_{\Gamma}\incl{\Gamma})(\omega)=\omega$$
	follows.
\end{proof}
\subsection{Fiberwise Integration}
Let $\mathfrak{g}$ be a connected $L_{\infty}$-algebra, $\Gamma\colon[p]\hookrightarrow[n]\times[r]$ be a global chain and $\omega\in\qdR{\dq}{}{p}{\mathbb{U}_{\infty}\mathfrak{g}}^{\wedge}$ be a $\mathfrak{g}$-valued formal differential form. Then there exists an essentially unique decomposition
\begin{align*}
	\omega=\sum_{i}\omega_{\Gamma,i,\fs{}}\wedge\chain{\Gamma}{\bs{}}^{\ast}\omega_{\Gamma,i,\bs{}}
\end{align*}
where $\omega_{\Gamma,i,\bs{}}$ is an element of $\qdR{\dq}{}{n}{\mathbb{U}_{\infty}\mathfrak{g}}^{\wedge}$ and $\omega_{\Gamma,i,\fs{}}$ is an element of $\qdR{\dq}{}{p}{\mathbb{U}_{\infty}\mathfrak{g}}^{\wedge}$ which  does not contain
$$x_{\bs{\Gamma}(1)}^{[N_{1}]},\dots,x_{\bs{\Gamma}(n)}^{[N_{n}]},\df{}x_{\bs{\Gamma}(1)},\dots,\df{}x_{\bs{\Gamma}(n)}.$$
In addition, there is a unique decomposition
\begin{align*}
	\omega_{\Gamma,i,\fs{}}=\omega_{\Gamma,i,\fs{}}^{(p-(n+1))}+\dots+\omega_{\Gamma,i,\fs{}}^{(0)}
\end{align*}
where $\omega_{\Gamma,i,\fs{}}^{(j)}$ is an element of $\prod_{\bullet}\mathbb{U}_{\infty}\mathfrak{g}_{j-\bullet}\otimes\qdR{\dq}{j}{p}{}$. Especially there is a decomposition
\begin{align*}
	\omega_{\Gamma,i,\fs{}}^{(p-(n+1))}=\sum_{\lambda=0}^{\infty}\sum_{j}g_{\Gamma,i,\lambda,j}\otimes f_{\Gamma,i,\lambda,j}\df{}x_{\Fs{\Gamma}(1)}\wedge\dots\wedge\df{}x_{\Fs{\Gamma}(p-(n+1))}
\end{align*}
where $g_{\Gamma,i,\lambda,j}\in\mathbb{U}_{\infty}\mathfrak{g}_{\lambda}$ and $f_{\Gamma,i,\lambda,j}\in\qdR{\dq}{0}{p}{}$. Using this essentially unique representation, we obtain the following (where $x_{0}\coloneqq\dq$ and $x_{n+r+1}\coloneqq0$):
\begin{align*}
	\int_{\simplex{}{r}\cact{\Gamma}{}}\omega
		\coloneqq
	\sum_{\lambda=0}^{\infty}\sum_{i,j}g_{\Gamma,i,\lambda,j}\otimes\bs{\Gamma}^{\ast}(\int_{x_{\Fs{\Gamma}(p-(n+1))+1}}^{x_{\us{\Gamma}(p-(n+1))}}\dots\int_{x_{\Fs{\Gamma}(1)+1}}^{x_{\us{\Gamma}(1)}}f_{\Gamma,i,\lambda,j}\df{}x_{\Fs{\Gamma}(1)}\cdots\df{}x_{\Fs{\Gamma}(p-(n+1))})\omega_{\Gamma,i,\bs{}}
\end{align*}
It converges at $\qdR{\dq}{}{n}{\mathbb{U}_{\infty}\mathfrak{g}}^{\wedge}$.

Let $\omega\colon\simplex{}{n}\times\simplex{}{r}\to\qdR{\dq}{}{}{\mathbb{U}_{\infty}\mathfrak{g}}^{\wedge}$ be a formal differential form with values in a connected $L_{\infty}$-algebra $\mathfrak{g}$ on $\simplex{}{n}\times\simplex{}{r}$. It gives an $n$-simplex $\omega^{\wedge}$ of $[\simplex{}{r},\qdR{\dq}{}{}{\mathbb{U}_{\infty}\mathfrak{g}}^{\wedge}]$. Hence, by the Eilenberg-Zilber lemma, we obtain a unique decomposition $\omega=(\sigma\times\id{})^{\ast}\tilde{\omega}$ where $\sigma\colon[n]\to[m]$ is a surjection and $\tilde{\omega}^{\wedge}$ is a non-degenerate $m$-simplex of $[\simplex{}{r},\qdR{\dq}{}{}{\mathbb{U}_{\infty}\mathfrak{g}}^{\wedge}]$. Using this unique decomposition, we define as
\begin{align*}
	\fint{\proj{\simplex{}{n}}}\omega\coloneqq\sum_{\Gamma\colon[n+r]\hookrightarrow[n]\times[r]}\sigma^{\ast}(\int_{\simplex{}{r}\cact{\Gamma}{}}\Gamma^{\ast}\tilde{\omega}).
\end{align*}
\begin{lem}\label{Int Lem21}
	let $\mathfrak{g}$ be a connected $L_{\infty}$-algebra. Then, for each $\mathfrak{g}$-valued differential form $\omega\colon\simplex{}{n}\times\simplex{}{r}\to\qdR{\dq}{}{}{\mathbb{U}_{\infty}\mathfrak{g}}^{\wedge}$ and order-preserving map $\alpha\colon[m]\to[n]$, the following holds:
	\begin{align*}
		\alpha^{\ast}\fint{\proj{\simplex{}{n}}}\omega=\fint{\proj{\simplex{}{m}}}((\alpha\times\id{})^{\ast}\omega).
	\end{align*}
\end{lem}
\begin{proof}
	From the Eilenberg-Zilber lemma, we obtain decompositions
	\begin{align*}
		\omega
			&=
		(\sigma\times\id{})^{\ast}\tilde{\omega},\\
		\alpha
			&=
		\delta_{\alpha}\sigma_{\alpha}
	\end{align*}
	where $\tilde{\omega}^{\wedge}\in[\simplex{}{r},\qdR{\dq}{}{}{\mathbb{U}_{\infty}\mathfrak{g}}^{\wedge}]_{p}$ is a non-degenerate $p$-simplex, $\sigma$ and $\sigma_{\alpha}$ are surjective and $\delta_{\alpha}$ is injective. In addition, there is a unique decomposition $\sigma\delta_{\alpha}=\delta_{\sigma\delta_{\alpha}}\sigma_{\sigma\delta_{\alpha}}$ where $\sigma_{\sigma\delta_{\alpha}}$ is surjective and $\delta_{\sigma\delta_{\alpha}}$ is injective. For these decompositions,
	\begin{align*}
		(\alpha\times\id{})^{\ast}\omega=(\sigma_{\alpha}\times\id{})^{\ast}(\delta_{\alpha}\times\id{})^{\ast}(\sigma\times\id{})^{\ast}\tilde{\omega}=(\sigma_{\alpha}\times\id{})^{\ast}((\sigma\delta_{\alpha})\times\id{})^{\ast}\tilde{\omega}=((\sigma_{\sigma\delta_{\alpha}}\sigma_{\alpha})\times\id{})^{\ast}(\delta_{\sigma\delta_{\alpha}}\times\id{})^{\ast}\tilde{\omega}
	\end{align*}
	holds. Especially $((\delta_{\sigma\delta_{\alpha}}\times\id{})^{\ast}\tilde{\omega})^{\wedge}$ is non-degenerate.
	\begin{center}
	\begin{tikzpicture}[auto]
		\node (11) at (0, 1.2) {$\simplex{}{m}\times\simplex{}{r}$};
		\node (12) at (3, 1.2) {$\simplex{}{n}\times\simplex{}{r}$};
		\node (13) at (6, 1.2) {$\qdR{\dq}{}{}{\mathbb{U}_{\infty}\mathfrak{g}}^{\wedge}$};
		\node (21) at (0, 0) {$\simplex{}{l}\times\simplex{}{r}$};
		\node (22) at (3, 0) {$\simplex{}{q}\times\simplex{}{r}$};
		\node (23) at (6, 0) {$\simplex{}{p}\times\simplex{}{r}$};
		\path[draw, ->] (11) --node {$\scriptstyle \alpha\times\id{}$} (12);
		\path[draw, ->] (12) --node {$\scriptstyle \omega$} (13);
		\path[draw, ->] (21) --node[swap] {$\scriptstyle \sigma_{\sigma\delta_{\alpha}}\times\id{}$} (22);
		\path[draw, ->] (22) --node[swap] {$\scriptstyle \delta_{\sigma\delta_{\alpha}}\times\id{}$} (23);
		\path[draw, ->] (11) --node[swap] {$\scriptstyle \sigma_{\alpha}\times\id{}$} (21);
		\path[draw, ->] (21) --node {$\scriptstyle \delta_{\alpha}\times\id{}$} (12);
		\path[draw, ->] (12) --node {$\scriptstyle \sigma\times\id{}$} (23);
		\path[draw, ->] (23) --node[swap] {$\scriptstyle \tilde{\omega}$} (13);
	\end{tikzpicture}
	\end{center}
	Since $\alpha^{\ast}$ is a ring morphism,
	\begin{align*}
		\alpha\fint{\simplex{}{n}}\omega
			=
		\sum_{\Gamma\colon[p+r]\hookrightarrow[p]\times[r]}(\sigma_{\sigma\delta_{\alpha}}\sigma_{\alpha})^{\ast}
		\delta_{\sigma\delta_{\alpha}}^{\ast}(\int_{\simplex{}{r}\cact{\Gamma}{}}\Gamma^{\ast}\tilde{\omega})
		\end{align*}
	holds. Fix a maximal chain $\Gamma\colon[p+r]\hookrightarrow[p]\times[r]$ and consider the pullback diagram
	\begin{center}
	\begin{tikzpicture}[auto]
		\node (11) at (0, 1.2) {$[h]$};
		\node (12) at (3, 1.2) {$[p+r]$};
		\node (21) at (0, 0) {$[q]\times[r]$};
		\node (22) at (3, 0) {$[p]\times[r]$};
		\path[draw, ->] (11) --node {$\scriptstyle \beta$} (12);
		\path[draw, right hook->] (11) --node[swap] {$\scriptstyle \tilde{\Gamma}$} (21);
		\path[draw, ->] (21) --node[swap] {$\scriptstyle \delta_{\sigma\delta_{\alpha}}\times\id{}$} (22);
		\path[draw, right hook->] (12) --node {$\scriptstyle \Gamma$} (22);
		\path[draw] (0.2, 0.6) -- (0.6, 0.6) -- (0.6, 1);
	\end{tikzpicture}
	\end{center}
	We can assume that there is a following decomposition:
	\begin{align*}
		\Gamma^{\ast}\tilde{\omega}
			=
	\sum_{\lambda=0}^{\infty}\sum_{i,j}g_{i,\lambda,j}\otimes \chain{\Gamma}{\fs{}}^{\ast}(f_{i,\lambda,j}\df{}x_{1}\wedge\dots\wedge\df{}x_{r})\wedge\chain{\Gamma}{\bs{}}^{\ast}\omega_{i,\bs{}}+\omega_{\text{others}}
	\end{align*}
	Then the following follows from Proposition \ref{Gl Lem1}:
	\begin{align*}
		\delta_{\sigma\delta_{\alpha}}^{\ast}(\int_{\simplex{}{r}\cact{\Gamma}{}}\Gamma^{\ast}\tilde{\omega})
			&=
		\delta_{\sigma\delta_{\alpha}}^{\ast}(\sum_{\lambda=0}^{\infty}\sum_{i,j}g_{i,\lambda,j}\otimes\bs{\Gamma}^{\ast}(\int_{x_{\fs{\Gamma}(r)+1}}^{x_{\us{\Gamma}(r)}}\dots\int_{x_{\fs{\Gamma}(1)+1}}^{x_{\us{\Gamma}(1)}}(\chain{\Gamma}{\fs{}}^{\ast}f_{i,\lambda,j})\df{}x_{\fs{\Gamma}(1)}\wedge\dots\wedge\df{}x_{\fs{\Gamma}(r)})\omega_{i,\bs{}})\\
			&=
		\sum_{\lambda=0}^{\infty}\sum_{i,j}g_{i,\lambda,j}\otimes\delta_{\sigma\delta_{\alpha}}^{\ast}\bs{\Gamma}^{\ast}(\int_{x_{\fs{\Gamma}(r)+1}}^{x_{\us{\Gamma}(r)}}\dots\int_{x_{\fs{\Gamma}(1)+1}}^{x_{\us{\Gamma}(1)}}(\chain{\Gamma}{\fs{}}^{\ast}f_{i,\lambda,j})\df{}x_{\fs{\Gamma}(1)}\wedge\dots\wedge\df{}x_{\fs{\Gamma}(r)})(\delta_{\sigma\delta_{\alpha}}^{\ast}\omega_{i,\bs{}})\\
			&=
		\sum_{\lambda=0}^{\infty}\sum_{i,j}g_{i,\lambda,j}\otimes\bs{\tilde{\Gamma}}^{\ast}\beta^{\ast}(\int_{x_{\fs{\Gamma}(r)+1}}^{x_{\us{\Gamma}(r)}}\dots\int_{x_{\fs{\Gamma}(1)+1}}^{x_{\us{\Gamma}(1)}}(\chain{\Gamma}{\fs{}}^{\ast}f_{i,\lambda,j})\df{}x_{\fs{\Gamma}(1)}\wedge\dots\wedge\df{}x_{\fs{\Gamma}(r)})(\delta_{\sigma\delta_{\alpha}}^{\ast}\omega_{i,\bs{}}).
	\end{align*}
	In the case of $h<q+r$, there exists a pair $(l_{1},l_{2})$ of element satisfying $l_{1}\not\in\mathrm{Im}\delta_{\sigma\delta_{\alpha}}$ and $\bs{\Gamma}(l_{1})+1=\fs{\Gamma}(l_{2})$ from Proposition \ref{Gl Lem5}. Then
	\begin{align*}
		\min\{l\in[h]|\beta(l)\geq\fs{\Gamma}(l_{2})+1\}
			=
		\min\{l\in[h]|\beta(l)\geq\fs{\Gamma}(l_{2})\}
			=
		\min\{l\in[h]|\beta(l)\geq\us{\Gamma}(l_{2})\}
	\end{align*}
	holds, thus
	\begin{align*}
		\beta^{\ast}(\int_{x_{\fs{\Gamma}(r)+1}}^{x_{\us{\Gamma}(r)}}\dots\int_{x_{\fs{\Gamma}(1)+1}}^{x_{\us{\Gamma}(1)}}(\chain{\Gamma}{\fs{}}^{\ast}f_{i,\lambda,j})\df{}x_{\fs{\Gamma}(1)}\wedge\dots\wedge\df{}x_{\fs{\Gamma}(r)})=0
	\end{align*}
	holds. In the case of $h=q+r$, it follows from Proposition \ref{Gl Lem4} and \ref{Gl Lem4} that $\beta^{\ast}$ preserves ``integral range'', from  Proposition \ref{Gl Lem1} that $\beta^{\ast}$ preserves  ``base direction'', and from Proposition \ref{Gl Lem2} that $\beta^{\ast}$ preserves ``fiber direction''. Therefore the following holds:
	\begin{align*}
		\beta^{\ast}(\int_{x_{\fs{\Gamma}(r)+1}}^{x_{\us{\Gamma}(r)}}\dots\int_{x_{\fs{\Gamma}(1)+1}}^{x_{\us{\Gamma}(1)}}(\chain{\Gamma}{\fs{}}^{\ast}f_{i,\lambda,j})\df{}x_{\fs{\Gamma}(1)}\wedge\dots\wedge\df{}x_{\fs{\Gamma}(r)})
			=
		\int_{x_{\fs{\tilde{\Gamma}}(r)+1}}^{x_{\us{\tilde{\Gamma}}(r)}}\dots\int_{x_{\fs{\tilde{\Gamma}}(1)+1}}^{x_{\us{\tilde{\Gamma}}(1)}}(\chain{\tilde{\Gamma}}{\fs{}}^{\ast}f_{i,\lambda,j})\df{}x_{\fs{\tilde{\Gamma}}(1)}\wedge\dots\wedge\df{}x_{\fs{\tilde{\Gamma}}(r)}.
	\end{align*}
	On the other hand,
	\begin{align*}
		\tilde{\Gamma}^{\ast}(\delta_{\sigma\delta_{\alpha}}\times\id{})^{\ast}\tilde{\omega}
			&=
		\beta^{\ast}\Gamma^{\ast}\tilde{\omega}\\
			&=
		\beta^{\ast}(\sum_{\lambda=0}^{\infty}\sum_{i,j}g_{i,\lambda,j}\otimes \chain{\Gamma}{\fs{}}^{\ast}(f_{i,\lambda,j}\df{}x_{1}\wedge\dots\wedge\df{}x_{r})\wedge\chain{\Gamma}{\bs{}}^{\ast}\omega_{i,\bs{}}+\omega_{\text{others}})\\
			&=
		\sum_{\lambda=0}^{\infty}\sum_{i,j}g_{i,\lambda,j}\otimes \chain{(\Gamma\beta)}{\fs{}}^{\ast}(f_{i,\lambda,j}\df{}x_{1}\wedge\dots\wedge\df{}x_{r})\wedge\chain{(\Gamma\beta)}{\bs{}}^{\ast}\omega_{i,\bs{}}+\beta^{\ast}\omega_{\text{others}}\\
			&=
		\sum_{\lambda=0}^{\infty}\sum_{i,j}g_{i,\lambda,j}\otimes \chain{\tilde{\Gamma}}{\fs{}}^{\ast}(f_{i,\lambda,j}\df{}x_{1}\wedge\dots\wedge\df{}x_{r})\wedge\chain{\tilde{\Gamma}}{\bs{}}^{\ast}\delta_{\sigma\delta_{\alpha}}^{\ast}\omega_{i,\bs{}}+\beta^{\ast}\omega_{\text{others}}
	\end{align*}
	holds. Thus
	\begin{align*}
		\delta_{\sigma\delta_{\alpha}}^{\ast}(\int_{\simplex{}{r}\cact{\Gamma}{}}\Gamma^{\ast}\tilde{\omega})
			=
		\int_{\simplex{}{r}\cact{\tilde{\Gamma}}{}}\tilde{\Gamma}^{\ast}\delta_{\sigma\delta_{\alpha}}^{\ast}(\delta_{\sigma\delta_{\alpha}}\times\id{})^{\ast}\tilde{\omega}
	\end{align*}
	holds. Hence
	\begin{align*}
		\alpha^{\ast}\fint{\simplex{}{n}}\omega
			&=
		\sum_{\Gamma\colon[p+r]\hookrightarrow[p]\times[r]}(\sigma_{\sigma\delta_{\alpha}}\sigma_{\alpha})^{\ast}\delta_{\sigma\delta_{\alpha}}^{\ast}(\int_{\simplex{}{r}\cact{\Gamma}{}}\Gamma^{\ast}\tilde{\omega})\\
			&=
		\sum_{\Gamma\colon[q+r]\hookrightarrow[q]\times[r]}(\sigma_{\sigma\delta_{\alpha}}\sigma_{\alpha})^{\ast}(\int_{\simplex{}{r}\cact{\Gamma}{}}\Gamma^{\ast}(\delta_{\sigma\delta_{\alpha}}\times\id{})^{\ast}\tilde{\omega})\\
			&=
		\fint{\simplex{}{n}}((\alpha\times\id{})^{\ast}\omega)
	\end{align*}
	holds from Lemma \ref{Gl Lem7}.
\end{proof}
\begin{lem}\label{Int Lem22}
	Let $\mathfrak{g}$ be a connected $L_{\infty}$-algebra. Then, for each $\mathfrak{g}$-valued differential form $\omega\colon\simplex{}{n}\times\simplex{}{r}\to\qdR{\dq}{}{}{\mathbb{U}_{\infty}\mathfrak{g}}^{\wedge}$ and surjective order-preserving map $\sigma_{h}\colon[r+1]\to[r]$, $\fint{\proj{\simplex{}{m}}}((\id{}\times\sigma_{h})^{\ast}\omega)=0$ holds.
\end{lem}
\begin{proof}
	For each maximal chain $\Gamma\colon[n+r]\to[n]\times[r]$, there exists a (unique) maximal chain ${\sigma_{h}}_{\ast}\Gamma$ satisfying $(\id{}\times\sigma_{h})\Gamma=({\sigma_{h}}_{\ast}\Gamma)\sigma_{\fs{\Gamma}(h+1)-1}$ from Lemma \ref{Gl Lem7}. Since
	\begin{align*}
		\sigma_{\fs{\Gamma}(h+1)-1}x_{i}
			&=
		\begin{cases}
			x_{i}&(i\leq\fs{\Gamma}(h+1)-1)\\
			x_{i+1}&(i>\fs{\Gamma}(h+1)-1)
		\end{cases}\\
			&\neq
		x_{\fs{\Gamma}(h+1)}
	\end{align*}
	holds, $\int_{\simplex{}{r}\cact{\Gamma}{}}\Gamma^{\ast}\tilde{\omega}=0$ holds by definition.
\end{proof}
\begin{dfn}
	We say that a formal differential form with values in a connected $L_{\infty}$-algebra on a simplicial set $X\times U$ has
	\Def{the finite support in the direction along the projection $\proj{X}\colon X\times U\to X$}
	if the set
	\begin{align*}
		\mathrm{supp}_{\proj{X}}(\omega)\coloneqq\bigcup_{r}\{u\in U_{r}|((\id{}\times u)^{\ast}\omega)^{\wedge}\in[X,\qdR{\dq}{}{}{\mathbb{U}_{\infty}\mathfrak{g}}^{\wedge}]_{r}\text{ is non-degenerate.}\}
	\end{align*}
	is a finite set.
\end{dfn}
We define an order on the set $\mathrm{supp}_{\proj{X}}(\omega)$ as follows:
\begin{align*}
	u_{1}\leq u_{2}\text{ iff }u_{1}=\delta^{\ast}u_{2}\text{ for some order-preserving map $\delta\colon[r_{1}]\to[r_{2}]$}.
\end{align*}
If $\mathrm{supp}_{\proj{X}}(\omega)$ is finite, we can consider a set of maximal elements of $\mathrm{supp}_{\proj{X}}(\omega)$. We denote the set as $\mathrm{part}_{\omega}(U)$.

Let $\omega\colon X\times U\to\qdR{\dq}{}{}{\mathbb{U}_{\infty}\mathfrak{g}}^{\wedge}$ be a formal differential form with values in a connected $L_{\infty}$-algebra $\mathfrak{g}$ on $X\times U$. Each simplices $u\in U_{r}$ determine a $\mathfrak{g}$-valued formal differential form $(\id{}\times u)^{\ast}\omega$. From Lemma \ref{Int Lem21}, we obtain a cocone
\begin{align*}
	\fint{\proj{\simplex{}{\bullet}}}((-\times u)^{\ast}\omega)\colon\simplex{}{\bullet}\to\qdR{\dq}{}{}{\mathbb{U}_{\infty}\mathfrak{g}}^{\wedge}
\end{align*}
and obtain a $\mathfrak{g}$-valued formal differential form $\fint{\proj{\simplex{}{\bullet}}}((\id{}\times u)^{\ast}\omega)\colon X\to\qdR{\dq}{}{}{\mathbb{U}_{\infty}\mathfrak{g}}^{\wedge}$. For each simplices $u\in U_{r}$ which the simplex $((\id{}\times u)^{\ast}\omega)^{\wedge}$ is degenerate (as a simplex $[X,\qdR{\dq}{}{}{\mathbb{U}_{\infty}\mathfrak{g}}^{\wedge}]$), $\fint{\proj{X}}((\id{}\times u)^{\ast}\omega)=0$ holds since $\fint{\proj{\simplex{}{n}}}((x\times u)^{\ast}\omega)=0$ holds for any simpleis $x\in X_{n}$ from Lemma \ref{Int Lem22}.
\begin{dfn}(simplicial integration)
	Let $\omega\colon X\times U\to\qdR{\dq}{}{}{\mathbb{U}_{\infty}\mathfrak{g}}^{\wedge}$ be a formal differential form with values in a connected $L_{\infty}$-algebra $\mathfrak{g}$ on $X\times U$ which has the finite support in the direction along the projection $\proj{X}$. The $\mathfrak{g}$-valued formal differential form on $X$
	\begin{align*}
		\fint{\proj{X}}\omega\coloneqq\sum_{u\in\mathrm{part}_{\omega}(U)}\fint{\proj{X}}((\id{}\times u)^{\ast}\omega)
	\end{align*}
	is called \Def{the fiberwise integration of $\omega$ along the projection $\proj{X}\colon X\times U\to X$}.
\end{dfn}
\subsection{Stokes's theorem}
One of the important theorems for integrals on smooth manifolds is Stokes's theorem. This is a theorem that connects the integration of closed form with the integration on the boundary, and it follows that the integration gives a chain map from the de Rham complex to the singular cochain complex. We would like to consider this analogy for fiberwise integration on simplicial sets, but roughly speaking, the following obstacles exist:
\begin{itemize}
	\item The boundary of simplicial set $U$ is unknown in general.
	\item For example, the boundary of standard $2$-simplex $\simplex{}{2}$ is already known as $\boundary{2}$, but the integration $\fint{\proj{X}}(\omega|_{X\times\boundary{2}})$ does not coincide with what we seek.
\end{itemize}
The second problem is considered to be caused by the fact that, unlike the case of smooth manifolds, orientation is not taken into account. In light of simplicial homology, it is presumed that it is suitable to consider the linear combination $\sum_{i=0}^{n}(-1)^{i}\Delta\{0,\dots,\check{i},\dots,n\}$ as ``the boundary of standard $n$-simplex with orientation taken into account''. Since fiberwise integration on a simplicial set is the sum of integration on each simplex, we can consider the following ``integration''.
\begin{dfn}
	Let $\omega\colon X\times U\to\qdR{\dq}{}{}{\mathbb{U}_{\infty}\mathfrak{g}}^{\wedge}$ be a formal differential form with values in a connected $L_{\infty}$-algebra $\mathfrak{g}$ on $X\times U$ which has the finite support in the direction along the projection $\proj{X}$. The $\mathfrak{g}$-valued formal differential form on $X$
	\begin{align*}
		\bint{\proj{X}}\omega\coloneqq\sum_{(\simplex{}{r}\xrightarrow{u}U)\in\mathrm{part}_{\omega}(U)}\sum_{i=0}^{r}\fint{\proj{X}}((\id{}\times u\delta_{i})^{\ast}\omega)
	\end{align*}
	is called \Def{the boundary fiberwise integration of $\omega$ along the projection $\proj{X}\colon X\times U\to X$}.
\end{dfn}
\begin{lem}\label{Stokes Lem}
	Let $\Gamma\colon[n+r]\hookrightarrow[n]\times[r]$ be a maximal chain. For any pair of integer $1\leq j\leq\mathsf{n}_{\Gamma}$ and $0\leq i\leq r_{j}$, denote
	\begin{align*}
		\Rs{\Gamma}{j}{i}\coloneqq r_{1}+\dots+r_{j-1}+i.
	\end{align*}
	In addition, $(\omega_{\Gamma,\fs{}}^{(r-1)},\omega_{\Gamma,\bs{}})\in(\qdR{\dq}{}{n}{\mathbb{U}_{\infty}\mathfrak{g}}^{\wedge})\times(\prod_{\bullet}\mathbb{U}_{\infty}\mathfrak{g}_{(r-1)-\bullet}\otimes\qdR{\dq}{r-1}{r}{})$ be a pair of $\mathfrak{g}$-valued formal differential forms. Then
	\begin{align*}
		&\int_{\simplex{}{r}\cact{\Gamma}{}}((\chain{\Gamma}{\fs{}}^{\ast}\df{}\omega_{\Gamma,\fs{}}^{(r-1)})\wedge(\chain{\Gamma}{\bs{}}^{\ast}\omega_{\Gamma,\bs{}}))\\
			=
		&\sum_{1\leq j\leq\mathsf{n}_{\Gamma}}(-1)^{\Rs{\Gamma}{j}{0}}\int_{\simplex{}{r}\cact{\Gamma\delta_{\vs{\Gamma}{j}{0}}}{}}(\Gamma\delta_{\vs{\Gamma}{j}{0}})^{\ast}((\proj{\simplex{}{r}}^{\ast}\omega_{\Gamma,\fs{}}^{(r-1)})\wedge(\proj{\simplex{}{n}}^{\ast}\omega_{\Gamma,\bs{}}))\\
			&+
		\underset{0<i<r_{j}}{\sum_{1\leq j\leq\mathsf{n}_{\Gamma}}}(-1)^{\Rs{\Gamma}{j}{i}}\int_{\simplex{}{r-1}\cact{\Gamma_{\vs{\Gamma}{j}{i}}}{}}\Gamma_{\vs{\Gamma}{j}{i}}^{\ast}(\id{}\times\delta_{\Rs{\Gamma}{j}{i}})^{\ast}((\proj{\simplex{}{r}}^{\ast}\omega_{\Gamma,\fs{}}^{(r-1)})\wedge(\proj{\simplex{}{n}}^{\ast}\omega_{\Gamma,\bs{}}))\\
			&+
		\sum_{1\leq j\leq\mathsf{n}_{\Gamma}}(-1)^{\Rs{\Gamma}{j}{r_{j}}}\int_{\simplex{}{r}\cact{\Gamma\delta_{\vs{\Gamma}{j}{r_{j}}}}{}}(\Gamma\delta_{\vs{\Gamma}{j}{r_{j}}})^{\ast}((\proj{\simplex{}{r}}^{\ast}\omega_{\Gamma,\fs{}}^{(r-1)})\wedge(\proj{\simplex{}{n}}^{\ast}\omega_{\Gamma,\bs{}}))
	\end{align*}
	holds where $\Gamma_{\vs{\Gamma}{j}{i}}\colon[n+r-1]\hookrightarrow[n-1]\times[r]$ is a maximal chain satisfying the following:
	\begin{align*}
		\Gamma\delta_{\vs{\Gamma}{j}{i}}=(\id{}\times\delta_{\chain{\Gamma}{\fs{}}(\vs{\Gamma}{j}{i})})\Gamma_{\vs{\Gamma}{j}{i}}=(\id{}\times\delta_{\Rs{\Gamma}{j}{i}})\Gamma_{\vs{\Gamma}{j}{i}}.
	\end{align*}
	(Existence and uniqueness of such a maximal chain follow from Prposition \ref{Gl Lem8}.)
\end{lem}
\begin{proof}
	We can assume that
\begin{align*}
	\omega_{\Gamma,\fs{}}^{(r-1)}=\sum_{\lambda=0}^{\infty}\sum_{h=1}^{r}g_{\lambda,h}\otimes f_{\lambda,h}\df{}x_{1}\wedge\cdots\check{\df{}x_{h}}\cdots\wedge\df{}x_{r}.
\end{align*}
Then
	\begin{align*}
		&\int_{\simplex{}{r}\cact{\Gamma}{}}((\chain{\Gamma}{\fs{}}^{\ast}\df{}\omega_{\Gamma,\fs{}}^{(r-1)})\wedge(\chain{\Gamma}{\bs{}}^{\ast}\omega_{\Gamma,\bs{}}))\\
			=
		&\sum_{\lambda=0}^{\infty}\sum_{h=1}^{r}(-1)^{h-1}g_{\lambda,h}\otimes\bs{\Gamma}^{\ast}(\int_{x_{\fs{\Gamma}(r)+1}}^{x_{\us{\Gamma}(r)}}\cdots\int_{x_{\fs{\Gamma}(1)+1}}^{x_{\us{\Gamma}(1)}}\frac{\partial}{\partial x_{\fs{\Gamma}(h)}}(\chain{\Gamma}{\fs{}}^{\ast}f_{\lambda,h})\df{}x_{\fs{\Gamma}(1)}\cdots\df{}x_{\fs{\Gamma}(r)})\omega_{\Gamma,\bs{}}
	\end{align*}
	holds by definition. Define maps $\ep_{h}^{+},\ep_{h}^{-}\colon\{\dq,x_{1},\dots,x_{n+r}\}\to\{\dq,x_{1},\dots,x_{n+r}\}$ as follows for each $h=1,\dots,n+r$
	\begin{align*}
		\ep_{i}^{\pm}(\dq)
			&=
		\dq,\\
		\ep_{i}^{\pm}(x_{h})
			&=
		\begin{cases}
			x_{h}&(h\neq\vs{\Gamma}{j}{i})\\
			x_{\vs{\Gamma}{j}{i\pm1}}&(h=\vs{\Gamma}{j}{i})
		\end{cases}.
	\end{align*}
	Denote the following as $I_{i,j}$ for any pair $(i,j)$ of integers $1\leq j\leq m$ and $1\leq i\leq r_{j}$:
	\begin{align*}
		\int_{x_{\vs{\Gamma}{j-1}{r_{j-1}}+1}}^{x_{\vs{\Gamma}{j-1}{0}}}\cdots\int_{x_{\fs{\Gamma}(1)+1}}^{x_{\us{\Gamma}(1)}}(\chain{\Gamma}{\fs{}}^{\ast}f_{\lambda,\Rs{\Gamma}{j}{i}})\df{}x_{\fs{\Gamma}(1)}\cdots\df{}x_{\vs{\Gamma}{j-1}{r_{j-1}}}.
	\end{align*}
	Since $1\leq\Rs{\Gamma}{j}{i}\leq r$ holds from the definition,
	\begin{align*}
		&\int_{x_{\vs{\Gamma}{j}{r_{j}}+1}}^{x_{\vs{\Gamma}{j}{0}}}\cdots\int_{x_{\fs{\Gamma}(1)+1}}^{x_{\us{\Gamma}(1)}}\frac{\partial}{\partial x_{\vs{\Gamma}{j}{i}}}(\chain{\Gamma}{\fs{}}^{\ast}f_{\lambda,\Rs{\Gamma}{j}{i}})\df{}x_{\fs{\Gamma}(1)}\cdots\df{}x_{\vs{\Gamma}{j}{r_{j}}}\\
			=
		&\int_{x_{\vs{\Gamma}{j}{r_{j}}+1}}^{x_{\vs{\Gamma}{j}{0}}}\cdots\int_{x_{\vs{\Gamma}{j}{1}+1}}^{x_{\vs{\Gamma}{j}{0}}}\frac{\partial}{\partial x_{\vs{\Gamma}{j}{i}}}I_{i,j}\df{}x_{\vs{\Gamma}{j}{1}}\cdots\df{}x_{\vs{\Gamma}{j}{r_{j}}}
	\end{align*}
	holds. In the case of $i=1$,
	\begin{align*}
		&\int_{x_{\vs{\Gamma}{j}{r_{j}}+1}}^{x_{\vs{\Gamma}{j}{0}}}\cdots\int_{x_{\vs{\Gamma}{j}{1}+1}}^{x_{\vs{\Gamma}{j}{0}}}\frac{\partial}{\partial x_{\vs{\Gamma}{j}{1}}}I_{1,j}\df{}x_{\vs{\Gamma}{j}{1}}\cdots\df{}x_{\vs{\Gamma}{j}{r_{j}}}\\
			=
		&\int_{x_{\vs{\Gamma}{j}{r_{j}}+1}}^{x_{\vs{\Gamma}{j}{0}}}\cdots\int_{x_{\vs{\Gamma}{j}{2}+1}}^{x_{\vs{\Gamma}{j}{0}}}(\overline{\ep_{\vs{\Gamma}{j}{1}}^{-}}(I_{1,j})-\overline{\ep_{\vs{\Gamma}{j}{1}}^{+}}(I_{1,j}))\df{}x_{\vs{\Gamma}{j}{2}}\cdots\df{}x_{\vs{\Gamma}{j}{r_{j}}}
	\end{align*}
	follows from lemma \ref{Int Lem11}. Then
	\begin{align*}
		&\sum_{\lambda=0}^{\infty}g_{\lambda,\Rs{\Gamma}{j}{1}}\otimes\bs{\Gamma}^{\ast}(\int_{x_{\fs{\Gamma}(r)+1}}^{x_{\us{\Gamma}(r)}}\cdots\int_{x_{\vs{\Gamma}{j}{2}+1}}^{x_{\vs{\Gamma}{j}{0}}}\overline{\ep_{1}^{-}}(I_{1,j})\df{}x_{\vs{\Gamma}{j}{2}}\cdots\df{}x_{\vs{\Gamma}{j}{r_{j}}})\omega_{\Gamma,\bs{}}\\
			=
		&\int_{\simplex{}{r}\cact{\Gamma\delta_{\vs{\Gamma}{j}{0}}}{}}((\chain{(\Gamma\delta_{\vs{\Gamma}{j}{0}})}{\fs{}}^{\ast}(\sum_{\lambda=0}^{\infty}g_{\lambda,\Rs{\Gamma}{j}{1}}\otimes f_{\lambda,\Rs{\Gamma}{j}{1}}\df{}x_{1}\wedge\dots\wedge\check{\df{}x_{\Rs{\Gamma}{j}{1}}}\wedge\dots\wedge\df{}x_{r}))\wedge(\chain{(\Gamma\delta_{\vs{\Gamma}{j}{0}})}{\fs{}}^{\ast}\omega_{\Gamma,\bs{}}))\\
			=
		&\int_{\simplex{}{r}\cact{\Gamma\delta_{\vs{\Gamma}{j}{0}}}{}}((\chain{(\Gamma\delta_{\vs{\Gamma}{j}{0}})}{\fs{}}^{\ast}\omega_{\Gamma,\fs{}}^{(r-1)})\wedge(\chain{(\Gamma\delta_{\vs{\Gamma}{j}{0}})}{\fs{}}^{\ast}\omega_{\Gamma,\bs{}}))
	\end{align*}
	follows from Proposition \ref{GC Lem5}.
	In the case of $i>1$, denote the following as $J_{i,j}$:
	\begin{align*}
		J_{i,j}
			=
		\int_{x_{\vs{\Gamma}{j}{i-2}+1}}^{x_{\vs{\Gamma}{j}{0}}}\cdots\int_{x_{\vs{\Gamma}{j}{1}+1}}^{x_{\vs{\Gamma}{j}{0}}}I_{i,j}\df{}x_{\vs{\Gamma}{j}{1}}\cdots\df{}x_{\vs{\Gamma}{j}{i-2}}.
	\end{align*}
	Then
	\begin{align*}
		&\int_{x_{\vs{\Gamma}{j}{i}+1}}^{x_{\vs{\Gamma}{j}{0}}}\cdots\int_{x_{\vs{\Gamma}{j}{1}+1}}^{x_{\vs{\Gamma}{j}{0}}}\frac{\partial}{\partial x_{\vs{\Gamma}{j}{i}}}I_{i,j}\df{}x_{\vs{\Gamma}{j}{1}}\cdots\df{}x_{\vs{\Gamma}{j}{r_{j}}}\\
			=
		&\int_{x_{\vs{\Gamma}{j}{i}+1}}^{x_{\vs{\Gamma}{j}{0}}}\int_{x_{\vs{\Gamma}{j}{i-1}+1}}^{x_{\vs{\Gamma}{j}{0}}}\frac{\partial}{\partial x_{\vs{\Gamma}{j}{i}}}J_{i,j}\df{}x_{\vs{\Gamma}{j}{i-1}}\df{}x_{\vs{\Gamma}{j}{i}}\cdots\df{}x_{\vs{\Gamma}{j}{r_{j}}}\\
			=
		&\int_{x_{\vs{\Gamma}{j}{i}+1}}^{x_{\vs{\Gamma}{j}{0}}}\overline{\ep_{\vs{\Gamma}{j}{i-1}}^{+}}(J_{i,j})\df{}x_{\vs{\Gamma}{j}{i}}
			-
		\int_{x_{\vs{\Gamma}{j}{i}+1}}^{x_{\vs{\Gamma}{j}{0}}}\overline{\ep_{\vs{\Gamma}{j}{i}}^{+}}(J_{i,j})\df{}x_{\vs{\Gamma}{j}{i-1}}
	\end{align*}
	follows from Corollary \ref{Int Lem12} and Lemma \ref{Int Lem13}. Then
	\begin{align*}
		&\sum_{\lambda=0}^{\infty}g_{\lambda,\Rs{\Gamma}{j}{r_{j}}}\otimes\bs{\Gamma}^{\ast}(\int_{x_{\fs{\Gamma}(r)+1}}^{x_{\us{\Gamma}(r)}}\cdots\int_{x_{\vs{\Gamma}{j}{r_{j}}+1}}^{x_{\vs{\Gamma}{j}{0}}}\overline{\ep_{\vs{\Gamma}{j}{r_{j}}}^{+}}(J_{r_{j},j})\df{}x_{\vs{\Gamma}{j}{r_{j}-1}}\df{}x_{\vs{\Gamma}{j+1}{1}}\cdots\df{}x_{\vs{\Gamma}{j}{r_{j}}})\omega_{\Gamma,\bs{}}\\
			=
		&\int_{\simplex{}{r}\cact{\Gamma\delta_{\vs{\Gamma}{j}{r_{j}}}}{}}((\chain{(\Gamma\delta_{\vs{\Gamma}{j}{r_{j}}})}{\fs{}}^{\ast}(\sum_{\lambda=0}^{\infty}g_{\lambda,\Rs{\Gamma}{j}{r_{j}}}\otimes f_{\lambda,\Rs{\Gamma}{j}{r_{j}}}\df{}x_{1}\wedge\dots\wedge\check{\df{}x_{\Rs{\Gamma}{j}{r_{j}}}}\wedge\dots\wedge\df{}x_{r}))\wedge(\chain{(\Gamma\delta_{\vs{\Gamma}{j}{r_{j}}})}{\fs{}}^{\ast}\omega_{\Gamma,\bs{}}))\\
			=
		&\int_{\simplex{}{r}\cact{\Gamma\delta_{\vs{\Gamma}{j}{r_{j}}}}{}}((\chain{(\Gamma\delta_{\vs{\Gamma}{j}{r_{j}}})}{\fs{}}^{\ast}\omega_{\Gamma,\fs{}}^{(r-1)})\wedge(\chain{(\Gamma\delta_{\vs{\Gamma}{j}{r_{j}}})}{\fs{}}^{\ast}\omega_{\Gamma,\bs{}}))
	\end{align*}
	follows from Proposition \ref{GC Lem5}. This equation hold in the case of $r_{j}=1$.
	
	For any integer $i$ satisfying $0<i<r_{j}$, $\vs{\Gamma}{j}{i}$ is an element of $\mathsf{Out}(\Gamma)$. Therefore there is a maximal chain $\Gamma_{i,j}\colon[n+r-1]\hookrightarrow[n]\times[r-1]$ satisfying
	\begin{align*}
		\Gamma\delta_{\vs{\Gamma}{j}{i}}=(\id{}\times\delta_{\chain{\Gamma}{\fs{}}(\vs{\Gamma}{j}{i})})\Gamma_{i,j}=(\id{}\times\delta_{\Rs{\Gamma}{j}{i}})\Gamma_{i,j}.
	\end{align*}
	For this maximal chain,
	\begin{align*}
		&\int_{\simplex{}{r-1}\cact{\Gamma_{i,j}}{}}\Gamma_{i,j}^{\ast}(\id{}\times\delta_{\Gamma_{\Rs{\Gamma}{j}{i}}})^{\ast}((\proj{\simplex{}{r}}^{\ast}\omega_{\Gamma,\fs{}}^{(r-1)})\wedge(\proj{\simplex{}{n}}^{\ast}\omega_{\Gamma,\bs{}}^{(r-1)}))\\
			=
		&\sum_{h=1}^{r}\int_{\simplex{}{r-1}\cact{\Gamma_{i,j}}{}}((\chain{(\Gamma_{i,j})}{\fs{}}^{\ast}\delta_{\Rs{\Gamma}{j}{i}}^{\ast}(\sum_{\lambda=0}^{\infty}g_{\lambda,h}\otimes f_{h}\df{}x_{1}\wedge\dots\wedge\check{\df{}x_{h}}\wedge\dots\wedge\df{}x_{r}))\wedge(\chain{(\Gamma_{i,j})}{\bs{}}^{\ast}\omega_{\Gamma,\bs{}}))\\
			=
		&\sum_{a=0,1}\int_{\simplex{}{r-1}\cact{\Gamma_{i,j}}{}}((\chain{(\Gamma_{i,j})}{\fs{}}^{\ast}(\sum_{\lambda=0}^{\infty}g_{\lambda,h}\otimes(\delta_{\Rs{\Gamma}{j}{i}}^{\ast}f_{\Rs{\Gamma}{j}{i}+a})\df{}x_{1}\wedge\dots\wedge\df{}x_{r-1}))\wedge(\chain{(\Gamma_{i,j})}{\bs{}}^{\ast}\omega_{\Gamma,\bs{}}))\\
			=
		&\sum_{a=0,1}\sum_{\lambda=0}^{\infty}g_{\lambda,\Rs{\Gamma}{j}{i}+a}\otimes\bs{\Gamma_{i,j}}^{\ast}(\int_{x_{\fs{\Gamma_{i,j}}(r-1)+1}}^{x_{\us{\Gamma_{i,j}}(r-1)}}\cdots\int_{x_{\fs{\Gamma_{i,j}}(1)+1}}^{x_{\us{\Gamma_{i,j}}(1)}}(\Gamma_{i,j}^{\ast}\delta_{\Rs{\Gamma}{j}{i}}^{\ast}f_{\Rs{\Gamma}{j}{i}+a})\df{}x_{\fs{\Gamma_{i,j}(1)}}\dots\df{}x_{\fs{\Gamma_{i,j}(r-1)}})\omega_{\Gamma,\bs{}}\\
			=
		&\sum_{a=0,1}\sum_{\lambda=0}^{\infty}g_{\lambda,\Rs{\Gamma}{j}{i}+a}\otimes\bs{\Gamma_{i,j}}^{\ast}(\int_{x_{\fs{\Gamma_{i,j}}(r-1)+1}}^{x_{\us{\Gamma_{i,j}}(r-1)}}\cdots\int_{x_{\fs{\Gamma_{i,j}}(1)+1}}^{x_{\us{\Gamma_{i,j}}(1)}}(\delta_{\vs{\Gamma}{j}{i}}^{\ast}\chain{\Gamma}{\fs{}}^{\ast}f_{\Rs{\Gamma}{j}{i}+a})\df{}x_{\fs{\Gamma_{i,j}(1)}}\dots\df{}x_{\fs{\Gamma_{i,j}(r-1)}})\omega_{\Gamma,\bs{}}
	\end{align*}
	holds.
	
	From the above,
	\begin{align*}
		&\int_{\simplex{}{r}\cact{\Gamma}{}}((\chain{\Gamma}{\fs{}}^{\ast}\df{}\omega_{\Gamma,\fs{}}^{(r-1)})\wedge(\chain{\Gamma}{\bs{}}^{\ast}\omega_{\Gamma,\bs{}}))\\
			=
		&\sum_{\lambda=0}^{\infty}\sum_{1\leq j\leq\mathsf{n}_{\Gamma}}(-1)^{\Rs{\Gamma}{j}{0}}g_{\lambda,\Rs{\Gamma}{j}{1}}\otimes\bs{\Gamma}^{\ast}(\text{integration of $\overline{\ep_{\Rs{\Gamma}{j}{1}}^{-}}(I_{1,j})$})\omega_{\Gamma,\bs{}}\\
			+
		&\sum_{\lambda=0}^{\infty}\underset{1=r_{j}}{\sum_{1\leq j\leq\mathsf{n}_{\Gamma}}}(-1)^{\Rs{\Gamma}{j}{1}}g_{\lambda,\Rs{\Gamma}{j}{1}}\otimes\bs{\Gamma}^{\ast}(\text{integration of $\overline{\ep_{\Rs{\Gamma}{j}{1}}^{+}}(I_{1,j})$})\omega_{\Gamma,\bs{}}\\
			+
		&\sum_{\lambda=0}^{\infty}\underset{1<r_{j}}{\sum_{1\leq j\leq\mathsf{n}_{\Gamma}}}(-1)^{\Rs{\Gamma}{j}{1}}g_{\lambda,\Rs{\Gamma}{j}{1}}\otimes\bs{\Gamma}^{\ast}(\text{integration of $\overline{\ep_{\Rs{\Gamma}{j}{1}}^{+}}(I_{1,j})$})\omega_{\Gamma,\bs{}}\\
			+
		&\sum_{\lambda=0}^{\infty}\underset{1\leq i-1<r_{j}}{\sum_{1\leq j\leq\mathsf{n}_{\Gamma}}}(-1)^{\Rs{\Gamma}{j}{i-1}}g_{\lambda,\Rs{\Gamma}{j}{i}}\otimes\bs{\Gamma}^{\ast}(\text{integration of $\overline{\ep_{\Rs{\Gamma}{j}{i-1}}^{+}}(J_{(i-1)+1,j})$})\omega_{\Gamma,\bs{}}\\
			+
		&\sum_{\lambda=0}^{\infty}\underset{1<i<r_{j}}{\sum_{1\leq j\leq\mathsf{n}_{\Gamma}}}(-1)^{\Rs{\Gamma}{j}{i}}g_{\lambda,\Rs{\Gamma}{j}{i}}\otimes\bs{\Gamma}^{\ast}(\text{integration of $\overline{\ep_{\Rs{\Gamma}{j}{i}}^{+}}(J_{i,j})$})\omega_{\Gamma,\bs{}}\\
			+
		&\sum_{\lambda=0}^{\infty}\underset{1<r_{j}}{\sum_{1\leq j\leq\mathsf{n}_{\Gamma}}}(-1)^{\Rs{\Gamma}{j}{r_{j}}}g_{\lambda,\Rs{\Gamma}{j}{r_{j}}}\otimes\bs{\Gamma}^{\ast}(\text{integration of $\overline{\ep_{\Rs{\Gamma}{j}{r_{j}}}^{+}}(J_{r_{j},j})$})\omega_{\Gamma,\bs{}}
	\end{align*}
	holds. Furthermore
	\begin{align*}
		&\int_{\simplex{}{r}\cact{\Gamma}{}}((\chain{\Gamma}{\fs{}}^{\ast}\df{}\omega_{\Gamma,\fs{}}^{(r-1)})\wedge(\chain{\Gamma}{\bs{}}^{\ast}\omega_{\Gamma,\bs{}}))\\
			=
		&\sum_{1\leq j\leq\mathsf{n}_{\Gamma}}(-1)^{\Rs{\Gamma}{j}{0}}\int_{\simplex{}{r}\cact{\Gamma\delta_{\vs{\Gamma}{j}{0}}}{}}(\Gamma\delta_{\vs{\Gamma}{j}{0}})^{\ast}((\proj{\simplex{}{r}}^{\ast}\omega_{\Gamma,\fs{}}^{(r-1)})\wedge(\proj{\simplex{}{n}}^{\ast}\omega_{\Gamma,\bs{}}))\\
			&+
		\underset{0<i<r_{j}}{\sum_{1\leq j\leq\mathsf{n}_{\Gamma}}}(-1)^{\Rs{\Gamma}{j}{i}}\int_{\simplex{}{r-1}\cact{\Gamma_{\vs{\Gamma}{j}{i}}}{}}\Gamma_{\vs{\Gamma}{j}{i}}^{\ast}(\id{}\times\delta_{\Rs{\Gamma}{j}{i}})^{\ast}((\proj{\simplex{}{r}}^{\ast}\omega_{\Gamma,\fs{}}^{(r-1)})\wedge(\proj{\simplex{}{n}}^{\ast}\omega_{\Gamma,\bs{}}))\\
			&+
		\sum_{1\leq j\leq\mathsf{n}_{\Gamma}}(-1)^{\Rs{\Gamma}{j}{r_{j}}}\int_{\simplex{}{r}\cact{\Gamma\delta_{\vs{\Gamma}{j}{r_{j}}}}{}}(\Gamma\delta_{\vs{\Gamma}{j}{r_{j}}})^{\ast}((\proj{\simplex{}{r}}^{\ast}\omega_{\Gamma,\fs{}}^{(r-1)})\wedge(\proj{\simplex{}{n}}^{\ast}\omega_{\Gamma,\bs{}}))
	\end{align*}
	holds.
\end{proof}
\begin{thm}\label{Stokes}
	Let $\omega\colon X\times U\to\qdR{\dq}{}{}{\mathbb{U}_{\infty}\mathfrak{g}}^{\wedge}$ be a formal differential form with values in a connected $L_{\infty}$-algebra $\mathfrak{g}$ on $X\times U$ which has the finite support in the direction along the projection $\proj{X}$. Then the following holds:
	\begin{align*}
		\fint{\proj{X}}\df{}\omega-\bint{\proj{X}}\omega=\sum_{(\simplex{}{r}\xrightarrow{u}U)\in\mathrm{part}_{\omega}(U)}(-1)^{r}\df{}\fint{\proj{X}}((\id{}\times u)^{\ast}\omega)
	\end{align*}
\end{thm}
\begin{proof}
	It is sufficient to show that for $\mathfrak{g}$-valued formal differential form $\omega$ on simplicial set $\simplex{}{n}\times\simplex{}{r}$. Since for any non-degenerate $n$-simplex $\omega\in[\simplex{}{r},\qdR{\dq}{}{}{\mathbb{U}_{\infty}\mathfrak{g}}^{\wedge}]_{n}$ and surjection $\sigma\colon[m]\to[n]$, 
	\begin{align*}
		&\fint{\proj{\simplex{}{n}}}(\df{}(\sigma\times\id{})^{\ast}\omega^{\vee})
			-
		\bint{\proj{\simplex{}{n}}}((\sigma\times\id{})^{\ast}\omega^{\vee})
			+
		(-1)^{r}\df{}\fint{\proj{\simplex{}{n}}}((\sigma\times\id{})^{\ast}\omega^{\vee})\\
			=
		&\sigma^{\ast}(\fint{\proj{\simplex{}{n}}}(\df{}\omega^{\vee})
			-
		\bint{\proj{\simplex{}{n}}}\omega^{\vee}
			+
		(-1)^{r}\df{}\fint{\proj{\simplex{}{n}}}\omega^{\vee})
	\end{align*}
	holds, we can assume that $\omega\colon\simplex{}{n}\times\simplex{}{r}\to\qdR{\dq}{}{}{\mathbb{U}_{\infty}\mathfrak{g}}^{\wedge}$ gives a non-degenerated $n$-simplex $\omega^{\wedge}$ of $[\simplex{}{r},\qdR{\dq}{}{}{\mathbb{U}_{\infty}\mathfrak{g}}^{\wedge}]$.
	
	For each maximal chain $\Gamma\colon[n+r]\hookrightarrow[n]\times[r]$, we have a decomposition
	\begin{align*}
		\Gamma^{\ast}\omega
			=
		\sum(\chain{\Gamma}{\fs{}}^{\ast}\omega_{\Gamma,\fs{}})\wedge(\chain{\Gamma}{\bs{}}^{\ast}\omega_{\Gamma,\bs{}})
	\end{align*}
	where $\omega_{\Gamma,\fs{}}$ and $\omega_{\Gamma,\bs{}}$ are $\mathfrak{g}$-valued formal differential forms $\omega_{\Gamma,\fs{}}\colon\simplex{}{r}\to\qdR{\dq}{}{}{\mathbb{U}_{\infty}\mathfrak{g}}^{\wedge}$ and $\omega_{\Gamma,\bs{}}\colon\simplex{}{n}\to\qdR{\dq}{}{}{\mathbb{U}_{\infty}\mathfrak{g}}^{\wedge}$, respectively. And $\omega_{\Gamma,\fs{}}$ can be decomposed into ``homogeneous'' elements
	\begin{align*}
		\omega_{\Gamma,\fs{}}=\omega_{\Gamma,\fs{}}^{(r)}+\omega_{\Gamma,\fs{}}^{(r-1)}+\sum_{p=0}^{r-2}\omega_{\Gamma,\fs{}}^{(p)}.
	\end{align*}
	Using this decomposition, we obtain the following:
	\begin{align*}
		\fint{\proj{\simplex{}{n}}}\df{}\omega
			&=
		\sum_{\Gamma}\int_{\simplex{}{r}\cact{\Gamma}{}}\Gamma^{\ast}\df{}\omega\\
			&=
		\sum_{\Gamma}\int_{\simplex{}{r}\cact{\Gamma}{}}\df{}\Gamma^{\ast}\omega\\
			&=
		\sum_{\Gamma}\int_{\simplex{}{r}\cact{\Gamma}{}}\df{}\sum(\chain{\Gamma}{\fs{}}^{\ast}\omega_{\Gamma,\fs{}}^{(r)})\wedge(\chain{\Gamma}{\bs{}}^{\ast}\omega_{\Gamma,\bs{}})
			+
		\sum_{\Gamma}\int_{\simplex{}{r}\cact{\Gamma}{}}\df{}\sum(\chain{\Gamma}{\fs{}}^{\ast}\omega_{\Gamma,\fs{}}^{(r-1)})\wedge(\chain{\Gamma}{\bs{}}^{\ast}\omega_{\Gamma,\bs{}})\\
			&=
		(-1)^{r}\sum_{\Gamma}\int_{\simplex{}{r}\cact{\Gamma}{}}\sum(\chain{\Gamma}{\fs{}}^{\ast}\omega_{\Gamma,\fs{}}^{(r)})\wedge(\chain{\Gamma}{\bs{}}^{\ast}\df{}\omega_{\Gamma,\bs{}})
			+
		\sum_{\Gamma}\int_{\simplex{}{r}\cact{\Gamma}{}}\sum(\chain{\Gamma}{\fs{}}^{\ast}\df{}\omega_{\Gamma,\fs{}}^{(r-1)})\wedge(\chain{\Gamma}{\bs{}}^{\ast}\omega_{\Gamma,\bs{}})\\
			&=
		(-1)^{r}\df{}\fint{\proj{\simplex{}{n}}}\omega
			+
		\sum_{\Gamma}\int_{\simplex{}{r}\cact{\Gamma}{}}\sum(\chain{\Gamma}{\fs{}}^{\ast}\df{}\omega_{\Gamma,\fs{}}^{(r-1)})\wedge(\chain{\Gamma}{\bs{}}^{\ast}\omega_{\Gamma,\bs{}})
	\end{align*}
	Let $\Gamma\colon[n+r]\hookrightarrow[n]\times[r]$ be a maximal chain. Then any elements $v\in\mathsf{Inn}_{\fs{}}$ can be represented as $v=\vs{\Gamma}{j}{r_{j}}$, on the other hand any elements $v\in\mathsf{Inn}_{\bs{}}$ can be represented as $v=\vs{\Gamma}{j}{0}$. For each $j=1,\dots,m$, we obtain a maximal chain $\Gamma_{j}\colon[n+r]\hookrightarrow[n]\times[r]$ satisfying $\Gamma\neq\Gamma_{j}$ and $\Gamma\delta_{\vs{\Gamma}{j}{r_{j}}}=\Gamma_{j}\delta_{\vs{\Gamma}{j}{r_{j}}}$ as follows (where $m=|\mathrm{Im}\Fs{\Gamma}|$):
	\begin{align*}
		\Gamma_{j}(h)
			\coloneqq
		\begin{cases}
			\Gamma(h)&(h\neq\vs{\Gamma}{j}{r_{j}})\\
			(\chain{\Gamma}{\bs{}}(h-1)+1,\chain{\Gamma}{\fs{}}(h-1))&(h=\vs{\Gamma}{j}{r_{j}})
		\end{cases}.
	\end{align*}
	For these maximal chains,
	\begin{align*}
		\vs{\Gamma}{j}{r_{j}}
			=
		\begin{cases}
			\vs{\Gamma_{j}}{j+1}{0}&(r_{j}>1)\\
			\vs{\Gamma_{j}}{j}{0}&(r_{j}=1)
		\end{cases}
		\Rs{\Gamma}{j}{r_{j}}
			=
		\begin{cases}
			\Rs{\Gamma_{j}}{j+1}{0}+1&(r_{j}>1)\\
			\Rs{\Gamma_{j}}{j}{0}+1&(r_{j}=1)
		\end{cases}
	\end{align*}
	holds. Thus
	\begin{align*}
		\sum_{\Gamma\colon[n+r]\hookrightarrow[n]\times[r]}\int_{\simplex{}{r}\cact{\Gamma}{}}\sum(\chain{\Gamma}{\fs{}}^{\ast}\df{}\omega_{\Gamma,\fs{}}^{(r-1)})\wedge(\chain{\Gamma}{\bs{}}^{\ast}\omega_{\Gamma,\bs{}})
			=
		\sum_{h=0}^{r}(-1)^{i}\sum_{\Gamma\colon[n+r]\hookrightarrow[n]\times[r]}\int_{\simplex{}{r}\cact{\Gamma}{}}\Gamma^{\ast}(\id{}\times\delta_{h})^{\ast}\omega
			=
		\bint{\proj{\simplex{}{n}}}\omega
	\end{align*}
	holds from Lemma \ref{Stokes Lem} and Lemma \ref{Gl Lem7}.
\end{proof}
\section{Simplicial Holonomy}\label{sec4}
\subsection{Iterated Integral}
Let $\mathfrak{g}$ be a connected $L_{\infty}$-algebra, $X$ be a simplicial set and $\omega_{1},\dots,\omega_{r}\colon X\to\qdR{\dq}{}{}{\mathbb{U}_{\infty}\mathfrak{g}}^{\wedge}$ be $\mathfrak{g}$-valued formal differential forms on $X$. Then we obtain a $\mathfrak{g}$-valued formal differential form on $X^{r}\coloneqq\underbrace{X\times\dots\times X}_{r}$ as $\proj{1}^{\ast}\omega_{1}\wedge\dots\wedge\proj{r}^{\ast}\omega_{r}$. It gives a $\mathfrak{g}$-valued formal differential form on $[\simplex{}{1},X]^{r}\times\simplex{}{1}^{r}$ by using a counit $\mathrm{ev}\colon\simplex{}{1}\times[\simplex{}{1},-]$ of the adjoint pair $\simplex{}{1}\times-\dashv[\simplex{}{1},-]$. In addition, by using a simplicial map $\iota_{r}\colon\simplex{}{r}\to\simplex{}{1}^{r}$ obtained from an order-preserving map $[r]\to[1]^{r}$ defined as $i\mapsto(\underbrace{1,\dots,1}_{i},0,\dots,0)$ and the diagonal map $[\simplex{}{1},X]\to[\simplex{}{1},X]^{r}$, we obtain a $\mathfrak{g}$-valued formal differential form $\omega_{1}\times\dots\times\omega_{r}$ on $[\simplex{}{1},X]\times\simplex{}{r}$. Then we obtain a $\mathfrak{g}$-valued formal differential form on path simplicial set $[\simplex{}{1},X]$ as a fiberwise integration of $\omega_{1}\times\dots\times\omega_{r}$ along the projection $[\simplex{}{1},X]\times\simplex{}{r}\to[\simplex{}{1},X]$. We call it the \Def{iterated integral of $\omega_{1},\dots,\omega_{r}$} and denote it as $\int\omega_{1}\cdots\omega_{r}$. It is precisely an analogy of Chen's iterated integral.
\begin{center}
\begin{tikzpicture}[auto]
	\node (11) at (0, 3) {$[\simplex{}{1},X]^{r}\times\simplex{}{1}^{r}$};
	\node (12) at (3.5, 3) {$(\simplex{}{1}\times[\simplex{}{1},X])^{r}$};
	\node (21) at (0, 2) {$[\simplex{}{1},X]\times\simplex{}{r}$};
	\node (22) at (3.5, 2) {$X^{r}$};
	\node (23) at (5.5, 2) {$X^{r}$};
	\node (24) at (7.5, 2) {$X$};
	\node (31) at (0, 1) {$[\simplex{}{1},X]\times\simplex{}{r}$};
	\node (32) at (3.5, 1) {$\qdR{\dq}{}{}{\mathbb{U}_{\infty}\mathfrak{g}}^{\wedge}$};
	\node (33) at (5.5, 1) {$(\qdR{\dq}{}{}{\mathbb{U}_{\infty}\mathfrak{g}}^{\wedge})^{r}$};
	\node (34) at (7.5, 1) {$\qdR{\dq}{}{}{\mathbb{U}_{\infty}\mathfrak{g}}^{\wedge}$};
	\node (41) at (0, 0) {$[\simplex{}{1},X]$};
	\node (42) at (3.5, 0) {$\qdR{\dq}{}{}{\mathbb{U}_{\infty}\mathfrak{g}}^{\wedge}$};
	\path[draw, transform canvas={yshift=1pt}] (11) -- (12);
	\path[draw, transform canvas={yshift=-1pt}] (11) -- (12);
	\path[draw, ->] (21) --node {$\scriptstyle \text{diagonal}\times\iota_{r}$} (11);
	\path[draw, ->] (12) --node {$\scriptstyle (\mathrm{ev})^{r}$} (22);
	\path[draw, ->] (21) --node {$\scriptstyle \phi_{r}$} (22);
	\path[draw, transform canvas={yshift=1pt}] (22) -- (23);
	\path[draw, transform canvas={yshift=-1pt}] (22) -- (23);
	\path[draw, ->] (23) --node {$\scriptstyle \proj{i}$} (24);
	\path[draw, transform canvas={xshift=1pt}] (21) -- (31);
	\path[draw, transform canvas={xshift=-1pt}] (21) -- (31);
	\path[draw, ->] (22) --node[swap] {$\scriptstyle \proj{1}^{\ast}\omega_{1}\wedge\dots\wedge\proj{r}^{\ast}\omega_{r}$} (32);
	\path[draw, ->] (23) --node {$\scriptstyle \sqcap_{i}\proj{i}^{\ast}\omega_{i}$} (33);
	\path[draw, ->] (24) --node {$\scriptstyle \omega_{i}$} (34);
	\path[draw, ->] (31) --node[swap] {$\scriptstyle \omega_{1}\times\dots\times\omega_{r}$} (32);
	\path[draw, ->] (33) --node {$\scriptstyle \wedge$} (32);
	\path[draw, ->] (33) --node[swap] {$\scriptstyle \proj{i}$} (34);
	\path[draw, ->] (41) --node {$\scriptstyle \int\omega_{1}\dots\omega_{r}$} (42);
\end{tikzpicture}
\end{center}
We obtain a degree $1$ map $\mathsf{C}\colon\mathsf{T}\qdR{\dq}{}{}{}(X,\mathfrak{g})[-1]\to\mathsf{T}\qdR{\dq}{}{}{}([\simplex{}{1},X],\mathfrak{g})[-1]$ as
\begin{align*}
	\mathsf{C}(\omega_{1}[-1]\otimes\dots\otimes\omega_{r}[-1])\coloneqq(-1)^{\overset{r}{\underset{i=1}{\sum}}(r-i)(|\omega_{i}|-1)}(\int\omega_{1}\dots\omega_{r})[-1].
\end{align*}
\begin{prop}
	We define a degree $1$ map $\overline{\df{}}\colon\mathsf{T}\qdR{\dq}{}{}{}(X,\mathfrak{g})[-1]\to\mathsf{T}\qdR{\dq}{}{}{}(X,\mathfrak{g})[-1]$ as 
	\begin{align*}
		\overline{\df{}}(\omega_{1}[-1]\otimes\dots\otimes\omega_{r}[-1])
			\coloneqq&
		\sum_{i=1}^{r}(-1)^{|\omega_{1}|+\dots+|\omega_{i-1}|+i}\omega_{1}[-1]\otimes\dots\otimes\df{}\omega_{i}[-1]\otimes\dots\otimes\omega_{r}[-1]\\
			&+
		\sum_{i=1}^{r-1}(-1)^{|\omega_{1}|+\dots+|\omega_{i}|+i}\omega_{1}[-1]\otimes\dots\otimes(\omega_{i}\wedge\omega_{i+1})[-1]\otimes\dots\otimes\omega_{r}[-1].
	\end{align*}
	Then, for each homogeneous $\mathfrak{g}$-valued formal differential form $\omega_{1},\dots,\omega_{r}$ on $X$,
	\begin{align*}
		\df{}\mathsf{C}(\omega_{1}[-1]\otimes\dots\otimes\omega_{r}[-1])
			=&
		\mathsf{C}\overline{\df{}}(\omega_{1}[-1]\otimes\dots\otimes\omega_{r}[-1])\\
			&+
		(\mathrm{E}_{1}^{\ast}\omega_{1}\wedge(-1)^{\overset{r-1}{\underset{i=1}{\sum}}(r-1-i)(|\omega_{i+1}|-1)}(\int\omega_{2}\dots\omega_{r}))[-1]\\
			&-
		(-1)^{|\omega_{1}|+\dots+|\omega_{i-1}|-(i-1)}((-1)^{\overset{r-1}{\underset{i=1}{\sum}}(r-1-i)(|\omega_{i}|-1)}(\int\omega_{1}\dots\omega_{r-1})\wedge\mathrm{E}_{0}^{\ast}\omega_{r})[-1]
	\end{align*}
	holds where $\mathrm{E}_{\ep}\colon[\simplex{}{1},X]\to X$ is obtained as a composition $[\simplex{}{1},X]\to[\Delta\{\ep\},X]\cong\simplex{}{0}\times[\simplex{}{0},X]\xrightarrow{\mathrm{ev}}X$ for each $\ep=0,1$.
\end{prop}
\begin{proof}
	From Stokes's theorem \ref{Stokes}, the following holds:
	\begin{align*}
		(-1)^{r}\df{}\int\omega_{1}\dots\omega_{r}
			=&
		\sum_{i=1}^{r}(-1)^{|\omega_{1}|+\dots+|\omega_{i-1}|}\fint{\proj{[\simplex{}{1},X]}}\phi_{r}^{\ast}(\proj{1}^{\ast}\omega_{1}\wedge\dots\wedge\proj{i}^{\ast}\df{}\omega_{i}\wedge\dots\wedge\proj{r}^{\ast}\omega_{r})\\
			&+
		\sum_{i=0}^{r}(-1)^{i+1}\fint{\proj{[\simplex{}{1},X]}}(\id{}\times\delta_{i})^{\ast}\phi_{r}^{\ast}(\proj{1}^{\ast}\omega_{1}\wedge\dots\wedge\proj{r}^{\ast}\omega_{r})
	\end{align*}
	\begin{center}
	\begin{tikzpicture}[auto]
		\node (11) at (0, 4) {$[\simplex{}{1},X]\times\simplex{}{r-1}$};
		\node (13) at (8, 4) {$\simplex{}{r-1}\times[\simplex{}{1},X]$};
		\node (21) at (0, 3) {$[\simplex{}{1},X]\times\simplex{}{r}$};
		\node (22) at (4, 3) {$[\simplex{}{1},X]\times\simplex{}{r}$};
		\node (23) at (8, 3) {$\simplex{}{r}\times[\simplex{}{1},X]$};
		\node (31) at (0, 2) {$[\simplex{}{1},X]^{r}\times\simplex{}{1}^{r}$};
		\node (32) at (4, 2) {$[\simplex{}{1},X]\times\simplex{}{1}^{r}$};
		\node (33) at (8, 2) {$\simplex{}{1}^{r}\times[\simplex{}{1},X]$};
		\node (41) at (0, 1) {$(\simplex{}{1}\times[\simplex{}{1},X])^{r}$};
		\node (43) at (8, 1) {$\simplex{}{1}\times[\simplex{}{1},X]$};
		\node (51) at (0, 0) {$X^{r}$};
		\node (53) at (8, 0) {$X$};
		\path[draw, ->] (11) --node[swap] {$\scriptstyle \id{}\times\delta_{i}$} (21);
		\path[draw, ->] (21) --node[swap] {$\scriptstyle \mathrm{diagonal}\times\iota_{r}$} (31);
		\path[draw, transform canvas={xshift=1pt}] (31) -- (41);
		\path[draw, transform canvas={xshift=-1pt}] (31) --node[swap] {\rotatebox{90}{$\scriptstyle \sim$}} (41);
		\path[draw, ->] (41) --node[swap] {$\scriptstyle \mathrm{ev}$} (51);
		\path[draw, ->] (21) -- (-2, 3) --node[swap] {$\scriptstyle \phi_{r}$} (-2, 0) -- (51);
		\path[draw, ->] (22) --node[swap] {$\scriptstyle \id{}\times\iota_{r}$}  (32);
		\path[draw, ->] (13) --node {$\scriptstyle \delta_{i}\times\id{}$} (23);
		\path[draw, ->] (23) --node {$\scriptstyle \iota_{r}\times\id{}$}  (33);
		\path[draw, ->] (33) --node {$\scriptstyle \proj{j}\times\id{}$} (43);
		\path[draw, ->] (43) --node {$\scriptstyle \mathrm{ev}$} (53);
		\path[draw, transform canvas={yshift=1pt}] (21) -- (22);
		\path[draw, transform canvas={yshift=-1pt}] (21) -- (22);
		\path[draw, transform canvas={yshift=1pt}] (22) --node {$\scriptstyle \sim$} (23);
		\path[draw, transform canvas={yshift=-1pt}] (22) -- (23);
		\path[draw, ->] (31) --node {$\scriptstyle \proj{j}\times\id{}$} (32);
		\path[draw, transform canvas={yshift=1pt}] (32) --node {$\scriptstyle \sim$} (33);
		\path[draw, transform canvas={yshift=-1pt}] (32) -- (33);
		\path[draw, transform canvas={yshift=1pt}] (11) --node {$\scriptstyle \sim$} (13);
		\path[draw, transform canvas={yshift=-1pt}] (11) -- (13);
		\path[draw, ->] (41) --node {$\scriptstyle \proj{j}$} (43);
		\path[draw, ->] (51) --node {$\scriptstyle \proj{j}$} (53);
	\end{tikzpicture}
	\end{center}
	For each pair of $i=0,\dots,r$ and $j=1,\dots,r$, respectively, the following holds:
	\begin{align*}
		\proj{j}\iota_{r}\delta_{i}
			=
		\begin{cases}
			\text{constant $1$}&((i,j)=(0,1))\\
			\proj{j-1}\iota_{r-1}&((i,j)\neq(0,1)\text{ and }i<j)\\
			\proj{j}\iota_{r-1}&((i,j)\neq(r,r)\text{ and }i\geq j)\\
			\text{constant $0$}&((i,j)=(r,r))
		\end{cases}.
	\end{align*}
	In addition, the following diagram is commutative:
	\begin{center}
	\begin{tikzpicture}[auto]
		\node (11) at (0, 2) {$\simplex{}{r-1}\times[\simplex{}{1},X]$};
		\node (12) at (4, 2) {$\simplex{}{0}\times[\simplex{}{1},X]$};
		\node (13) at (8, 2) {$\simplex{}{1}\times[\simplex{}{1},X]$};
		\node (21) at (0, 1) {$\simplex{}{r-1}\times[\simplex{}{1},X]$};
		\node (22) at (4, 1) {$\simplex{}{0}\times[\simplex{}{0},X]$};
		\node (23) at (8, 1) {$X$};
		\node (31) at (0, 0) {$[\simplex{}{1},X]$};
		\node (32) at (4, 0) {$[\simplex{}{0},X]$};
		\node (33) at (8, 0) {$X$};
		\path[draw, ->] (11) -- (12);
		\path[draw, ->] (31) --node[swap] {$\scriptstyle [\delta_{\ep},X]$} (32);
		\path[draw, double, double distance=2pt] (11) -- (21);
		\path[draw, ->] (21) --node[swap] {$\scriptstyle \proj{[\simplex{}{1},X]}$} (31);
		\path[draw, ->] (22) --node[swap] {$\scriptstyle \proj{[\simplex{}{1},X]}$} (32);
		\path[draw, ->] (21) -- (22);
		\path[draw, ->] (12) --node {$\scriptstyle \delta_{\ep}\times\id{}$} (13);
		\path[draw, ->] (12) --node[swap] {$\scriptstyle \simplex{}{0}\times[\delta_{\ep},X]$} (22);
		\path[draw, ->] (13) --node {$\scriptstyle \mathrm{ev}$} (23);
		\path[draw, ->] (22) --node[swap] {$\scriptstyle \mathrm{ev}$} (23);
		\path[draw, ->] (32) -- (33);
		\path[draw, double, double distance=2pt] (23) -- (33);
	\end{tikzpicture}.
	\end{center}
	Therefore
	\begin{align*}
		&\sum_{i=0}^{r}(-1)^{i+1}\fint{\proj{[\simplex{}{1},X]}}(\id{}\times\delta_{i})^{\ast}\phi_{r}^{\ast}(\proj{1}^{\ast}\omega_{1}\wedge\dots\wedge\proj{r}^{\ast}\omega_{r})\\
			=&
		-\fint{\proj{[\simplex{}{1},X]}}\proj{[\simplex{}{1},X]}^{\ast}\mathrm{E}_{1}^{\ast}\omega_{1}\wedge\phi_{r-1}^{\ast}(\proj{1}^{\ast}\omega_{2}\wedge\dots\wedge\proj{r-1}^{\ast}\omega_{r})\\
			&+
		\sum_{i=1}^{r-1}(-1)^{i+1}\fint{\proj{[\simplex{}{1},X]}}\phi_{r-1}^{\ast}(\proj{1}^{\ast}\omega_{1}\wedge\dots\wedge\proj{i}^{\ast}(\omega_{i}\wedge\omega_{i+1})\wedge\dots\wedge\proj{r}^{\ast}\omega_{r})\\
			&+
		(-1)^{r+1}\fint{\proj{[\simplex{}{1},X]}}(\phi_{r-1}^{\ast}(\proj{1}^{\ast}\omega_{1}\wedge\dots\wedge\proj{r-1}^{\ast}\omega_{r-1})\wedge\proj{[\simplex{}{1},X]}^{\ast}\mathrm{E}_{0}^{\ast}\omega_{r})\\
			=&
		\sum_{i=1}^{r-1}(-1)^{i+1}\int\omega_{1}\dots(\omega_{i}\wedge\omega_{i+1})\dots\omega_{r})\\
			&+
		(-1)^{|\omega_{1}|(|\omega_{2}|+\dots+|\omega_{r}|)+1}(\int\omega_{2}\dots\omega_{r})\wedge\mathrm{E}_{1}^{\ast}\omega_{1}
			-
		(-1)^{r}(\int\omega_{1}\dots\omega_{r-1})\wedge\mathrm{E}_{0}^{\ast}\omega_{r}\\
			=&
		\sum_{i=1}^{r-1}(-1)^{i+1}\int\omega_{1}\dots(\omega_{i}\wedge\omega_{i+1})\dots\omega_{r})
			+
		(-1)^{1-|\omega_{1}|(r-1)}\mathrm{E}_{1}^{\ast}\omega_{1}\wedge(\int\omega_{2}\dots\omega_{r})
			-
		(-1)^{r}(\int\omega_{1}\dots\omega_{r-1})\wedge\mathrm{E}_{0}^{\ast}\omega_{r}
	\end{align*}
	holds.
\end{proof}
\begin{cor}
For each homogeneous $\mathfrak{g}$-valued formal differential form $\omega_{1},\dots,\omega_{r}$ on $X$, the following holds:
\begin{align*}
	\df{}\int\omega_{1}\dots\omega_{r}
		=&
	\sum_{i=1}^{r}(-1)^{|\omega_{1}|+\dots+|\omega_{i-1}|+r}(\int\omega_{1}\dots\df{}\omega_{i}\dots\omega_{r})
		+
	\sum_{i=1}^{r-1}(-1)^{r-1-i}(\int\omega_{1}\dots(\omega_{i}\wedge\omega_{i+1})\dots\omega_{r})\\
		&+
	(-1)^{(r-1)(|\omega_{1}|-1)}\mathrm{E}_{1}^{\ast}\omega_{1}\wedge(\int\omega_{2}\dots\omega_{r})
		-
	(\int\omega_{1}\dots\omega_{r-1})\wedge\mathrm{E}_{0}^{\ast}\omega_{r}.
\end{align*}
\end{cor}
\subsection{de Rham's Map}
For any simplicial set $X$, we obtain a chain complex $\Z[X]$
$$\cdots\to\Z[X]_{n}\xrightarrow{\overset{n}{\underset{i=0}{\sum}}(-1)^{i}d_{i}}\Z[X]_{n-1}\to\dots\to\Z[X]_{0}\to0\to\cdots.$$
Using the Alexander-Whitney map, We can define a coproduct $\cup^{\ast}$ on $\Z[X]$ as follows:
\begin{align*}
	\cup^{\ast}_{n}(\sum_{i}m_{i}x_{i})\coloneqq\sum_{i}\sum_{p+q=n}m_{i}(x_{i}|_{\Delta\{0,\dots,p\}})\otimes(x_{i}|_{\Delta\{p,\dots,p+q\}})
\end{align*}
In addition, the unique map $X\to\simplex{}{0}$ determines a chain map $\ep\colon\Z[X]\to\Z$. They give a dg coalgebra $(\Z[X],\cup^{\ast},\ep)$. Hence, for any connected $L_{\infty}$-algebra $\mathfrak{g}$, we obtain a dg algebra
$$\qC{\dq}{\bullet}{}{}(X,\mathfrak{g})\coloneqq\prod_{p+\bullet=q}\mathbb{U}_{\infty}\mathfrak{g}_{p}\otimes\Hom{\Z}(\Z[X]_{q},\Z\langle\dq\rangle).$$
\begin{lem}
	Let $X$ be a simplicial set and $\mathfrak{g}$ be an $L_{\infty}$-algebra. For each $\mathfrak{g}$-valued formal differential form on X $\omega\colon X\to\qdR{\dq}{}{}{\mathbb{U}_{\infty}\mathfrak{g}}^{\wedge}$ and a linear combination of simplices of $X$ $\sum_{x}m_{x}x$, we define $\langle\omega,\sum_{x}m_{x}x\rangle$ as
	\begin{align*}
		\langle\omega,\sum_{x}m_{x}x\rangle=\int_{\sum_{x}m_{x}x}\omega\coloneqq\sum_{x}m_{x}\fint{\proj{\simplex{}{0}}}x^{\ast}\omega.
	\end{align*}
	The we obtain a chain map $\int\colon\qdR{\dq}{\bullet}{}{}(X,\mathfrak{g})\to\qC{\dq}{\bullet}{}{}(X,\mathfrak{g})$.
\end{lem}
\begin{proof}
	From Stokes's theorem \ref{Stokes}, the following follows:
	\begin{align*}
		\int_{x}\df{}\omega=\fint{\proj{\simplex{}{0}}}x^{\ast}\df{}\omega=(\pm)\df{}\fint{\proj{\simplex{}{0}}}x^{\ast}\omega+\bint{\proj{}{\simplex{}{0}}}x^{\ast}\omega=\sum_{i}(-1)^{i}\fint{\proj{\simplex{}{0}}}\delta_{i}^{\ast}x^{\ast}\omega=\sum_{i}(-1)^{i}\int_{x\delta_{i}}\omega=\int_{\partial x}\omega
	\end{align*}
\end{proof}
\subsection{Simplicial Holonomy}
Let $\mathfrak{g}$ be a connected $L_{\infty}$-algebra and $\hat{\mathbb{U}}_{\infty}\mathfrak{g}$ be the completion of universal enveloping algebra $\mathbb{U}_{\infty}\mathfrak{g}$ of $\mathfrak{g}$. They obtain the following (dg) algebra for each non-negative integer $n\geq0$:
\begin{align*}
	\qG{\dq}{\bullet}{n}{\mathfrak{g}}\coloneqq\prod_{p+\bullet=q}\hat{\mathbb{U}}_{\infty}\mathfrak{g}_{p}\otimes\Hom{\Z}(\Z[\simplex{}{n}]_{q},\Z\langle\dq\rangle).
\end{align*}
It is obvious that there is an embedding $\qC{\dq}{\bullet}{}{}(\simplex{}{n},\mathfrak{g})\hookrightarrow\qG{\dq}{\bullet}{}{\mathfrak{g}}$ as a simplicial set for each non-negative integer $n$.
\begin{thm}
	A generalized connection $\nabla$ with values in connected $L_{\infty}$-algebra $\mathfrak{g}$ (over $\Z$) on simplicial set $X$ gives a simplicial map $\shol{\nabla}{}\colon[\simplex{}{1},X]\to\qG{\dq}{}{}{\mathfrak{g}}$.
\end{thm}
\begin{proof}
For each non-negative integer $r\geq0$, we obtain a simplicial map
\begin{align*}
	\int\circ\int\underbrace{\nabla\cdots\nabla}_{r}\colon[\simplex{}{1},X]\to\qC{\dq}{}{}{}(\simplex{}{-},\mathfrak{g})
\end{align*}
using the iterated integral and a (simplicial) chain map $\int\colon\qdR{\dq}{}{}{\mathbb{U}_{\infty}\mathfrak{g}}^{\wedge}\to\qC{\dq}{}{}{}(\simplex{}{-},\mathfrak{g})$. Furthermore, we obtain a simplicial map
\begin{align*}
	\shol{\nabla}{}\coloneqq\sum_{r=0}^{\infty}\int\int\underbrace{\nabla\cdots\nabla}_{r}\colon[\simplex{}{1},X]\to\qG{\dq}{}{}{\mathfrak{g}}.
\end{align*}
\end{proof}
\subsection{Path $A_{\infty}$-categories}
Fix a commutative ring $\K$. Let $X$ be a simplicial set. A family of simplicial sets $\{X(x,y)\}_{x,y\in X_{0}}$ is obtained by assigning the following pullback to each pair $(x,y)$ of $0$-simplices of $X$:
\begin{center}
\begin{tikzpicture}[auto]
	\node (11) at (0, 3) {$X(x,y)$};+
	\node (13) at (5, 3) {$[\simplex{}{1},X]$};+
	\node (23) at (5, 2) {$[\simplex{}{1},X]\times[\simplex{}{1},X]$};
	\node (33) at (5, 1) {$[\Delta\{0\},X]\times[\Delta\{1\},X]$};
	\node (41) at (0, 0) {$\simplex{}{0}$};+
	\node (42) at (2.5, 0) {$\simplex{}{0}\times\simplex{}{0}$};+
	\node (43) at (5, 0) {$X\times X$};
	\path[draw, ->] (11) -- (41);
	\path[draw, right hook->] (11) -- (13);
	\path[draw, ->] (13) -- (23);
	\path[draw, ->] (23) -- (33);
	\path[draw, transform canvas={xshift=1pt}] (33) -- (43);
	\path[draw, transform canvas={xshift=-1pt}] (33) -- (43);
	\path[draw, ->] (41) -- (42);
	\path[draw, ->] (42) -- (43);
	\path[draw] (0.2, 2) -- (1, 2) -- (1, 2.8);
\end{tikzpicture}.
\end{center}
\begin{ex}\label{path space of standard simplices}
	Let $n$ be a non-negative integer and $(i,j)$ be a pair of integers satisfying $0\leq i,j\leq n$. For each $p\geq0$, a $p$-simplex $\simplex{}{p}\to\simplex{}{n}(i,j)$ corresponds to an order-preserving map $\gamma\colon[1]\times[p]\to[n]$ satisfying $\gamma(-,0)=i$ and $\gamma(-,1)=j$. In other words,
	\begin{align*}
		\simplex{}{n}(i,j)
			\cong
		\{\gamma\colon[1]\to[n]|\gamma(0)=i\text{ and }\gamma(1)=j\}
			\cong
		\begin{cases}
			\{\ast\}&(i\leq j)\\
			\emptyset&(i>j)
		\end{cases}
	\end{align*}
	holds.
\end{ex}
And then a dg quiver $\mathcal{Q}(X,\K)$ is obtained by assigning a chain complex $\K[X(x,y)]$
\begin{align*}
	\cdots\to\K[X(x,y)]_{n}\xrightarrow{\overset{n}{\underset{i=0}{\sum}}(-1)^{i}d_{i}}\K[X(x,y)]_{n-1}\to\dots\to\K[X(x,y)]_{0}\to0\to\cdots
\end{align*}
to each pair $(x,y)$ of $0$-simplices of $X$. In addition, we obtain an $A_{\infty}$-category $\mathcal{F}\mathcal{Q}(X,\K)$ as a free $A_{\infty}$-category generated by a dg quiver $\mathcal{Q}(X,\K)$.
\begin{prop}
	There exists a canonical natural transformation $\pi\colon\mathcal{F}\mathcal{Q}(-,\K)\to\category{A}_{\infty}\colon\D\to\mathsf{u}A_{\infty}\mathsf{Cat}_{\K}$.
\end{prop}
\begin{proof}
	By theorem \ref{free A infty category} and proposition \ref{unit}, it suffices to show the existence of a natural transformation $\pi\colon\mathcal{Q}(-,\K)\to\category{A}_{\infty}$. Since the simplicial set $\simplex{}{n}(i,j)$ is not empty if and only if $i\leq j$ for each integers $i,j\in[n]$, we obtain a canonical family of maps $\{\pi^{n}_{i,j}\colon\mathcal{Q}(\simplex{}{n},\K)(i,j)\to\category{A}_{\infty}^{n}(i,j)\}_{i,j}$. It define a natural transformation $\pi\colon\mathcal{Q}(-,\K)\to\category{A}_{\infty}$.
\end{proof}
We obtain functors and natural transformations
\begin{center}
\begin{tikzpicture}[auto]
	\node (11) at (0, 2) {$\mathfrak{N}_{A_{\infty}}(-)_{\bullet}$};
	\node (12) at (5.5, 2) {$\tilde{\mathfrak{N}}_{A_{\infty}}(-)_{\bullet}$};
	\node (13) at (10, 2) {$\mathscr{N}_{A_{\infty}}(-)_{\bullet}$};
	\node (21) at (0, 1) {$\Hom{A_{\infty}\mathsf{Cat}_{\K}}(\mathcal{F}\mathcal{Q}(\mathrm{Ex}^{\infty}\simplex{}{\bullet},\K),-)$};
	\node (22) at (5.5, 1) {$\Hom{A_{\infty}\mathsf{Cat}_{\K}}(\mathcal{F}\mathcal{Q}(\simplex{}{\bullet},\K),-)$};
	\node (23) at (10, 1) {$\Hom{A_{\infty}\mathsf{Cat}_{\K}}(\category{A}_{\infty}^{\bullet},-)$};
	\node (31) at (0, 0) {$\Hom{\mathsf{dgQ}}(\mathcal{Q}(\mathrm{Ex}^{\infty}\simplex{}{\bullet},\K),-)$};
	\node (32) at (5.5, 0) {$\Hom{\mathsf{dgQ}}(\mathcal{Q}(\simplex{}{\bullet},\K),-)$};
	\path[draw, transform canvas={xshift=1pt}] (11) -- (21);
	\path[draw, transform canvas={xshift=-1pt}] (11) -- (21);
	\path[draw, transform canvas={xshift=1pt}] (12) -- (22);
	\path[draw, transform canvas={xshift=-1pt}] (12) -- (22);
	\path[draw, transform canvas={xshift=1pt}] (13) -- (23);
	\path[draw, transform canvas={xshift=-1pt}] (13) -- (23);
	\path[draw, transform canvas={xshift=1pt}] (21) --node {\rotatebox{-90}{$\scriptstyle \sim$}} (31);
	\path[draw, transform canvas={xshift=-1pt}] (21) -- (31);
	\path[draw, transform canvas={xshift=1pt}] (22) --node {\rotatebox{-90}{$\scriptstyle \sim$}} (32);
	\path[draw, transform canvas={xshift=-1pt}] (22) -- (32);
	\path[draw, ->] (11) -- (12);
	\path[draw, ->] (13) -- (12);
	\path[draw, ->] (21) -- (22);
	\path[draw, ->] (23) -- (22);
	\path[draw, ->] (31) -- (32);
\end{tikzpicture}
\end{center}
which are similar to $A_{\infty}$-nerve but ``laxer''.

For each simplicial set $X$, we call the free $A_{\infty}$-category $\mathcal{F}\category{Q}(\mathrm{Ex}^{\infty}X,\K)$ the \Def{$\K$-coefficient path $A_{\infty}$-category of simplicial set $X$} and denote $\mathcal{P}(X,\K)$. It is an invariant since the above assignment defines a functor from the category of simplicial sets $\sSet$ to the category of dg quivers. 

Let $\nabla\colon X\to\qG{\dq}{}{}{\mathfrak{g}}$ be a generalized connection with values in connected $L_{\infty}$-algebra $\mathfrak{g}$. Since $\qG{\dq}{}{}{\mathfrak{g}}$ is a Kan complex, there is a lift $\tilde{\nabla}\colon\mathrm{Ex}^{\infty}X\to\qG{\dq}{}{}{\mathfrak{g}}$ of $\nabla$.
\begin{center}
\begin{tikzpicture}[auto]
	\node (11) at (0, 1) {$X$};
	\node (12) at (2, 1) {$\qG{\dq}{}{}{\mathfrak{g}}$};
	\node (21) at (0, 0) {$\mathrm{Ex}^{\infty}X$};
	\path[draw, ->] (11) --node {$\scriptstyle \nabla$} (12);
	\path[draw, ->, dashed] (21) --node[swap] {$\scriptstyle \tilde{\nabla}$} (12);
	\path[draw, right hook->] (11) --node[swap] {\rotatebox{90}{$\scriptstyle \sim$}} (21);
\end{tikzpicture}
\end{center}
For each $0$-simplex $x,y\in\mathrm{Ex}^{\infty}X_{0}$, the map gives a simplicial map
\begin{align*}
	\mathrm{Ex}^{\infty}X(x,y)\hookrightarrow[\simplex{}{1},\mathrm{Ex}^{\infty}X]\xrightarrow{\shol{\tilde{\nabla}}{}}\qG{\dq}{}{}{\mathfrak{g}},
\end{align*}
thus we obtain a simplicial linear map $\Z[\mathrm{Ex}^{\infty}X(x,y)]\to\qG{\dq}{}{}{\mathfrak{g}}$ and a morphism of dg quiver $\mathcal{Q}(\mathrm{Ex}^{\infty}X,\Z)\to\qG{\dq}{}{}{\mathfrak{g}}$. Since $\qG{\dq}{}{}{\mathfrak{g}}$ is a simplicial algebra, we can regard $\qG{\dq}{}{}{\mathfrak{g}}$ as a (strict unital) $A_{\infty}$-algebra. Therefore we obtain an $A_{\infty}$-functor $\Ahol{\nabla}\colon\mathcal{P}(\mathrm{Ex}^{\infty}X,\Z)\to\qG{\dq}{}{}{\mathfrak{g}}$.
\begin{rem}
	The $A_{\infty}$-functor $\Ahol{\nabla}$ depends on the choice of lift $\tilde{\nabla}$.
\end{rem}
We hope that the $A_{\infty}$-category $\mathcal{P}(X,\K)$ is a $\K$-linearization of a simplicial set $X$ and the $A_{\infty}$-functor $\Ahol{\nabla}$ is a linearization of the simplicial map $\shol{\nabla}{}\colon[\simplex{}{1},X]\to\qG{\dq}{}{}{\mathfrak{g}}$. However, there are several problems. These are discussed in the next section.
\subsection{Comparison with Known Results and Future Problems}
For each $m\geq0$, we denote the subposet $\{U\subset\R^{n}|U\supset\topsimplex{n}\}\subset\mathfrak{O}(\R^{n})$ of the poset of open subsets of Euclidian space $\R^{n}$ as $\mathfrak{O}(\R^{n},\topsimplex{n})$. Then any smooth manifold $\mathcal{M}$ gives a (canonical) presheaf $\tilde{S}^{\infty}_{n}(\mathcal{M})\colon\mathfrak{O}(\R^{n},\topsimplex{n})\opposite\to\Set$ as
\begin{align*}
	\tilde{S}^{\infty}_{n}(\mathcal{M})(U)
		&\coloneqq
	\{\gamma\colon U\to\mathcal{M}|\gamma\text{ is a smooth map.}\}.
\intertext{For each positive integer $n>0$, we obtain a subpresheaf $tS^{\infty}_{n}(\mathcal{M})$ as follows:}
	tS^{\infty}_{n}(\mathcal{M})(U)
		&\coloneqq
	\{\gamma\colon U\to\mathcal{M}|\mathrm{Ker}(d\gamma_{x})\neq0\text{ for some $x\in \topsimplex{n}$}\}.
\end{align*}
Any order-preserving map $\alpha\colon[m]\to[n]$ gives an affine map $\alpha_{\ast}\colon\R^{m}\to\R^{n}$ satisfying $\alpha_{\ast}(\topsimplex{m})\subset\topsimplex{n}$, we obtain an order-preserving map $\alpha_{\ast}^{-1}\colon\mathfrak{O}(\R^{n},\topsimplex{n})\to\mathfrak{O}(\R^{m},\topsimplex{m})$. In addition, we obtain a presheaf $(\alpha_{\ast}^{-1})^{\ast}\tilde{S}^{\infty}_{m}(\mathcal{M})$ as
\begin{align*}
	(\alpha_{\ast}^{-1})^{\ast}\tilde{S}^{\infty}_{m}(\mathcal{M})(U)\coloneqq\tilde{S}^{\infty}_{m}(\mathcal{M})(\alpha_{\ast}^{-1}(U))
\end{align*}
and obtain a morphism $\alpha_{\ast}^{\ast}\colon\tilde{S}^{\infty}_{n}(\mathcal{M})\to(\alpha_{\ast}^{-1})^{\ast}\tilde{S}^{\infty}_{m}(\mathcal{M})$ as $\alpha_{\ast}^{\ast}(\gamma)\coloneqq\gamma\circ\alpha_{\ast}$. Since presheaves determine an inductive system, we obtain colimits.
\begin{center}
\begin{tikzpicture}[auto]
	\node (11) at (0, 1) {$tS^{\infty}_{n}(\mathcal{M})(U)$};
	\node (12) at (3, 1) {$\tilde{S}^{\infty}_{n}(\mathcal{M})(U)$};
	\node (13) at (7, 1) {$(\alpha_{\ast}^{-1})^{\ast}\tilde{S}^{\infty}_{m}(\mathcal{M})(U)$};
	\node (14) at (12, 1) {$\tilde{S}^{\infty}_{m}(\mathcal{M})(\alpha_{\ast}^{-1}(U))$};
	\node (21) at (0, 0) {$tS^{\infty}_{n}(\mathcal{M})$};
	\node (22) at (3, 0) {$\tilde{S}^{\infty}_{n}(\mathcal{M})$};
	\node (23) at (7, 0) {$\varinjlim_{\topsimplex{n}\subset U}(\alpha_{\ast}^{-1})^{\ast}\tilde{S}^{\infty}_{m}(\mathcal{M})(U)$};
	\node (24) at (12, 0) {$\tilde{S}^{\infty}_{m}(\mathcal{M})$};
	\path[draw, right hook->] (11) -- (12);
	\path[draw, ->] (12) --node {$\scriptstyle \alpha_{\ast}^{\ast}$} (13);
	\path[draw, transform canvas={yshift=1pt}] (13) -- (14);
	\path[draw, transform canvas={yshift=-1pt}] (13) -- (14);
	\path[draw, right hook->, dashed] (21) -- (22);
	\path[draw, dashed, ->] (22) -- (23);
	\path[draw, dashed, ->] (23) -- (24);
	\path[draw, ->] (11) -- (21);
	\path[draw, ->] (12) -- (22);
	\path[draw, ->] (13) -- (23);
	\path[draw, ->] (14) -- (24);
\end{tikzpicture}
\end{center}
They give a stratified simplicial set. We call the stratified simplicial set the \Def{$C^{\infty}$-singular stratified simplicial set} and denote it as $S^{\infty}(\mathcal{M})$. The homotopy category $\tau_{1}\tilde{S}^{\infty}(\mathcal{M})$ of the (underlying) simplicial set coincides with the fundamental groupoid $\pi_{1}(\mathcal{M})$. On the other hand, we can consider a presheaf $(\Omega_{\mathsf{sm}})_{n}\colon\mathfrak{O}(\R^{n},\topsimplex{n})\opposite\to\Set$ defined as 
\begin{align*}
	(\Omega_{\mathsf{sm}})_{n}(U)\coloneqq\Omega^{\bullet}(U)=\{\text{smooth differential forms on $U$}\}.
\end{align*}
It gives a simplicial set $\Omega_{\mathsf{sm}}$ in the same way as above.
\begin{center}
\begin{tikzpicture}[auto]
	\node (11) at (0, 1) {$(\Omega_{\mathsf{sm}})_{m}(U)$};
	\node (12) at (4, 1) {$(\alpha_{\ast}^{-1})^{\ast}(\Omega_{\mathsf{sm}})_{n}(U)$};
	\node (13) at (9, 1) {$(\Omega_{\mathsf{sm}})_{m}(\alpha_{\ast}^{-1}(U))$};
	\node (21) at (0, 0) {$(\Omega_{\mathsf{sm}})_{n}$};
	\node (22) at (4, 0) {$\varinjlim_{\topsimplex{n}\subset U}(\alpha_{\ast}^{-1})^{\ast}(\Omega_{\mathsf{sm}})_{m}(U)$};
	\node (23) at (9, 0) {$(\Omega_{\mathrm{sm}})_{m}$};
	\path[draw, ->] (11) --node {$\scriptstyle \alpha_{\ast}^{\ast}$} (12);
	\path[draw, transform canvas={yshift=1pt}] (12) -- (13);
	\path[draw, transform canvas={yshift=-1pt}] (12) -- (13);
	\path[draw, dashed, ->] (21) -- (22);
	\path[draw, dashed, ->] (22) -- (23);
	\path[draw, ->] (11) -- (21);
	\path[draw, ->] (12) -- (22);
	\path[draw, ->] (13) -- (23);
\end{tikzpicture}
\end{center}
Then any smooth differential form on $\mathcal{M}$ gives a simplicial map $\omega\colon\tilde{S}^{\infty}(\mathcal{M})\to\Omega_{\mathsf{sm}}$ as $\omega([\gamma])\coloneqq[\gamma^{\ast}\omega]$. Chen's iterated integral makes a pair of smooth differential forms on $\mathcal{M}$ $(\omega_{1},\dots,\omega_{r})$ corresponds to a differential form on the path space $C^{\infty}(\topsimplex{1},\mathcal{M})$, that is a family of differential forms $\{(\int\omega_{1}\dots\omega_{r})_{\alpha}\in\Omega(U)|\alpha\colon U\times\topsimplex{1}\to\mathcal{M}\text{: smooth}\}$. For each smooth map $\alpha\colon U\times\topsimplex{1}\to\mathcal{M}$, a differential form $(\int\omega_{1}\dots\omega_{r})_{\alpha}\in\Omega(U)$ is given as a fiberwise integration of a differential form $\phi_{\alpha}^{\ast}(\proj{1}^{\ast}\omega_{1}\wedge\dots\wedge\proj{r}^{\ast}\omega_{r})$ along the projection $\proj{U}\colon U\times\topsimplex{r}\to U$ where smooth map $\phi_{\alpha}\colon U\to\mathcal{M}^{r}$ is defined as $\phi_{\alpha}(x,t_{1},\dots,t_{r})\coloneqq(\alpha(x,t_{1}),\dots,\alpha(x,t_{r}))$.
\begin{center}
\begin{tikzpicture}[auto]
	\node (11) at (3, 1) {$C^{\infty}(\topsimplex{1},\mathcal{M})^{r}\times(\topsimplex{1})^{r}$};
	\node (12) at (7, 1) {$(\topsimplex{1}\times C^{\infty}(\topsimplex{1},\mathcal{M}))^{r}$};
	\node (20) at (0, 0) {$U\times\topsimplex{r}$};
	\node (21) at (3, 0) {$C^{\infty}(\topsimplex{1},\mathcal{M})\times\topsimplex{r}$};
	\node (22) at (7, 0) {$\mathcal{M}^{r}$};
	\node (23) at (10, 0) {$\mathcal{M}$};
	\path[draw, transform canvas={yshift=1pt}] (11) -- (12);
	\path[draw, transform canvas={yshift=-1pt}] (11) -- (12);
	\path[draw, ->] (21) --node {$\scriptstyle \text{diagonal}\times\iota_{r}$} (11);
	\path[draw, ->] (12) --node {$\scriptstyle (\mathrm{ev})^{r}$} (22);
	\path[draw, ->] (20) -- (0, -1) --node {$\scriptstyle \phi_{\alpha}$} (7, -1) -- (22);
	\path[draw, ->] (20) -- (21);
	\path[draw, ->] (21) -- (22);
	\path[draw, ->] (22) --node {$\scriptstyle \proj{i}$} (23);
\end{tikzpicture}
\end{center}
Let $V$ be a finite-dimensional $\R$-vector space and $\omega$ be a $\mathfrak{gl}(V)$-valued flat connection on $\mathcal{M}$. Then the holonomy $\mathsf{Hol}_{\omega}\colon\pi_{1}(\mathcal{M})\to\mathrm{GL}(V)$ is given by
\begin{align*}
	\gamma
		\mapsto
	\sum_{r=0}^{\infty}\int_{\topsimplex{1}}(\int\underbrace{\omega\cdots\omega}_{r})_{\gamma}
		=
	1+\int_{\topsimplex{1}}(\int\omega)_{\gamma}+\int_{\topsimplex{1}}(\int\omega\omega)_{\gamma}+\cdots.
\end{align*}
The simplicial holonomy is an analogy to classical holonomy in the above sense.

We construct an $A_{\infty}$-category $\mathcal{P}(X,\Z)$ and an $A_{\infty}$-functor $\Ahol{\nabla}\colon\mathcal{P}(X,\Z)\to\qG{\dq}{}{}{\mathfrak{g}}$. We can regard the path $A_{\infty}$-category $\mathcal{P}(X,\Z)$ as the linearization of a (stratified) simplicial set $X$ and we expect that (the analogy of) Chen's fundamental theorem and Hain's theorem \cite{MR727818} induced the $A_{\infty}$-functor $\Ahol{\nabla}\colon\mathcal{P}(X,\Z)\to\qG{\dq}{}{}{\mathfrak{g}}$. 

Chen's fundamental theorem (resp. Hain's theorem \cite{MR727818}) state existence of isomorpshim of $\R$-algebra (resp. Lie algebra over $\R$) using (ordinaly) de Rham complex, de Rham's theorem and real coefficient homology groups. Therefore it seems that it is impossible to obtain data on torsion (as Abelian group) using these theorems. On the other hand, we expect that it is possible to obtain data on torsion (as Abelian group) using the functor $\Ahol{\nabla}$.

\end{document}